\documentclass[11pt,a4paper]{article}
\usepackage[utf8]{inputenc}
\usepackage{a4}
\usepackage{amsmath,amsfonts,amssymb,amsthm, mathrsfs}
\usepackage{graphicx}
\usepackage{fancyhdr}
\usepackage{mathtools}
\usepackage{enumitem}
\usepackage{tikz}
\usepackage{geometry}
\usepackage{multirow}
\usepackage{authblk}
\usepackage[colorlinks=true,linkcolor=blue, citecolor=red]{hyperref}
\usepackage{cleveref}
\usepackage{appendix}
\usepackage{array} 
\usepackage{rotating} 
\usepackage{booktabs}
\usepackage{longtable}
\usepackage{pdflscape} 
\usepackage{afterpage} 
\usepackage{float}
\usepackage{etoolbox}
\usepackage{makecell}

\newtheorem{theorem}{Theorem}[section]
\newtheorem{coro}[theorem]{Corollary}
\newtheorem{lemma}[theorem]{Lemma}
\theoremstyle{definition}

\theoremstyle{remark}

\newcommand{\rht}{\operatorname{ht}}
\newcommand{\ksl}{\mathfrak{sl}}
\newcommand{\kso}{\mathfrak{so}}
\newcommand{\kg}{\mathfrak{g}}
\newcommand{\kh}{\mathfrak{h}}
\newcommand{\kn}{\mathfrak{n}}

\renewcommand\leq{\leqslant}
\renewcommand\geq{\geqslant}

\numberwithin{equation}{section}

\setlist[enumerate]{label=(\arabic*)}
\setlist[enumerate,2]{label=(\alph*)}
\setlist[enumerate,3]{label=(\roman*)}

\makeatletter
\patchcmd{\LT@makecaption}
{\sbox}
{\normalsize\sbox}
{}{}
\makeatother

\begin{document}
	
	\title{On singular vectors and associated Weyl elements for simply-laced universal affine vertex operator algebras}
	\author{Cuipo Jiang and Jingtian Song}
	\date{}
	\maketitle
	
	\begin{abstract} 
		Given  a finite-dimensional complex simple Lie algebra $\kg$  and a complex number $\kappa$, let $V^{\kappa}(\kg)$ be the associated universal affine vertex algebra.  Gorelik and Kac \cite{gorelik2007simplicity} gave a sufficient and necessary condition for $V^{\kappa}(\kg)$ to be simple.  In this paper, for simply-laced $\mathfrak{g}$ and non-critical $\kappa$, we determine the weights of singular vectors of $V^{\kappa}(\kg)$ with minimal conformal weights, when $V^{\kappa}(\kg)$ is not simple. We further determine all the longest Weyl elements in the  Kashiwara-Tanisaki character theorem \cite{KT00} which correspond to  the  weights of the  singular vectors.
	\end{abstract}
	
	\tableofcontents
	
	\section{Introduction}
	Let $\mathfrak{g}$ be a finite-dimensional complex simple Lie algebra  with the normalized non-degenerate bilinear form $(\cdot|\cdot)$. Let $\widehat{\mathfrak{g}}=\mathfrak{g}\otimes\mathbb{C}[t,t^{-1}]\oplus \mathbb{C}K$ be the associated affine Lie algebra \cite{kac1990infinite}.  Given a complex number ${\kappa}$,  let $V^{\kappa}(\mathfrak{g})$ be the universal affine vertex algebra associated to $\mathfrak{g}$ and $\kappa$ \cite{FZ92,LL04}. $V^{\kappa}(\mathfrak{g})$ is also called the vacuum module of the affine Lie algebra $\widehat{\mathfrak{g}}$ at level $\kappa$.  	Let $\mathbf{h}^\vee$ be the dual Coxeter number of $\mathfrak{g}$, as listed in \cite{kac1990infinite}. It was shown in  \cite{gorelik2007simplicity} that $V^\kappa(\kg)$ is non-simple if and only if
	$r^\vee(\kappa+\mathbf{h}^\vee)\in\mathbb{Q}_{\geq 0}\setminus\{\frac{1}{m}:m\in\mathbb{Z}_{\geq 1}\}$, where $r^\vee$ is the lacing number of $\kg$. One fundamental problem is to characterize the maximal ideal of $V^{\kappa}(\mathfrak{g})$. 
	If  $\kappa$ is a non-negative integer, the maximal ideal of $V^{\kappa}(\mathfrak{g})$ is generated by  the singular vector $e_{\theta}(-1)^{\kappa+1}{\mathbf{1}}$ \cite{KW89, kac1990infinite, FZ92,  LL04}, etc.   When $\kappa$ is an admissible number, that is, 
	\begin{align*}
		\kappa+\mathbf{h}^{\vee}=\frac{p}{q},\quad p,q\in \mathbb{Z}_{\geq 1}, \ (p,q)=1,\
		p\geq \begin{cases}\mathbf{h}^{\vee}&\text{if }(r^{\vee},q)=1,\\
			\mathbf{h}&\text{if }(r^{\vee},q)\ne 1,
		\end{cases}
	\end{align*}
	the maximal ideal of $V^{\kappa}(\mathfrak{g})$ is generated by a singular vector, whose   weight 
	with respect to the generalized Cartan subalgebra of $\widehat{\mathfrak{g}}$ can be determined by the character formula given in \cite{KW89}, although it is  usually hard to characterize the singular vector explicitly in general, except for some special cases  (see \cite{MFF86, APV23}, etc.). 
	For other cases of $\kappa=-\mathbf{h}^{\vee}+\frac{p}{q}$,  which we usually call  the non-admissible levels,  no systematic results have been established with respect to  determination of maximal ideals, 
	although there has been nice progress in  characterizing singular vectors of  $V^{\kappa}(\mathfrak{g})$ in general \cite{AJM21} and maximal ideals for some special cases  of $\mathfrak{g}$ and $\kappa$ \cite{AP08, Pe08, AM17, AKMPP20, ADFLM25, JS25}.
	
	The first  purpose of this paper is to use the Gorelik-Kac determinant formula given in \cite{gorelik2007simplicity} to characterize  weights of singular vectors of $V^{\kappa}(\mathfrak{g})$ when $\mathfrak{g}$ is simply-laced.  We determine the weights of singular vectors with minimal
	conformal weights at all non-critical levels $\kappa$ for which
	$V^\kappa(\mathfrak g)$ is not simple.
	
	Let $\mathfrak{h}$ and $\widehat{\mathfrak{h}}$ be the Cartan subalgebras of $\mathfrak{g}$ and $\widehat{\mathfrak{g}}$, respectively.  Let 
	$r^\vee(\kappa+\mathbf{h}^\vee)=\frac{p}{q}$ be such that $p\in\mathbb{Z}_{\geq 2}$, $q\in\mathbb{Z}_{\geq 1}$, and $(p,q)=1$.  Then by \cite{gorelik2007simplicity}, the universal affine vertex operator algebra  $V^{\kappa}(\mathfrak{g})$ is non-simple.  For simplicity, we call a singular vector  minimal if its conformal weight is the smallest one among all the singular vectors. To characterize the maximal ideal of $V^{\kappa}(\mathfrak{g})$, the first thing  is to determine the weights  with respect to the Cartan subalgebra $\widehat{\mathfrak{h}}$  of the minimal singular vectors.  To achieve this  goal  we can use the powerful Gorelik-Kac determinant formula. 
	However,  the determinant formula is rather complicated.  Therefore, our first step is to rewrite the formula in a form that is easier to handle. We next analyze the domain of the parameters to determine the weights of minimal singular vectors.  
	The Weyl group plays an  important role in the analysis. The following three theorems characterize weights of minimal singular vectors  for $\mathfrak{g}$ of types $\mathsf{A}$, $\mathsf{D}$,  and $\mathsf{E}$, respectively.
	\begin{theorem}\label{JS25b}
		Let  $\mathfrak{g}=\mathfrak{sl}_n$ and $\kappa+n=\frac{p}{q}$ with $p\in\mathbb{Z}_{\geq 2},q\in\mathbb{Z}_{\geq 1}$ and $(p,q)=1$.  
		Then all the weights $\Lambda_{sing}=\kappa\Lambda_0-\nu_\kappa$ of minimal singular vectors in $V^{\kappa}(\mathfrak{g})$ are given as follows, and the number of linearly independent singular vectors with respect to each weight is exactly one: 
		\begin{enumerate}[wide] 
			\item When $p\geq n$, $\nu_{\kappa}=(p-n+1)q\delta-(p-n+1)(\alpha_1+\alpha_2+\cdots+\alpha_{n-1})$.
			\item When $p=2<n$,  if $2|n$, then 
			$$\nu_{\kappa}=
			\dfrac{n}{2}q\delta-(\alpha_1+\alpha_2+\cdots+\alpha_{n-1});$$ 
			if $2\nmid n$, then $\nu_{\kappa}$ can be taken as $$nq\delta-(2\alpha_1+2\alpha_2+\cdots+2\alpha_{n-2}+\alpha_{n-1})$$
			and $$nq\delta-(\alpha_1+2\alpha_2+\cdots+2\alpha_{n-2}+2\alpha_{n-1}).$$
			\item When $n=5$ and $p=3$, $\nu_{\kappa}$ can be taken as $4q\delta-2(\alpha_1+\alpha_2+\alpha_3+\alpha_4)$, $4q\delta-(2\alpha_1+3\alpha_2+2\alpha_3+\alpha_4)$ and $4q\delta-(\alpha_1+2\alpha_2+3\alpha_3+2\alpha_4)$.
			\item When $n=7$ and  $p=4$, $\nu_{\kappa}=4q\delta-2(\alpha_1+\alpha_2+\cdots+\alpha_6)$.
			\item When $n=8$ and $p=3$,  $\nu_{\kappa}=6q\delta-2(\alpha_1+\alpha_2+\cdots+\alpha_7)$.
			\item For other cases, set $s_1=\lfloor n/p\rfloor, s_2=\lceil n/p\rceil$ and define the function $D(s)=(|sp-n|+1)s$. Set $D_p=\min\{D(s_1),D(s_2)\}$. Then 
			$$\nu_{\kappa}=D_pq\delta-\lambda_i$$ 
			for $i$ with $D(s_i)=D_p$.  Here
			\begin{align*}
				\lambda_1
				=&\sum_{j=1}^{n-1} \min\{j,r_1,n-j \} \alpha_j\\
				=&\alpha_1+2\alpha_2+3\alpha_3+\cdots+r_1\alpha_{r_1}+r_1\alpha_{r_1+1}+\cdots+r_1\alpha_{n-(r_1+1)}+\\
				&r_1\alpha_{n-r_1}+(r_1-1)\alpha_{n-(r_1-1)}+(r_1-2)\alpha_{n-(r_1-2)}+\cdots+\alpha_{n-1}
			\end{align*}
			and 
			$$\lambda_2=r_2(\alpha_1+\alpha_2+\cdots+\alpha_{n-1}),$$
			where $r_i=|s_ip-n|+1$. 
			When $D(s_1)=D(s_2)$, up to scalar, if $p|n$, then $\lambda_1=\lambda_2$ and the singular vector is the same one;
			if $p\nmid n$, then both singular vectors appear.	
		\end{enumerate}
		
	\end{theorem}
	
	\begin{theorem}\label{main2}
		Let $\kg$ be the simple Lie algebra of type $\mathsf{D}_{n}(n\geq 4)$. 
		Assume that $\kappa+\mathbf{h}^\vee=\dfrac{p}{q}$ with $p\in \mathbb{Z}_{\geq 2}$, $q\in \mathbb{Z}_{\geq 1}$ and $(p,q)=1$. 
		Then all the weights $\Lambda_{sing}=\kappa\Lambda_0-\nu_\kappa$ of minimal singular vectors in $V^{\kappa}(\mathfrak{g})$ are given as follows: 
		\begin{enumerate}[wide] 
			\item When $p\geq 2n-2$, 
			$\nu_{\kappa}=(p-2n+3)q\delta-(p-2n+3)(\epsilon_1+\epsilon_2).$
			\item When $p=3, 3|n-1$,  $\nu_{\kappa}= (2n-1)q\delta-(3\epsilon_1+2\epsilon_2+\epsilon_3).$
			\item When $p=5,n=7$,  $\nu_{\kappa}= 11q\delta-(3\epsilon_1+2\epsilon_2+2\epsilon_3+2\epsilon_4+\epsilon_5).$
			\item When $p=5,n=12$,  $\nu_{\kappa}= 20q\delta-4(\epsilon_1+\epsilon_2).$
			\item When $p=4,n=4$,  $\nu_{\kappa}$ can be taken as
			$2q\delta-2\epsilon_1,~2q\delta-(\epsilon_1+\epsilon_2+\epsilon_3+\epsilon_4)$ and $2q\delta-(\epsilon_1+\epsilon_2+\epsilon_3-\epsilon_4)$.
			\item When $p=5,n=4$,  $\nu_{\kappa}$ can be taken as
			$4q\delta-4\epsilon_1$, $4q\delta-2(\epsilon_1+\epsilon_2+\epsilon_3+\epsilon_4)$ and $4q\delta-2(\epsilon_1+\epsilon_2+\epsilon_3-\epsilon_4)$.
			\item For other cases, set 
			$s_1=\lfloor (2n-1)/p\rfloor,~ s_2=\lceil (2n-1)/p\rceil$. 
			Set $D= \min\{D_{(0)},D_{(1)},D_{(2)}\}$ where
			$$D_{(0)}=
			\begin{cases}
				n-\frac{p}{2},& 2|p,\\
				+\infty,& 2\nmid p,
			\end{cases}$$
			$$D_{(1)}=
			\begin{cases}
				s_1(n-s_1p/2),& 2\nmid s_1(p-1),\\
				s_1(2n-s_1p+1),& 2|s_1(p-1),
			\end{cases}$$
			$$D_{(2)}=
			\begin{cases}
				s_2(s_2p-2n+3),& 2\nmid s_2,\\
				s_2(s_2p/2-n+1),& 2| s_2.
			\end{cases}$$
			Then $\nu_{\kappa}= qD\delta-\lambda_i$ for $i$ with $D=D_{(i)}$, where
			$$\lambda_0=\begin{cases}
				\epsilon_1+\epsilon_2+\cdots+\epsilon_{p},
				&p<n,\\
				\epsilon_1+\epsilon_2+\cdots+\epsilon_{2n-p},
				&p>n,
			\end{cases}
			$$
			$$\lambda_1=\begin{cases}
				\epsilon_1+\epsilon_2+\cdots+\epsilon_{2n-s_1p},
				&2\nmid s_1(p-1),\\
				2(\epsilon_1+\epsilon_2+\cdots+\epsilon_{2n-s_1p+1}),
				&2| s_1(p-1),
			\end{cases}$$
			$$\lambda_2=\begin{cases}
				(s_2p-2n+3)(\epsilon_1+\epsilon_2),& 2\nmid s_2,\\
				(s_2p-2n+2)\epsilon_1,& 2| s_2.
			\end{cases}$$
		\end{enumerate}
		Furthermore, if $p=2$ and $2|n$, there are two linearly independent minimal singular vectors with respect to weight $\kappa\Lambda_0-(n-1)q\delta+\epsilon_1+\epsilon_2$. 
		For other cases, the number of linearly independent minimal singular vectors  with respect to each weight is exactly one. 
	\end{theorem}
		\begin{theorem}\label{singE}
		Let $\kg$ be the Lie algebra of type $\mathsf{E}_\ell$, $\ell=6,7,8$. 
		Assume that $\kappa+\mathbf{h}^\vee=\dfrac{p}{q}$ with $p\in \mathbb{Z}_{\geq 2}$, $q\in \mathbb{Z}_{\geq 1}$ and $(p,q)=1$.
		Then all the weights $\Lambda_{sing}^{(\kappa)}=\kappa\Lambda_0-D_pq\delta+\lambda_{sing}$ of minimal singular vectors in $V^{\kappa}(\mathfrak{g})$ are given as follows, and  the number of linearly independent singular vectors  with respect to each weight is exactly one: 
		\begin{enumerate}[wide] 
			\item When $p\geq \mathbf{h}^\vee$, we have $D_p=p-\mathbf{h}^\vee+1$ and $\lambda_{sing}=D_p\theta$;
			\item When $p<\mathbf{h}^\vee$, values of $D_p$ and $\lambda_{sing}$ are given in \Cref{tab:Esing}, where $(c_1,c_2,\dots,c_\ell)$ denotes the weight $c_1\alpha_1+c_2\alpha_2+\cdots+c_\ell\alpha_\ell$.
		\end{enumerate}
	\end{theorem}
	
	\begin{table}[H]          
		\centering
		\tiny           
		\caption{$D_p$ and $\lambda_{sing}$ for type $\mathsf{E}$ ($p < \mathbf{h}^\vee$)}
		\label{tab:Esing}
		\begin{tabular}{c||cc|cc|cc}
			\toprule
			$p$ & \multicolumn{2}{c|}{$\mathsf{E}_6$} & 
			\multicolumn{2}{c|}{$\mathsf{E}_7$} & 
			\multicolumn{2}{c}{$\mathsf{E}_8$} \\
			& $D_p$ & $\lambda_{{sing}}$ & 
			$D_p$ & $\lambda_{{sing}}$ & 
			$D_p$ & $\lambda_{{sing}}$ \\
			\midrule
			2  & 12 & (2, 3, 4, 6, 4, 2)                & 9  & (2, 2, 3, 4, 3, 2, 1)           & 30 & (4, 6, 8, 12, 10, 8, 6, 3) \\
			3  & 4  & (1, 2, 2, 3, 2, 1)                & 14 & (3, 5, 6, 9, 7, 5, 3)           & 31 & (6, 9, 12, 18, 15, 12, 9, 5) \\
			4  & 12 & (4, 6, 7, 10, 7, 4)               & 15 & (6, 6, 9, 12, 9, 6, 3)          & 12 & (4, 5, 7, 10, 8, 6, 4, 2) \\
			5  & 6  & (4, 3, 5, 6, 4, 2), (2, 3, 4, 6, 5, 4) & 14 & (6, 7, 10, 14, 11, 8, 4)  & 18 & (6, 9, 12, 18, 15, 12, 8, 4) \\
			6  & 3  & (2, 2, 3, 4, 3, 2)                & 5  & (2, 3, 4, 6, 5, 4, 3)           & 22 & (8, 12, 16, 24, 20, 16, 12, 8) \\
			7  & 8  & (5, 6, 8, 11, 8, 5)               & 4  & (2, 3, 4, 6, 5, 4, 2)           & 15 & (7, 11, 14, 21, 17, 13, 9, 5) \\
			8  & 3  & (2, 3, 4, 6, 4, 2)                & 7  & (4, 7, 8, 12, 9, 6, 3)          & 6  & (4, 5, 7, 10, 8, 6, 4, 2) \\
			9  & 2  & (2, 2, 3, 4, 3, 2)                & 6  & (4, 6, 8, 12, 9, 6, 3)          & 19 & (11, 17, 22, 33, 27, 21, 14, 7) \\
			10 & 4  & (4, 4, 6, 8, 6, 4)                & 3  & (2, 3, 4, 6, 5, 4, 3)           & 10 & (8, 10, 14, 20, 16, 12, 8, 4) \\
			11 & 6  & (6, 6, 9, 12, 9, 6)               & 6  & (4, 6, 8, 12, 10, 8, 6)         & 24 & (18, 24, 33, 48, 39, 30, 20, 10) \\
			12 & -- & --                                   & 3  & (3, 4, 6, 8, 6, 4, 2)           & 5  & (4, 6, 8, 12, 10, 8, 6, 3) \\
			13 & -- & --                                   & 6  & (6, 8, 12, 16, 12, 8, 4)        & 14 & (12, 17, 23, 34, 28, 22, 16, 8) \\
			14 & -- & --                                   & 2  & (2, 3, 4, 6, 5, 4, 2)           & 7  & (7, 10, 14, 20, 16, 12, 8, 4) \\
			15 & -- & --                                   & 4  & (4, 6, 8, 12, 10, 8, 4)         & 6  & (6, 9, 12, 18, 15, 12, 8, 4) \\
			16 & -- & --                                   & 6  & (6, 9, 12, 18, 15, 12, 6)       & 12 & (12, 18, 24, 36, 30, 24, 16, 8) \\
			17 & -- & --                                   & 8  & (8, 12, 16, 24, 20, 16, 8)      & 18 & (18, 27, 36, 54, 45, 36, 24, 12) \\
			18 & -- & --                                   & -- & --                                 & 4  & (5, 8, 10, 15, 12, 9, 6, 3) \\
			19 & -- & --                                   & -- & --                                 & 8  & (10, 16, 20, 30, 24, 18, 12, 6) \\
			20 & -- & --                                   & -- & --                                 & 3  & (4, 6, 8, 12, 10, 8, 6, 3) \\
			21 & -- & --                                   & -- & --                                 & 6  & (8, 12, 16, 24, 20, 16, 12, 6) \\
			22 & -- & --                                   & -- & --                                 & 9  & (12, 18, 24, 36, 30, 24, 18, 9) \\
			23 & -- & --                                   & -- & --                                 & 12 & (16, 24, 32, 48, 40, 32, 24, 12) \\
			24 & -- & --                                   & -- & --                                 & 2  & (4, 5, 7, 10, 8, 6, 4, 2) \\
			25 & -- & --                                   & -- & --                                 & 4  & (8, 10, 14, 20, 16, 12, 8, 4) \\
			26 & -- & --                                   & -- & --                                 & 6  & (12, 15, 21, 30, 24, 18, 12, 6) \\
			27 & -- & --                                   & -- & --                                 & 8  & (16, 20, 28, 40, 32, 24, 16, 8) \\
			28 & -- & --                                   & -- & --                                 & 10 & (20, 25, 35, 50, 40, 30, 20, 10) \\
			29 & -- & --                                   & -- & --                                 & 12 & (24, 30, 42, 60, 48, 36, 24, 12) \\
			\bottomrule
		\end{tabular}
	\end{table}

	Let $W$ be the Weyl group of $\mathfrak{g}$, and 
	$\widehat{W}$  the affine Weyl group of $\widehat{\mathfrak{g}}$. For $\lambda\in\widehat{\mathfrak{h}}^*$, denote $\Delta(\lambda)=\{\alpha\in\widehat{\Delta}^{re}:(\lambda+\widehat{\rho})(\alpha^{\vee})\in{\mathbb Z}\}$.  Let $\widehat{W}(\lambda)$ be the associated Weyl group generated by reflections $s_{\alpha}$, $\alpha\in \Delta(\lambda)$, and $\widehat{W}_0(\lambda)$ the subgroup of 
	$\widehat{W}(\lambda)$ consisting of stabilizers of $\lambda$.  By the Kashiwara-Tanisaki character theorem \cite{KT00}, for given 
	$\kappa\Lambda_0$, there exist the unique corresponding dot-dominant $\Lambda\in{\widehat{\mathfrak h}}^*$ and longest  elements $\tilde{y}$ in $\tilde{y}\widehat{W}_0(\Lambda)$ and $\tilde{z}=\tilde{z}_0\tilde{y}$ in  $\tilde{z}\widehat{W}_0(\Lambda)$, respectively, 
	such that  $$\kappa\Lambda_0=\tilde{y}\circ\Lambda,$$
	and 
	$$\Lambda_{sing}=\tilde{z}\circ \Lambda=\tilde{z}_0\tilde{y}\circ \Lambda.$$
	The second goal of this paper is to characterize $\tilde{y}$ and $\tilde{z}_0$ explicitly when  $\mathfrak{g}$ is simply-laced and $\kappa+\mathbf{h}^\vee=\dfrac{p}{q}$.  
	We first give criteria to  determine $\tilde{y}$ and $\tilde{z}$ by \Cref{criteriononlongesty,criteriononlongestz}.  Based on these criteria we finally  give the characterization of  $\tilde{y}$ and $\tilde{z}_0$ in terms of products of elements in $W$ and affine translations  $t_{q\gamma}$, for some $\gamma$ in  the root lattice of $\mathfrak{g}$. 
	Specifically, \Cref{Alongesty,Alongestz}   give the explicit characterization of $\tilde{y}$ and $\tilde{z}_0$, respectively, when $\mathfrak{g}$ is of type  $\mathsf{A}$.   \Cref{Dlongesty,Dlongestz} are results on $\tilde{y}$ and $\tilde{z}_0$ when $\mathfrak{g}$ is of type  $\mathsf{D}$. We give the characterization of $\tilde{y}$ and $\tilde{z}_0$ when $\mathfrak{g}$ is of type  $\mathsf{E}$ by \Cref{Elongesty,Elongestz}.

	\vskip 0.2cm
	The rest of this paper is organized as follows. In Section 2, we recall the universal affine vertex algebras, Shapovalov form  and vacuum determinant. In Section 3, we analyze the determinant formula and then give an equivalent condition for determining weights of minimal singular vectors.  Section 4 is dedicated to proving \Cref{JS25b}.  The proof of \Cref{main2} is given in Section 5. We also deal with the case  that $\mathfrak{g}$ is of $\mathsf{E}$-type in this section. But we
	omit the proof of \Cref{singE}.  Section  6 is dedicated to giving criteria of longest elements in left cosets of $\widehat{W}_0(\Lambda^{(\kappa)})$. Finally in Section 7, we give characterization of $\tilde{y}$ and $\tilde{z}_0$ for simply-laced $\mathfrak{g}$ by \Cref{Alongesty,Alongestz},   \Cref{Dlongesty,Dlongestz}, and \Cref{Elongesty,Elongestz}.

	\section{Preliminaries}
	\subsection{Affine Lie algebras}
	We introduce affine Lie algebras in this subsection following \cite{kac1990infinite}.
	Let $\kg$ be a finite-dimensional simple Lie algebra over $\mathbb{C}$, and $\kh$ its Cartan subalgebra.
	Let $\Delta$ be the root system, and $\Delta^+$ the set of positive roots.
	Denote the root lattice by $Q$, and its positive part by $Q^+$. 
	There is a non-degenerate invariant bilinear form:
	$$(\cdot|\cdot)=\frac{1}{2\mathbf{h}^\vee}\times \text{ Killing form of }\kg.$$  
	The triangular decomposition is $\kg=\kn_-\oplus\kh\oplus\kn_+$. For root $\alpha\in\Delta$, we denote its root space by $\kg_\alpha$, the height of root by $\rht(\alpha)$, and the Chevalley basis by $e_\alpha$. 
	The Weyl group $W$ is generated by all the reflections $s_{\alpha}(\alpha\in\Delta^+)$. Denote the length of $w$ by $\ell(w)$. Recall that the shifted action of $W$ on $\kh^*$ is  $w\circ\lambda=w(\lambda+\rho)-\rho$, where $\rho$ is the half sum of all positive roots. 
	
	The associated affine Lie algebra $\widehat{\kg}=\kg\otimes\mathbb{C}[t,t^{-1}]\oplus \mathbb{C}K$ has the relations
	\begin{align*}
		[a\otimes t^m,b\otimes t^n]=[a,b]\otimes t^{m+n}+m\delta_{m,-n}(a|b)K,~
		[a\otimes t^m,K]=[K,K]=0.
	\end{align*}
	where $a,b\in\kg$ and $m,n\in\mathbb{Z}$. 
	We shall write $a(m)$ for $a\otimes t^m$. Then $\widehat{\kg}$ has the following triangular decomposition:
	$$
	\widehat{\kg}=\widehat{\kg}_+\oplus\widehat{\kg}_0\oplus\widehat{\kg}_-,
	$$
	where 
	$$
	\widehat{\kg}_+=\kg\otimes t\mathbb{C}[t],  \ \ \widehat{\kg}_-=\kg\otimes t^{-1}\mathbb{C}[t^{-1}], \ \ \widehat{\kg}_0=\kg+{\mathbb C}K.
	$$
	Also, we have the corresponding root system $\widehat{\Delta}$, the root lattice $\widehat{Q}$.
	Let $\delta$ be the positive imaginary root such that any imaginary root is an integral multiple of it.
	Then the root system is given by
	$$\widehat{\Delta}=\{m\delta+\alpha:m\in\mathbb{Z},\alpha\in \Delta\cup\{0\}\}\setminus\{0\},$$
	and its positive part is
	$$\widehat{\Delta}^+=\{m\delta+\alpha:m\in\mathbb{Z}_{>0},\alpha\in \Delta\cup\{0\}\}\cup\Delta^+.$$
	Denote the set of real roots by $\widehat{\Delta}^{re}$. Set $\alpha^\vee=2\alpha/(\alpha|\alpha)$ for $\alpha\in \widehat{\Delta}^{re}$. 
	
	The affine Weyl group $\widehat{W}$ is generated by all reflections of positive real roots. 
	For $\gamma\in \kh^*$, define $t_{\gamma}\in\operatorname{End}\widehat{\kh}^*$ by
	$$
	t_{\gamma}(\lambda)=\lambda+(\lambda|\delta)\gamma-((\lambda|\gamma)+\frac{1}{2}(\lambda|\delta)(\gamma|\gamma))\delta, 
	$$
	where $\lambda\in\widehat{\kh}^*$. When $\kg$ is simply-laced, we have 
	$$
	\widehat{W}=T\rtimes W=W\ltimes T,
	$$
	where $T=\{t_\gamma:\gamma\in Q\}$.
	So any element in $\widehat{W}$ can be uniquely written as the product $t_\gamma w$, where $\gamma\in Q$, $w\in W$. 
	
	\subsection{Universal affine vertex algebras}
	
	Let $\mathbb{C}_\kappa (\kappa\in \mathbb{C})$ be a 1-dimensional $\kg+\widehat{\kg}_++\mathbb{C}K$-module, where $\kg+\widehat{\kg}_+$ acts trivially on $\mathbb{C}_\kappa$ and $K$ as scalar $\kappa$. Define the induced $\widehat{\kg}$-module
	$$V^\kappa(\kg):=U(\widehat{\kg})\otimes_{U(\kg+ \widehat{\kg}_++ \mathbb{C}K)}\mathbb{C}_\kappa,$$
	which is called the vacuum module.  It is known that  $V^\kappa(\kg)$ carries a structure of vertex algebra, called the universal affine vertex algebra associated to $\mathfrak{g}$ at level $\kappa$
	\cite{FZ92,LL04}.
	
	By the PBW theorem, $V^\kappa(\kg)$ has a natural
	$\mathbb{Z}_{\geq0}$-grading
	\[
	V^\kappa(\kg)
	=
	\bigoplus_{c\in\mathbb{Z}_{\geq0}}
	V^\kappa(\kg)_{[c]},
	\]
	where
	\[
	V^\kappa(\kg)_{[c]}
	=
	\operatorname{span}_{\mathbb C}
	\left\{
	x_1(-n_1)\cdots x_r(-n_r)\mathbf{1}:
	n_1+\cdots+n_r=c
	\right\}.
	\]
	Here $x_1,\dots,x_r\in\kg$ and
	$n_1,\dots,n_r\in\mathbb{Z}_{\geq1}$.
	Each $V^\kappa(\kg)_{[c]}$ is a finite-dimensional
	$\kg$-module.
	
	Assume that $\kappa\neq-\mathbf{h}^{\vee}$. Let
	$\{x_i\}_{i=1}^{\dim\kg}$ and
	$\{x^i\}_{i=1}^{\dim\kg}$ be dual bases of $\kg$ with respect
	to $(\cdot|\cdot)$. Then
	\[
	\omega_\kappa
	=
	\frac{1}{2(\kappa+\mathbf{h}^{\vee})}
	\sum_{i=1}^{\dim\kg}
	x_i(-1)x^i(-1)\mathbf{1}
	\]
	is the Sugawara conformal vector of $V^\kappa(\kg)$
	\cite{kac1990infinite,FZ92,LL04}. Write
	\[
	Y(\omega_\kappa,z)
	=
	\sum_{n\in\mathbb Z}L(n)z^{-n-2}.
	\]
	The operator $L(0)$ agrees with the above PBW grading operator:
	\[
	L(0)v=cv
	\qquad
	\text{for }v\in V^\kappa(\kg)_{[c]}.
	\]
	In particular, $c$ is called the conformal weight of $v$.
	The central charge of $V^\kappa(\kg)$ is
	$
	\frac{\kappa\dim\kg}
	{\kappa+\mathbf{h}^{\vee}}.
	$
	Let
	\[
	\Omega=\sum_{i=1}^{\dim\kg}x_ix^i
	\in Z(U(\kg))
	\]
	be the quadratic Casimir element, where elements of $\kg$ are
	identified with their zero modes. The Sugawara construction
	gives
	\begin{equation}\label{eq:Sugawara-L0}
		2(\kappa+\mathbf{h}^{\vee})L(0)
		=
		\Omega+
		2\sum_{m=1}^{\infty}
		\sum_{i=1}^{\dim\kg}
		x_i(-m)x^i(m).
	\end{equation}

	Denote the (unique) maximal submodule of $V^\kappa(\kg)$ by $N_\kappa(\kg)$, and the simple quotient module $V^\kappa(\kg)/N_\kappa(\kg)$ by $L_\kappa(\kg)$, which is also called the simple affine vertex algebra associated to $\mathfrak{g}$ at level $\kappa$. 
	
	Let $\Lambda_0 \in \widehat{\kh}^*$ be the basic fundamental weight of $\widehat{\mathfrak{g}}$ with respect to $\widehat{\mathfrak{h}}$, that is,  $\Lambda_0(h) = 0$ for $h \in \kh$ and  $\Lambda_0(K) = 1$. Obviously, the highest weight of $V^\kappa(\kg)$ is $\kappa\Lambda_0$. 
	Denote the weight space of $V^\kappa(\kg)$ (resp. $L_\kappa(\kg)$) with weight $\kappa\Lambda_0-\nu$ by $V^\kappa(\kg)_\nu$  (resp. $L_\kappa(\kg)_\nu$).
	We	recall the following result from \cite{gorelik2007simplicity}.
	\begin{theorem}[\cite{gorelik2007simplicity}]\label{nonsimple}
		The vacuum module $V^\kappa(\kg)$ is non-simple if and only if
		$$r^{\vee}(\kappa+\mathbf{h}^\vee)\in\mathbb{Q}_{\geq 0}\setminus\left\{\dfrac{1}{m}:m\in\mathbb{Z}_{\geq 1}\right\}.$$
	\end{theorem}
	Recall that a  non-zero vector $v\in V^\kappa(\kg)$ is called singular if $v\in N_\kappa(\kg)$ and $(\kn_++\widehat{\kg}_+).v=0$.
	
	\subsection{Shapovalov form and the vacuum determinant}
	
	As shown in \cite{shapovalov1972bilinear}, there exists a unique bilinear form $S(\kappa)(\cdot,\cdot)$ on the vacuum module $V^\kappa(\kg)$ satisfying the following conditions:
	\begin{enumerate}
		\item $S(\kappa)(\mathbf{1},\mathbf{1})=1,$
		\item $S(\kappa)(u(n)v_1,v_2)=S(\kappa)(v_1,\sigma(u)(-n)v_2)\quad(u\in{\kg},v_1,v_2\in V^\kappa(\kg))$,
	\end{enumerate}
	where $\mathbf{1}$ stands for the vacuum of $V^\kappa(\kg)$ and $\sigma$ is the standard anti-involution on ${\kg}$, that is, $\sigma(h)=h$ for $h\in{\kh}$ and $\sigma(e_\alpha)=e_{-\alpha}$ for $\alpha\in{\Delta}$.
	This form is called the Shapovalov form  on $V^\kappa(\kg)$. 
	Moreover, the maximal submodule $N_\kappa(\kg)$ coincides with the radical of the Shapovalov form.

	Let $S_\nu(\kappa)$ be the restriction of $S(\kappa)$ to $V^\kappa(\kg)_{\nu}$. 
	The vacuum determinant is given in \cite{gorelik2007simplicity} by
	$$\det S_\nu(\kappa)=\prod_{r=1}^\infty \prod_{\gamma\in \widehat{\Delta}^+\setminus\Delta}\phi_{r,\gamma}(\kappa)^{d_{r,\gamma}(\nu)\dim \widehat{\kg}_\gamma}$$
	where 
	$$\phi_{r,\gamma}(\kappa)=(\Lambda_0|\gamma)\kappa+(\widehat{\rho}|\gamma)-r(\gamma|\gamma)/2,$$
		$$\widehat{R}=\prod_{\alpha\in\widehat{\Delta}^+}(1-e^{-\alpha})^{\dim \widehat{\kg}_\alpha},$$
	and
	$$\sum_\nu d_{r,\gamma}(\nu)e^{-\nu}=\widehat{R}^{-1}\sum_{w\in W}(-1)^{\ell(w)}e^{w\circ(-r\gamma)}.$$

	Notice that $\det S_\nu(\kappa)$ is a polynomial in one variable $\kappa$. 
	Up to a nonzero constant factor, we can write the vacuum determinant as a product of linear functions:
	$$\det S_\nu(\kappa)=\prod_{b\in\mathbb{C}}(\kappa+\mathbf{h}^\vee-b)^{m_b(\nu)},$$
	where $m_b(\nu)$ is the multiplicity of $(\kappa+\mathbf{h}^\vee-b)$ in the polynomial $\det S_\nu$. 
	It follows that, for a given $\kappa$, the weight space $N_{\kappa}(\kg)_\nu\neq 0$ if and only if $m_{\kappa+\mathbf{h}^\vee}(\nu)\neq 0$.

	\subsection{Jantzen filtration and sum formula}
	
	Let $T$ be an indeterminate. Define the following $\widehat{\kg}-\mathbb{C}[T]$ bimodule:
	\begin{equation}
		M_\kappa(\kg,T):=U(\widehat{\kg})\otimes_{U(\kg+\widehat{\kg}_++\mathbb{C}K)}\mathbb{C}[T],
	\end{equation}
	where $\mathbb{C}[T]$ is a $\kg+\widehat{\kg}_++\mathbb{C}K$-module, where $\kg+\widehat{\kg}_+$ acts trivially, and $K$ acts by the multiplication $\kappa+T$.
	In particular, when $T=0$ we have $M_\kappa(\kg,0)=V^\kappa(\kg)$. 
	
	The PBW grading on $V^\kappa(\kg)$ extends naturally to
	$M_\kappa(\kg,T)$:
	\[
	M_\kappa(\kg,T)
	=
	\bigoplus_{c\in\mathbb{Z}_{\geq0}}
	(M_\kappa(\kg,T))_{[c]}.
	\]
	For every $c$, the subspace
	$(M_\kappa(\kg,T))_{[c]}$ is a free
	$\mathbb{C}[T]$-module of finite rank.
	
	When $\kappa\neq-\mathbf{h}^{\vee}$, the element
	$\kappa+T+\mathbf{h}^{\vee}$ is invertible in
	$\mathbb{C}[T]_{(T)}$. Hence, after localizing at $(T)$, the
	Sugawara construction applies to $M_\kappa(\kg,T)$ with level
	$\kappa+T$. Its zero mode $L(0)$ agrees with the PBW grading,
	and \eqref{eq:Sugawara-L0} becomes
	\begin{equation}\label{eq:Sugawara-L0-T}
		2(\kappa+T+\mathbf{h}^{\vee})L(0)
		=
		\Omega+
		2\sum_{m=1}^{\infty}
		\sum_{i=1}^{\dim\kg}
		x_i(-m)x^i(m).
	\end{equation}
	
	As shown in \cite{gorelik2007simplicity,jantzen1977kontravariante}, there exists a unique $\mathbb{C}[T]$-bilinear form $S(\kappa,T): M_\kappa(\kg,T)\otimes M_\kappa(\kg,T)\to\mathbb{C}[T]$ on
	$M_\kappa(\kg,T)$ satisfying analogous properties of Shapovalov form.
	For $r\in\mathbb{Z}_{\geq 0}$, define 
	\begin{equation}
		M_\kappa^r(\kg,T):=
		\{v\in M_\kappa(\kg,T):S(\kappa,T)(v,v')\in T^r\mathbb{C}[T],~\forall v'\in M_\kappa(\kg,T)\}.
	\end{equation}
	This gives a decreasing filtration of $M_\kappa(\kg,T)$ as a bimodule.
	
	Specializing this filtration at $T=0$, we obtain the Jantzen filtration $\mathcal{F}^r(V^\kappa(\kg))$ on $V^\kappa(\kg)$. 
	It is known that $\mathcal{F}^0(V^\kappa(\kg))=V^\kappa(\kg)$, $\mathcal{F}^1(V^\kappa(\kg))=N_\kappa(\kg)$, and $\bigcap_{r=1}^\infty \mathcal{F}^r(V^\kappa(\kg))=0$. 
	We have the following Jantzen sum formula \cite{gorelik2007simplicity,jantzen1977kontravariante}:
	\begin{align}\label{sumformula}
		\sum_{r=1}^\infty \dim \mathcal{F}^r(V^\kappa(\kg)_\nu)= m_{\kappa+\mathbf{h}^\vee}(\nu),
	\end{align}
	where  $\mathcal{F}^r(V^\kappa(\kg)_\nu)=\mathcal{F}^r(V^\kappa(\kg))\cap V^\kappa(\kg)_\nu$.

	\section{Minimal conformal weight of $N_\kappa(\kg)$}

	Define the formal character
	\begin{align}\label{defMk}
		M_\kappa:=\widehat{R}\sum_\nu m_{\kappa+\mathbf{h}^\vee}(\nu)e^{-\nu}.
	\end{align}
	It follows from the Jantzen sum formula \eqref{sumformula} that
	\begin{align}\label{sumformulaMk}
		\sum_{r=1}^\infty \operatorname{ch} \mathcal{F}^r(V^\kappa(\kg))= \widehat{R}^{-1}M_\kappa.
	\end{align}
	This formula shows that the vacuum module $V^\kappa(\kg)$ is non-simple if and only if $M_\kappa\neq 0$, since the Jantzen filtration is decreasing and $\mathcal{F}^1(V^\kappa(\kg))=N_\kappa(\kg)$.
	
	When $\kappa\neq-\mathbf{h}^\vee$, there is an equivalent form of $M_\kappa$ given in  \cite{gorelik2007simplicity}: 
	\begin{align}\label{M_k}
		M_\kappa
		=\sum_{(r,\gamma):\phi_{r,\gamma}(\kappa)=0}E^\rho(-r\gamma),
	\end{align}
	where
	\begin{align*}
		E^\rho(\lambda)=\sum_{w\in W}(-1)^{\ell(w)} e^{w\circ(\lambda)}.
	\end{align*}
	For any weight $\lambda\in\kh^*$, we shall denote $E^\rho(\lambda-\rho)$ also by $E(\lambda)$, that is,
	$$E(\lambda)=e^{-\rho}\sum_{w\in W}(-1)^{\ell(w)} e^{w(\lambda)}.$$
	It is clear that 
	\begin{align}\label{EWlambda}
		E^\rho(w\circ \lambda)=(-1)^{\ell(w)}E^\rho(\lambda),~~~
		E(w(\lambda))=(-1)^{\ell(w)}E(\lambda).
	\end{align}
	We have the following lemma.
	\begin{lemma}\label{weyl_transition}
		\begin{enumerate}[wide] 
			\item If $\lambda_1=w\circ(\lambda_2)$ for some $w\in W$, then 
			$E^\rho(\lambda_1)=(-1)^{\ell(w)}E^\rho(\lambda_2)$. 
			\item If $E^\rho(\lambda_1)=\pm E^\rho(\lambda_2)\neq 0$, then  $\lambda_1=w\circ(\lambda_2)$ for some $w\in W$.
			\item There exists some  $w\in W\setminus\{1\}$ such that $w\circ \lambda=\lambda$
			 if and only if $E^\rho(\lambda)=0$.
		\end{enumerate}
	\end{lemma}
	\begin{proof}
		\begin{enumerate}[wide]
			\item This follows directly from \eqref{EWlambda}.
			\item The condition $E^\rho(\lambda_1)=\pm E^\rho(\lambda_2)\neq 0$ implies that $e^{w_1\circ(\lambda_1)}=e^{w_2\circ (\lambda_2)}$ for some $w_1,w_2\in W$. So $\lambda_1=(w_1^{-1}w_2)\circ (\lambda_2)$.
			\item If the sum $E^\rho(\lambda)$ is zero, we have   $(-1)^{\ell(w)}e^{w\circ(\lambda)}=-e^\lambda$ for some $w\in W\setminus\{1\}$. 
			The ``only if'' part follows from the fact that the stabilizer for a given weight in $W$  is generated by reflections it contains.\qedhere
		\end{enumerate}
	\end{proof}
	
	\begin{lemma}\label{ReflectionRelation}
		For any $a\in\mathbb{Z}$ and $\alpha\in\Delta$, we have
		$$E^\rho(-a\alpha)=-E^\rho((a-(\rho|\alpha^\vee))\alpha).$$
	\end{lemma}
	\begin{proof}
		Note that
		$s_\alpha\circ(-a\alpha)=(a-(\rho|\alpha^\vee))\alpha$. 
		Then use \Cref{weyl_transition} with $w=s_\alpha$. 
	\end{proof}

	\subsection{Non-critical  levels}
	
	From now on, we consider the case that $V^\kappa(\kg)$ is non-simple and that $\kappa\neq -\mathbf{h}^\vee$. By the equivalent condition given in \Cref{nonsimple}, assume that 
	$$r^\vee(\kappa+\mathbf{h}^\vee)=\dfrac{p}{q} \quad
	\text{with}\quad p\in \mathbb{Z}_{\geq 2},~ q\in \mathbb{Z}_{\geq 1},~ (p,q)=1.$$
	
	Recall the formula \eqref{M_k} for $M_\kappa$.  
	Each root  $\gamma\in\widehat{\Delta}^+\setminus\Delta$ can  be written as $m\delta+\alpha$, where $m\in\mathbb{Z}_{\geq 1}$ and $\alpha\in\Delta\cup\{0\}$. 
	By $\widehat{\rho}=\rho+\mathbf{h}^\vee\Lambda_0$ and $(\delta|\Lambda_0)=1$, we have
	\begin{align*}
		\phi_{r,m\delta+\alpha}(\kappa)
		=m(\kappa+\mathbf{h}^\vee)+(\rho|\alpha)-r\|\alpha\|^2/2.
	\end{align*} 
	So $\phi_{r,m\delta+\alpha}(\kappa)=0$ if and only if $\alpha\in\Delta$ and $r,m\in\mathbb{Z}_{\geq 1}$ satisfy that
	$$mp+r^\vee((\rho|\alpha)-r\|\alpha\|^2/2) q=0.$$
	Using $(p,q)=1$, it follows that
	\begin{align*}
		m=qs,~r^\vee(r\|\alpha\|^2/2-(\rho|\alpha))=ps, ~\text{for some }s\in\mathbb{Z}_{\geq 1}.
	\end{align*}
	Thus we  obtain that 
	\begin{align}\label{MkLayerForm}
		M_\kappa=
		\sum_{r,s\in\mathbb{Z}_{\geq 1}}\sum_{\alpha\in\Delta_{r,s}^p}E^\rho(-r(qs\delta+\alpha)),
	\end{align}
	where $\Delta_{r,s}^p:=\{\alpha\in\Delta : r^\vee(r\|\alpha\|^2/2-(\rho|\alpha))=ps\}$.
	For $D\in\mathbb{Z}_{\geq 1}$, define 
	\begin{align*}
		M_{p,D}:=\sum_{r,s\in\mathbb{Z}_{\geq 1}:rs=D}\sum_{\alpha\in\Delta_{r,s}^p}E^\rho(-r\alpha).
	\end{align*}
	It follows from \eqref{MkLayerForm} that
	\begin{equation}\label{MkMpD}
		M_\kappa=\sum_{D\in\mathbb{Z}_{\geq 1}} e^{-Dq\delta}M_{p,D}.
	\end{equation}
	
	Denote the simple $\kg$-module with highest weight $\lambda$
	by ${L}_\kg(\lambda)$. By complete reducibility, each
	finite-dimensional graded component of a $\kg$-submodule of
	$V^\kappa(\kg)$ is a direct sum of modules
	${L}_\kg(\lambda)$ with $\lambda$ dominant.
	
	\begin{lemma}\label{minimalsingnocapF2}
		Let $D$ be the minimal conformal weight of $N_\kappa(\kg)$ with $\kappa\neq -\mathbf{h}^{\vee}$. Then 
		$$\mathcal{F}^2(V^\kappa(\kg))_{[D]}=\{0\}.$$
		In particular, $(N_\kappa(\kg))_{[D]}\cap \mathcal{F}^2(V^\kappa(\kg))=\{0\}$.
	\end{lemma}
	
	\begin{proof}
		Suppose, to the contrary, that there 
		exists
		$ v\in(M_\kappa^2(\kg,T))_{[D]}$
		whose specialization at $T=0$ is nonzero. Using the PBW
		identification, write $v=v_1+Tv_2+T^2v_3$, 
		where $0\neq v_1\in
		(U(\widehat{\kg}_-)\mathbf{1})_{[D]}$.
		Since $T^2v_3\in M_\kappa^2(\kg,T)$, replacing $v$ by
		$v-T^2v_3$, we may assume that $v=v_1+Tv_2$.
		
		We first claim that, for every $c<D$,
		\begin{equation}\label{eq:lower-degree-M2}
			(M_\kappa^2(\kg,T))_{[c]}
			=
			T^2(M_\kappa(\kg,T))_{[c]}.
		\end{equation}
		Indeed, we have $(N_\kappa(\kg))_{[c]}=0$ by the minimality of $D$.
		Since the
		restriction of $S(\kappa)$ to $V^\kappa(\kg)_{[c]}$ is
		non-degenerate, the determinant 
		of $S(\kappa,T)$ on $(M_\kappa(\kg,T))_{[c]}$ is not divisible
		by $T$. It follows directly from the definition of
		$M_\kappa^2(\kg,T)$ that \eqref{eq:lower-degree-M2} holds.
		
		Since $M_\kappa^2(\kg,T)$ is a
		$\widehat{\kg}$-submodule, for every $x\in\kg$ and
		$m\in\mathbb{Z}_{\geq1}$, we have
		\begin{equation}\label{eq:positive-modes-T2}
			x(m)v\in(M_\kappa^2(\kg,T))_{[D-m]}\subset T^2M_\kappa(\kg,T).
		\end{equation}
		Applying \eqref{eq:Sugawara-L0-T} to $v$ and using
		\eqref{eq:positive-modes-T2}, we obtain
		\[
		\Omega v
		\equiv
		2D(\kappa+T+\mathbf{h}^{\vee})v
		\pmod{T^2M_\kappa(\kg,T)}.
		\]
		Substituting $v=v_1+Tv_2$ and comparing the
		constant terms and the coefficients of $T$, we obtain
		\begin{align}
			\bigl(\Omega-2D(\kappa+\mathbf{h}^{\vee})\bigr)v_1
			&=0,\label{eq:Casimir-v1}\\
			\bigl(\Omega-2D(\kappa+\mathbf{h}^{\vee})\bigr)v_2
			&=2Dv_1.\label{eq:Casimir-v2}
		\end{align}
		
		The space $V^\kappa(\kg)_{[D]}$ is a finite-dimensional
		$\kg$-module and hence is completely reducible. Therefore,
		the Casimir operator $\Omega$ acts semisimply on it. It follows
		that
		\[
		\ker\bigl(\Omega-2D(\kappa+\mathbf{h}^{\vee})\bigr)
		\cap
		\operatorname{Im}
		\bigl(\Omega-2D(\kappa+\mathbf{h}^{\vee})\bigr)
		=\{0\}.
		\]
		However, \eqref{eq:Casimir-v1} shows that $2Dv_1$ belongs to
		the kernel, whereas \eqref{eq:Casimir-v2} shows that $2Dv_1$
		belongs to the image. Since $D>0$ and $v_1\neq0$, this is a
		contradiction. Therefore, we have $\mathcal{F}^2(V^\kappa(\kg))_{[D]}=\{0\}$.
	\end{proof}
	
	\begin{theorem}\label{thmNk}
		Let $\kg$ be a simple finite-dimensional Lie algebra. 
		Assume that $r^\vee(\kappa+\mathbf{h}^\vee)=\dfrac{p}{q}$ with $p\in \mathbb{Z}_{\geq 2}$, $q\in \mathbb{Z}_{\geq 1}$ and $(p,q)=1$. 
		Let $D_p$ be the minimal integer $D$ such that $M_{p,D}\neq 0$, and assume that 
		\begin{align}\label{M_pD2E}
			M_{p,D_p}=a_1E^\rho(\lambda_1)+a_2E^\rho(\lambda_2)+\cdots+a_mE^\rho(\lambda_m),
		\end{align}
		where $a_1,a_2,\dots,a_m\in\mathbb{Z}_{>0}$, and $\lambda_1,\lambda_2,\dots,\lambda_m\in\kh^*$ are dominant and distinct. Then for the maximal submodule $N_\kappa(\kg)$, we have:
		\begin{enumerate}
			\item For $c<D_pq$, $N_\kappa(\kg)_{[c]}=0$.
			\item As a $\kg$-module, 
			$$N_\kappa(\kg)_{[D_pq]}=a_1{L}_\kg(\lambda_1)\oplus a_2{L}_\kg(\lambda_2)\oplus\cdots\oplus a_m{L}_\kg(\lambda_m).$$
			Moreover, up to scalars, there are exactly $a_i$ linearly independent singular vector(s) of weight $\kappa\Lambda_0-D_pq\delta+\lambda_i$ for each $i$, while no other linearly independent singular vector associated with conformal weight $D_pq$ exists.
		\end{enumerate}
		
	\end{theorem}
	
	\begin{proof}
		Combining \eqref{sumformulaMk} and \eqref{MkMpD}, 
		we have
		\begin{align}\label{sumformulaexpand}
			\sum_{r=1}^\infty \operatorname{ch} \mathcal{F}^r(V^\kappa(\kg)) 
			=\widehat{R}^{-1}(e^{-D_pq\delta}M_{p,D_p}+\sum_{D>D_p} e^{-Dq\delta}M_{p,D}).
		\end{align}
		
		\begin{enumerate}[wide]
			\item Focusing on the terms associated with conformal weight $c<D_pq$ in \eqref{sumformulaexpand}, we  obtain that
			$$\sum_{r=1}^\infty \operatorname{ch}  \mathcal{F}^r(V^\kappa(\kg))_{[c]}=0.$$
			Recall that $\mathcal{F}^1(V^\kappa(\kg))=N_\kappa(\kg)$ and that $\mathcal{F}^r(V^\kappa(\kg))\subset \mathcal{F}^1(V^\kappa(\kg))$ for all $r$. It follows that $N_\kappa(\kg)_{[c]}=0$.
			
			\item Focusing on the terms associated with conformal weight $D_pq$ in \eqref{sumformulaexpand}, we can obtain that
			$$\sum_{r=1}^\infty \operatorname{ch}  {\mathcal{F}^r(V^\kappa(\kg))}_{[D_pq]}=R^{-1}M_{p,D_p},~~~\text{ where }R=\prod_{\alpha\in\Delta^+}(1-e^{-\alpha}).$$
			So 
			\begin{align}\label{chsimplemulti}
				R^{-1}M_{p,D_p}
				=\sum_{\lambda}\sum_{r=1}^\infty [  {\mathcal{F}^r(V^\kappa(\kg))}_{[D_pq]}:{L}_\kg(\lambda)]\operatorname{ch} {L}_\kg(\lambda),
			\end{align}
			where $[M:{L}_\kg(\lambda)]$ is the multiplicity of ${L}_\kg(\lambda)$ in  $M$.
			Note that 
			$R\operatorname{ch} {L}_\kg(\lambda)=E^\rho(\lambda)$ holds for dominant $\lambda$.
			So the assumption \eqref{M_pD2E} is equivalent to that
			$$R^{-1}M_{p,D_p}=a_1\operatorname{ch} {L}_\kg(\lambda_1)+a_2\operatorname{ch} {L}_\kg(\lambda_2)+\cdots+a_m\operatorname{ch} {L}_\kg(\lambda_m).$$
			Since the Jantzen filtration is decreasing, the coefficient of $\operatorname{ch} {L}_\kg(\lambda)$ in \eqref{chsimplemulti} is nonzero if and only if 
			$$[N_\kappa(\kg)_{[D_pq]}:{L}_\kg(\lambda)]=[{\mathcal{F}^1(V^\kappa(\kg))}_{[D_pq]}:{L}_\kg(\lambda)]\neq 0.$$
			It follows from \Cref{minimalsingnocapF2} that 
			$$\operatorname{ch} N_\kappa(\kg)_{[D_pq]}=a_1\operatorname{ch} {L}_\kg(\lambda_1)+a_2\operatorname{ch} {L}_\kg(\lambda_2)+\cdots+a_m\operatorname{ch} {L}_\kg(\lambda_m).$$
			Since $\dim\operatorname{Hom}({L}_\kg(\lambda),{L}_\kg(\mu))=\delta_{\lambda,\mu}$ for $\lambda,\mu$ both dominant, we have
			$$N_\kappa(\kg)_{[D_pq]}=a_1{L}_\kg(\lambda_1)\oplus a_2{L}_\kg(\lambda_2)\oplus\cdots\oplus a_m{L}_\kg(\lambda_m)$$
			as a $\kg$-module. Then the result follows.
			\qedhere
		\end{enumerate}
	\end{proof}

	We have the following corollary which is also covered in \cite{KW89}.
	
	\begin{coro}\label{integrable_case}
		Let $\kg$ be a simple finite-dimensional Lie algebra. 
		Assume that $r^\vee(\kappa+\mathbf{h}^\vee)=\dfrac{p}{q}$ with $p\in \mathbb{Z}_{\geq 2}$, $q\in \mathbb{Z}_{\geq 1}$, $(p,q)=1$ and $\frac{p}{r^\vee}-\mathbf{h}^\vee\in\mathbb{Z}_{\geq 0}$.
		Then the weight of singular vectors with the minimal conformal weight is
		$$\kappa\Lambda_0-(\frac{p}{r^\vee}-\mathbf{h}^\vee+1)(q\delta-\theta),$$
		where $\theta$ is the highest root of $\kg$.
	\end{coro}
	
	\begin{proof}
		By \Cref{thmNk}, $D_p$ and $M_{p,D_p}$ are independent of $q$, so we may assume that $q=1$, that is, $\kappa$ is a non-negative integer. Then the result follows from the integrable case.
	\end{proof}
	
	\subsection{When $\kg$ is simply-laced}
	
	From now on, we assume that $\kg$ is of type $\mathsf{A}$, $\mathsf{D}$, or $\mathsf{E}$. Then $r^\vee=1$ and $\|\alpha\|^2=2$ for all roots of $\kg$. 
	It follows from $\rho(\alpha^\vee)=\rht(\alpha)$ that
	$$\Delta_{r,s}^p=\{\alpha\in \Delta:\rht(\alpha)=r-ps\}.$$
	And thus
	\begin{align*}
	M_{p,D}=&\sum_{(r,s\in\mathbb{Z}_{\geq 1}:rs=D)}\ \ \sum_{\alpha\in\Delta_{r,s}^p}E^\rho(-r\alpha)\\
	=&\sum_{r,s\in\mathbb{Z}_{\geq 1}, rs=D} \ \ \sum_{\alpha\in\Delta,\rht(\alpha)=r-ps}E^\rho(-r\alpha).
\end{align*}
	Taking the substitution $a=r$ and $b=ps$, we have
	\begin{align}\label{MkD}
		M_{p,D}=\sum_{ab=pD,p|b} E_{a,b},
	\end{align}
	where 
	\begin{align}\label{Eab}
	E_{a,b}:=\sum\limits_{\alpha\in\Delta, \rht(\alpha)=a-b} E^\rho(-a\alpha).
\end{align}	
	\begin{lemma}\label{EabEba}
		$E_{a,b}=-E_{b,a}$ for all $a,b\in\mathbb{Z}_{\geq 1}$.
	\end{lemma}
	\begin{proof}
		Since $\alpha^\vee=\alpha$ for each root $\alpha\in \Delta$, we have
		$E^\rho(-a\alpha)=-E^\rho((a-\rht(\alpha))\alpha)$ by \Cref{ReflectionRelation}.
		Then
		\begin{align*}
			E_{a,b}
			=&\sum_{\alpha\in\Delta,\rht(\alpha)=a-b} E^\rho(-a\alpha)
			=-\sum_{\alpha\in\Delta,\rht(\alpha)=a-b} E^\rho((a-\rht(\alpha))\alpha)\\
			=&-\sum_{\alpha\in\Delta,\rht(\alpha)=a-b} E^\rho(b\alpha)
			=-\sum_{-\beta\in\Delta,\rht(-\beta)=a-b} E^\rho(b(-\beta))\\
			=&-\sum_{\beta\in\Delta,\rht(\beta)=b-a} E^\rho(-b\beta)=-E_{b,a}.\qedhere
		\end{align*}
	\end{proof}	
	Let $\mathbf{P}$ be a subset of $\mathbb{Z}_{\geq 1}^2$. 
	We say $\mathbf{P}$ is symmetric if $(a,b)\in \mathbf{P}$ implies $(b,a)\in \mathbf{P}$. The ``transpose'' of $\mathbf{P}$ is denoted by $\mathbf{P}'=\{(a,b):(b,a)\in \mathbf{P}\}$.
	Define the subset $A(\mathbf{P}):=\mathbf{P}\setminus\mathbf{P}'$. 
	It is clear that a set $\mathbf{P}$ is symmetric if and only if $A(\mathbf{P})=\varnothing$.
	
	For $\mathbf{P}\subset \mathbb{Z}_{\geq 1}^2$, define the formal character 
	$$E_{\mathbf{P}}:=\sum_{(a,b)\in\mathbf{P}}E_{a,b}.$$
	Then we have
	\begin{align}\label{E_A(P)}
		E_{\mathbf{P}}=E_{A(\mathbf{P})}.
	\end{align}
	Indeed, it follows from \Cref{EabEba} that
	\begin{align*}
		2E_{\mathbf{P}}=&E_{\mathbf{P}}-E_{{\mathbf{P}'}}
		=(E_{\mathbf{P}\setminus{\mathbf{P}'}}+E_{\mathbf{P}\cap{\mathbf{P}'}})-(E_{{\mathbf{P}'}\setminus \mathbf{P}}+E_{\mathbf{P}\cap{\mathbf{P}'}})\\
		=&E_{\mathbf{P}\setminus{\mathbf{P}'}}-E_{{\mathbf{P}'}\setminus \mathbf{P}}
		=2E_{\mathbf{P}\setminus{\mathbf{P}'}}=2E_{A(\mathbf{P})}.
	\end{align*}
	
	Set $$\mathbf{P}_{p,D}:=\{(a,b)\in\mathbb{Z}_{\geq 1}^2:ab=pD,p|b\}.$$ Then we can write
	\begin{align}\label{ee1}
		M_{p,D}=E_{\mathbf{P}_{p,D}}.
	\end{align}

	\section{Type $\mathsf{A}$}
	
	This section is dedicated to the proof of \Cref{JS25b}.
	Let $\kg$ be the Lie algebra of type $\mathsf{A}_{n-1}$, that is, $\kg=\ksl_n$.
	Recall from  \cite{Humphreys} (see also \cite{tauvel2005lie}) that the simple roots of $\ksl_n$ can be realized by taking $\alpha_i=\epsilon_i-\epsilon_{i+1}$, where $\epsilon_1,\dots,\epsilon_n$ are orthonormal.
	The root system is the set of all vectors $\epsilon_i-\epsilon_j$ with $1\leq i\neq j\leq n$.  In addition, we have $\rht(\epsilon_i-\epsilon_j)=j-i$, and the weight 
	\begin{equation}\label{rho_v}
		\rho=\dfrac{1}{2}[(n-1)\epsilon_1+(n-3)\epsilon_2+\cdots-(n-3)\epsilon_{n-1}-(n-1)\epsilon_n].
	\end{equation}
	The Weyl group $W$ is the symmetric group $\mathfrak{S}_n$ acting on $\epsilon_1, \epsilon_2,\dots,\epsilon_n$. 
	
	In the discussion of type $\mathsf{A}$, we shall write $a_1\epsilon_1+a_2\epsilon_2+\cdots+a_n\epsilon_n$ as an $n$-tuple $[a_1,a_2,\dots,a_n]$ for brevity. Then the $n$-tuple $[a_1,a_2,\dots,a_n]$ is in $\kh^*$ if and only if $a_1+a_2+\cdots+a_n=0$. 
	To be specific, the corresponding weight of $[a_1,a_2,\dots,a_n]$ is 
	\begin{align}\label{v2alpha}
		a_1\alpha_1+(a_1+a_2)\alpha_2+(a_1+a_2+a_3)\alpha_3+\cdots+(a_1+a_2+\cdots+a_{n-1})\alpha_{n-1}.
	\end{align}
	An $n$-tuple $[a_1,a_2,\dots,a_n]$ is said to be decreasing if $a_1\geq a_2\geq \cdots\geq a_n$.  We have the following lemma.

	\begin{lemma}\label{A:dominant_in_orbit}
		The $W$-orbit of weight $[a_1,a_2,\dots,a_n]$ is the set of $n$-tuples with all permutations of $a_1,a_2,\dots,a_n$. Moreover, the unique highest weight in the orbit is the decreasing one among all the permutations. 
	\end{lemma}
	\begin{proof}
		Without loss of generality, we may assume that $[a_1,a_2,\dots,a_n]$ is decreasing. 
		Then for any simple root $\alpha_i$, we have
		\begin{align*}
			(\alpha_i|[a_1,a_2,\dots,a_n])=a_i-a_{i+1}\geq 0.
		\end{align*}
		This implies that $[a_1,a_2,\dots,a_n]$ is highest in its orbit.
	\end{proof}
	
	\begin{lemma}\label{A:identical_elimination}
		Let $\lambda$ be a weight, with its $n$-tuple being $[a_1,a_2,\dots,a_n]$. Then the following conditions are equivalent:
		\begin{enumerate}
			\item $a_i=a_j$ for some  $i,j$ distinct;
			\item $E(\lambda)=E^\rho(\lambda-\rho)=0$.
		\end{enumerate}
	\end{lemma}
	\begin{proof}
		By \Cref{weyl_transition}, the second condition holds if and only if there exists some $w\in W\setminus\{1\}$ such that $w(\lambda)=\lambda$. 
		If $a_i=a_j$  for some  $i,j$ distinct, take $w$ to be the transposition of $i$ and $j$ in $\mathfrak{S}_n$. Otherwise, all $a_i$'s are distinct. Then $w(\lambda)=\lambda$ holds only when $w=1$. 
	\end{proof}
	
	\subsection{$M_{p,D}$ for type $\mathsf{A}$}
	
	Assume that  $a\in\mathbb{Z}_{\geq 1}$ and $\alpha=-(\epsilon_i-\epsilon_j)\in\Delta^-$. 
	For brevity, we shall denote $[a_1-a,a_2-a,\dots,a_n-a]$ by   $[a_1,a_2,\dots,a_n]-a$. Then the weight $\rho$ can be written as $[n-1,n-2,\dots,0]-\frac{n-1}{2}$.
	The weight $\rho-a\alpha$ is
	$$[n-1,\dots,n-(i-1),n-i+a,n-(i+1),\dots,n-(j-1),n-j-a,n-(j+1),\dots,0]-\dfrac{n-1}{2}.$$
	It follows from \Cref{A:identical_elimination} that
	\begin{align}\label{A:E_nonzero_cond}
		E^\rho(-a\alpha)\neq 0\iff i\leq a,~ j>n-a,
	\end{align}
	in this case, by \Cref{A:dominant_in_orbit}, the corresponding dominant weight in the $W$-orbit of $\rho-a\alpha$, i.e., $\{w(\rho-a\alpha):w\in W\}$, is 
	\begin{align}\label{A:dominant_weight}
		[n-i+a,n-1,n-2,\dots,\widehat{n-i},\dots,\widehat{n-j},\dots,0,n-j-a]-\dfrac{n-1}{2}
	\end{align}
	Notice that there are no two non-zero $E^\rho(-a\alpha)$'s with $a\in\mathbb{Z}_{\geq 1}$ and $\alpha\in\Delta^-$ satisfying that their dominant weights \eqref{A:dominant_weight} are the same. 
	Thus we have the following lemma.
	\begin{lemma}\label{A:uniqueness of dominant weight}
		All non-zero $E^\rho(-a\alpha)$'s with $a\in\mathbb{Z}_{\geq 1}$ and $\alpha\in\Delta^-$ are linearly independent. 
	\end{lemma}

	For $a,b\in\mathbb{Z}_{\geq 1}$, define the following sets of $(a,b)$:
	\begin{align*}
		\mathbf{S}_{-}:=\{(a,b):a<b,b-a<n<a+b\}, \ \ \mathbf{S}_{+}:=\mathbf{S}_{-}', \ \  \mathbf{S}:=\mathbf{S}_{+}\cup \mathbf{S}_{-},
	\end{align*}
where $\mathbf{S}_{-}'$ is the transpose of $\mathbf{S}_{-}$.
	Notice that $\mathbf{S}_{+}\cap \mathbf{S}_{-}=\varnothing$, and 
	$$\mathbf{S}=\{(a,b):0<|a-b|<n<a+b\}.$$

	\begin{lemma}\label{A:nonzero Eab}
		$E_{a,b}\neq 0$ if and only if $(a,b)\in \mathbf{S}$.
	\end{lemma}
	\begin{proof}
		By \Cref{EabEba}, we may assume that $a<b$. 
		If $E_{a,b}\neq 0$ and $a<b$, there exists some non-zero  $E^\rho(-a\alpha)$ occurring in the sum $E_{a,b}$.  Write $\alpha=-(\epsilon_i-\epsilon_j)\in\Delta^-$. It follows from \eqref{A:E_nonzero_cond} that $|a-b|=|\rht(\alpha)|<n$ and 
		$a+b=a+a-(a-b)>(n-j)+i-\rht(\alpha)=n$. 
		So $(a,b)\in \mathbf{S}$.
		
		If $(a,b)\in \mathbf{S}\cap \{(a,b):a<b\}=\mathbf{S}_-$, then the term $E^\rho(-a\alpha)$ with 
		$$\alpha=
		\begin{cases}
			-(\epsilon_{n+a-b}-\epsilon_n),&b\geq n,\\
			-(\epsilon_a-\epsilon_b),&b<n
		\end{cases}$$
		is non-zero and appears in $E_{a,b}$ by \eqref{A:E_nonzero_cond}.
	\end{proof}

	Combining \Cref{EabEba}, \Cref{A:uniqueness of dominant weight} and \Cref{A:nonzero Eab}, we deduce that the sum $E_{\mathbf{P}}=0$ 
	if and only if $\mathbf{P}\cap \mathbf{S}$ is symmetric. Then it follows from \eqref{E_A(P)}  that $$E_{\mathbf{P}}=E_{A(\mathbf{P})}=E_{A(\mathbf{P})\cap\mathbf{S}},$$ and  the sum $E_{\mathbf{P}}=0$
	if and only if $A(\mathbf{P})\cap \mathbf{S}=\varnothing$. 
	Define $${\mathbf{Q}}_{p,D}:=A(\mathbf{P}_{p,D})=\{(a,b):ab=pD,p|b,p\nmid a\}.$$
	Then the following lemma follows from \eqref{ee1}. 
	\begin{lemma}\label{A:criterion4M}
		We have  $M_{p,D}=E_{\mathbf{Q}_{p,D}\cap\mathbf{S}}$. The sum $M_{p,D}=0$
		if and only if $\mathbf{Q}_{p,D}\cap \mathbf{S}=\varnothing$. 
	\end{lemma}
	
	\subsection{Values of $D_p$ and corresponding weights}
	
	For simplicity, we shall simply use the notation $\mathbf{Q}$ instead of $\mathbf{Q}_{p,D}$ if there is no ambiguity. 
	
	\begin{lemma}\label{A:Dp} 
		For $n,p\in\mathbb{Z}_{\geq 2}$ with $p<n$, we have the following results on $D_p$:
		\begin{enumerate}[wide] 
			\item When $p=2$, $D_p\geq n/2$ if $2|n$, while $D_p\geq n$ if $2\nmid n$.
			\item When $n=5$ and $p=3$, $D_p\geq 4$;
			\item When $n=8$ and $p=3$, $D_p\geq 6$;
			\item When $n=7$ and $p=4$, $D_p\geq 4$;
			\item For other cases, set $s_1=\lfloor n/p\rfloor, s_2=\lceil n/p\rceil$ and define a function $D(s)=(|sp-n|+1)s$. Then $D_p\geq\min\{D(s_1),D(s_2)\}$.  
		\end{enumerate}
	\end{lemma}
	\begin{proof}
		By \Cref{A:criterion4M}, it suffices to show that  $\mathbf{Q}_{p,D}\cap\mathbf{S}=\varnothing$ for $D$ less than the given bound. 
		\begin{enumerate}[wide]
			\item Assume that $2|n$ and $D<n/2$. 
			Take $(a,b)\in \mathbf{Q}\cap\mathbf{S}$. 
			We have $a=pD/b\leq D<n/2$ since $b\geq p=2$. 
			If $a=1$, then $b=pD/a=2D$ and thus $a+b=2D+1\leq n-1$, which contradicts $(a,b)\in \mathbf{S}$. 
			Also $a\neq 2$ since $p\nmid a$.
			In summary, we have $3\leq a<n/2$, and hence 
			$$ab\geq a(n+1-a)\geq 3(n-2)> n-2\geq pD,$$
			which contradicts $ab= pD$.
			So the set $\mathbf{Q}\cap\mathbf{S}$ is empty.
			
			Assume that $2\nmid n$ and $D<n$. 
			Take $(a,b)\in \mathbf{Q}\cap\mathbf{S}$. 
			Analogously, we have $a<n$. 
			If $a=1$, then $b=pD/a=2D$, which is absurd since  $b=n$ by $(a,b)\in \mathbf{S}$ while $n$ is odd.
			Also $a\neq n-1$ since $2\nmid a$.
			In summary, we have $3\leq a\leq n-2$, and hence $n\geq 5$. Then it follows that
			$$ab\geq a(n+1-a)\geq 3(n-2)> 2(n-1)\geq pD,$$
			which contradicts $ab= pD$.
			So the set $\mathbf{Q}\cap\mathbf{S}$ is empty.
			\item Assume that $D\leq 3$.  Take $(a,b)\in \mathbf{Q}\cap\mathbf{S}$. 
			Then $ab=pD\leq 9$, so $(a,b)\in\{(1,3),(2,3),(1,6),(1,9)\}$ by $(a,b)\in \mathbf{Q}$. However, none of these is in $\mathbf{S}$.
			\item Assume that $D\leq 5$.  Take $(a,b)\in \mathbf{Q}\cap\mathbf{S}$. 
			Then $ab=pD\leq 15$, so $(a,b)\in\{(1,3),(2,3),(4,3),(5,3),(1,6),(2,6),(1,9),(1,12),(1,15)\}$ by $(a,b)\in \mathbf{Q}$. However, none of these is in $\mathbf{S}$.
			\item Assume that $D\leq 3$.  Take $(a,b)\in \mathbf{Q}\cap\mathbf{S}$. 
			Then $ab=pD\leq 12$, so $(a,b)\in\{(1,4),(2,4),(3,4),(1,8),(1,12)\}$ by $(a,b)\in \mathbf{Q}$. However, none of these is in $\mathbf{S}$.
			\item 
			Take $(a,b)\in \mathbf{Q}\cap\mathbf{S}$. 
			Write $b=ps$, and then 
			$$pD=ab\geq (|b-n|+1)b=ps(|ps-n|+1).$$
			The result follows by applying \Cref{Appen1} to $D=F_{n,p,1}$.\qedhere
		\end{enumerate}
	\end{proof}

	\begin{proof}[Proof of \Cref{JS25b}]
		By \Cref{thmNk}, \Cref{A:criterion4M} and \Cref{A:Dp}, it suffices to show that $\mathbf{Q}_{p,D}\cap \mathbf{S}\neq \varnothing$ for the given value of $D$. 
		Then the weights of minimal singular vectors are given by those dominant $\lambda_i$'s in the formula \eqref{M_pD2E}.
		Recall that the dominant weight in $E^\rho(\cdot)$ can be calculated by \Cref{A:dominant_in_orbit}. 
		\begin{enumerate}[wide]
			\item This case follows from \Cref{integrable_case}.
			\item Assume that $2|n$ and $D=n/2$. We can obtain from the proof of \Cref{A:Dp} that $\mathbf{Q}\cap \mathbf{S}=\{(1,n)\}$. So 
			$$M_{p,D}=E_{1,n}=E^\rho(\epsilon_1-\epsilon_n).$$
			
			Assume that $2\nmid n$ and $D=n$. We can obtain from the proof of \Cref{A:Dp} that $\mathbf{Q}\cap \mathbf{S}=\{(n,2)\}$. So 
			$$M_{p,D}=E_{n,2}=E^\rho(-n(\epsilon_1-\epsilon_{n-1}))+E^\rho(-n(\epsilon_2-\epsilon_{n})).$$
			\item Set $D=4$. Then $\mathbf{Q}\cap \mathbf{S}=\{(4,3),(2,6)\}$, and thus 
			\begin{align*}
				M_{p,D}=&E_{4,3}+E_{2,6}=\sum_{i=1}^4 E^\rho(-4(\epsilon_i-\epsilon_{i+1}))+E^\rho(2(\epsilon_1-\epsilon_5))\\
				=&E^\rho(-4(\epsilon_2-\epsilon_3))+E^\rho(-4(\epsilon_3-\epsilon_4))+E^\rho(2(\epsilon_1-\epsilon_5)).
			\end{align*}
			\item Set $D=6$. Then $\mathbf{Q}\cap \mathbf{S}=\{(2,9)\}$, and thus 
			$M_{p,D}=E_{2,9}=E^\rho(2(\epsilon_1-\epsilon_8))$. 
			\item Set $D=4$. Then $\mathbf{Q}\cap \mathbf{S}=\{(2,8)\}$, and thus 
			$M_{p,D}=E_{2,8}=E^\rho(2(\epsilon_1-\epsilon_7))$.
			
			\item Set $D=\min\{D(s_1),D(s_2)\}$. 
			We claim that $\mathbf{Q}\cap \mathbf{S}=\{(|n-ps_i|+1,ps_i):D(s_i)=D\}$. 
			And then $M_{p,D}$ follows from \Cref{A:criterion4M}.
			
			Let $(a,b)\in\mathbf{Q}\cap\mathbf{S}$. Since $p\mid b$, write
			$b=ps$ with $s\in\mathbb{Z}_{\geq1}$. By \Cref{A:Dp} and \Cref{Appen1}, we can obtain that $D_p= \min\{D(s_1),D(s_2)\}=D(s)$. In particular,
			$s\in\{s_i:D(s_i)=D\}$, and hence $
			\mathbf{Q}\cap\mathbf{S}
			\subset
			\{(|n-ps_i|+1,ps_i):D(s_i)=D\}$.
			
			Set $(a,b)=(|n-ps_i|+1,ps_i)$ with $D(s_i)=D$.
			We have $pD=ab$, $p|b$ and $|n-b|<a$. Then it suffices to show that $a<p$, which implies that $a< b$ and $p\nmid a$, thus $(a,b)\in \mathbf{Q}\cap \mathbf{S}$.
			Here we use the notation  $n_0:=n-ps_1$.
			
			\begin{enumerate}
				\item If $n_0=0$, we have $s_1=s_2$, $D=D(s_1)=D(s_2)$ and $(a,b)=(1,n)$. It follows that
				$E_{a,b}=E^{\rho}(\epsilon_{1}-\epsilon_n)$.
				\item If $n_0\neq 0$, then $s_2=s_1+1$. 
				\begin{enumerate}
					\item Suppose that $D=D(s_1)\leq D(s_2)$ and 
					$(a,b)=(n_0+1,ps_1)$. 
					Then we have $n_0<p-1$. 
					Indeed, if $n_0=p-1$, then
					$$0\leq D(s_2)-D(s_1)=2(s_1+1)-s_1p=2-(p-2)s_1.$$
					The inequality holds if and only if $(s_1,p)=(1,3),(1,4)$ or $(2,3)$, that is, $(n,p)=(5,3),(7,4)$ or $(8,3)$, which are the cases discussed before.
					
					So $a=n_0+1<p$. It follows from \Cref{A:identical_elimination} that
					$E_{a,b}=E^{\rho}(a(\epsilon_a-\epsilon_{n-(a-1)}))$. 
					The root $\epsilon_a-\epsilon_{n-(a-1)}$ is positive since  $$n+1=ps_1+n_0+1=ps_1+a\geq p+a>2a.$$ 
					
					\item Suppose that $D=D(s_2)\leq D(s_1)$ and 
					$(a,b)=(p-n_0+1,ps_2)$.
					Then $n_0>1$. Indeed, if $n_0=1$, then 
					$$0\leq D(s_1)-D(s_2)=2s_1-p(s_1+1)=-(p-2)(s_1+1)-2\leq -2.$$
					
					So $a=p-n_0+1<p$. It follows that
					$E_{a,b}=E^{\rho}(a(\epsilon_{1}-\epsilon_n))$.
				\end{enumerate}
			\end{enumerate}
			
			Note that there are two minimal singular vectors when $|\mathbf{Q}\cap \mathbf{S}|=2$. As discussed above, this condition is equivalent to  $D(s_1)=D(s_2)$ and $n_0\neq 0$, that is, $p\nmid n$.\qedhere
		\end{enumerate} 
	\end{proof}

	\section{Types $\mathsf{D}$ and $\mathsf{E}$}
	This section is dedicated to the proof of \Cref{main2}. 
	Let $\kg$ be the Lie algebra of type $\mathsf{D}_{n}(n\geq 4)$, that is, $\kg=\kso_{2n}$.
	Recall that \cite{Humphreys} (see also \cite{tauvel2005lie}) the simple roots of $\kg$ can be realized by taking  $\alpha_i=\epsilon_i-\epsilon_{i+1}(i=1,2,\dots,n-1)$ and $\alpha_n=\epsilon_{n-1}+\epsilon_n$, where $\epsilon_1,\dots,\epsilon_n$ are orthonormal.
	The root system is the set of all vectors $\pm(\epsilon_i\pm\epsilon_j)$ with $1\leq i,j\leq n$ distinct.
	In addition, we have $\rht(\epsilon_i\pm\epsilon_j)=(n-i)\pm(n-j)$.
	The weight 
	\begin{equation}\label{rho_v_D}
		\rho=(n-1)\epsilon_1+(n-2)\epsilon_2+\cdots+\epsilon_{n-1}.
	\end{equation}
	The Weyl group $W$ is the group generated by permutations of the set $\{\epsilon_1,\dots,\epsilon_n\}$ and sign changes in an even number of coordinates, hence isomorphic to the $(\mathbb{Z}_2)^{n-1}\rtimes\mathfrak{S}_n$. 
	
	In the discussion of type $\mathsf{D}$, we shall write $a_1\epsilon_1+a_2\epsilon_2+\cdots+a_n\epsilon_n$ as an $n$-tuple $[a_1,a_2,\dots,a_n]$ for brevity.  Then we can obtain that:
	\begin{lemma}\label{D:dominant_in_orbit}
		The $W$-orbit of weight $[a_1,a_2,\dots,a_n]$ is the set
		$$\{[(-1)^{m_1}a_{\sigma(1)},\dots,(-1)^{m_n}a_{\sigma(n)}]:\sigma\in \mathfrak{S}_n,m_i=0\text{ or } 1,~2|m_1+\cdots+m_n\}.$$
		Moreover, the unique highest weight $[b_1,b_2,\dots,b_n]$ in the orbit is the one satisfying $b_1\geq b_2\geq \cdots\geq b_{n-1}\geq |b_n|$.  
	\end{lemma}
	
	\begin{proof}
		Without loss of generality, we can assume that $[a_1,a_2,\dots,a_n]$ satisfies that $a_1\geq a_2\geq \cdots\geq a_{n-1}\geq |a_n|$. 
		Then we have 
		\begin{align*}
			&(\alpha_i|[a_1,a_2,\dots,a_n])=a_i-a_{i+1}\geq 0,~~i=1,2,\dots,n-1,\\
			&(\alpha_n|[a_1,a_2,\dots,a_n])=a_{n-1}+a_n\geq 0.
		\end{align*}
		This implies that $[a_1,a_2,\dots,a_n]$ is highest in its orbit.
	\end{proof}
	
	\begin{lemma}\label{D:identical_elimination}
		Let $\lambda$ be a weight, with its $n$-tuple being $[a_1,a_2,\dots,a_n]$. Then the following conditions are equivalent:
		\begin{enumerate}
			\item $|a_i|=|a_j|$ for some  $i,j$ distinct;
			\item $E(\lambda)=E^\rho(\lambda-\rho)=0$.
		\end{enumerate}
	\end{lemma}
	\begin{proof}
		The proof is analogous to \Cref{A:identical_elimination}. The absolute values are used because the Weyl group can change signs by an even number of sign flips.
	\end{proof}
	
	\subsection{$E_{\mathbf{P}}$ for type $\mathsf{D}$}
	
	Assume that  $a\in\mathbb{Z}_{\geq 1}$ and $\alpha \in\Delta^-$. 
	Recall that a root  $\alpha\in\Delta^-$ must be of the form $-(\epsilon_i+\epsilon_j)$  or $-(\epsilon_i-\epsilon_j)$ with $1\leq i<j\leq n$. So here we discuss the two cases separately:
	\begin{enumerate}
		\item When $\alpha=-(\epsilon_i+\epsilon_j)$,  it follows that
		$$E^\rho(-a\alpha)=E[n-1,n-2,\dots,n-i+a,\dots,n-j+a,\dots,0].$$
		From \Cref{D:identical_elimination}, $E^\rho(-a\alpha)$ is non-zero if and only if
		one of the following conditions holds:
		\begin{enumerate}
			\item $n-j+a\geq n$;
			\item $n-i+a\geq n,n-j+a=n-i$.
		\end{enumerate}
		\item When $\alpha=-(\epsilon_i-\epsilon_j)$,  it follows that
		$$E^\rho(-a\alpha)=E[n-1,n-2,\dots,n-i+a,\dots,n-j-a,\dots,0].$$
		From \Cref{D:identical_elimination}, $E^\rho(-a\alpha)$ is non-zero if and only if
		one of the following conditions holds:
		\begin{enumerate}
			\item $-(n-j-a)\geq n$;
			\item $n-i+a\geq n,-(n-j-a)=n-i$;
			\item $n-i+a\geq n,-(n-j-a)=n-j$.
		\end{enumerate}
	\end{enumerate}
	
	By the discussion above, we can obtain the following lemma.
	\begin{lemma}\label{D:EwithNeg}
		Assume that $a\in\mathbb{Z}_{\geq 1}$ and $\alpha\in\Delta^{-}$. 
		If $E^{\rho}(-a\alpha)$ is non-zero, it must be one of the following cases:   
		\begin{enumerate}[wide] 
			\item $\alpha=-(\epsilon_i\pm\epsilon_n)(1\leq i\leq n-1)$, $a\geq i$ and $a=n-i$, then 
			$$E^\rho(-a\alpha)=(-1)^{n} E([n-i+a,n-1,\dots,2,\pm 1]).$$
			\item $\alpha=-(\epsilon_i\pm\epsilon_j)(1\leq i<j\leq n-1)$, $a\geq i$ and $n-i=\pm(n-j)+a$, then
			$$E^\rho(-a\alpha)=(-1)^{j} E([n-i+a,n-1,\dots,\widehat{n-j},\dots, 1,0]).$$
			\item $\alpha=-(\epsilon_i-\epsilon_j)(1\leq i<j\leq n-1)$, $a\geq i$ and $2n=2j+a$, then
			$$E^\rho(-a\alpha)=-(-1)^{i} E([n-i+a,n-1,\dots,\widehat{n-i},\dots, 1,0]).$$
			\item $\alpha=-(\epsilon_i\pm\epsilon_j)(1\leq i<j\leq n)$ and $n\leq \pm(n-j)+a$, then
			$$E^\rho(-a\alpha)=
			\begin{cases}
				(-1)^{i+n+1} E([n-i+a,a,n-1,\dots,\widehat{n-i},\dots,2,\pm 1]),&j=n;\\
				(-1)^{i+j+1} E([n-i+a,\pm(n-j )+a,\dots,\widehat{n-i},\dots,\widehat{n-j},\dots, 1,0]),&j\neq n.
			\end{cases}$$
		\end{enumerate}
		Moreover, all the weights given on the right-hand side are dominant.
	\end{lemma}
	
	Define the following sets:
	$$\mathbf{X}:=\{(a,b)\in\mathbb{Z}_{\geq 1}^2:  2|b, b\geq 2n-2a,b>2a-2n,b<2n-1,a\neq b\},$$
	and
	$$\mathbf{Y}:=\{(a,b)\in\mathbb{Z}_{\geq 1}^2:a+b>2n,1\leq b-a\leq 2n-3\}.$$
	For $(a,b)\in \mathbf{X}$, define 
	$$X_{a,b}=\begin{cases}
		(-1)^{n}(E([2a,n-1,\dots,2,1])+E[2a,n-1,\dots,2,-1]),&2a=b;\\
		(-1)^{n+a+\frac{b}{2}} E([a+\frac{b}{2},n-1,n-2,\dots,\widehat{|a-\frac{b}{2}|},\dots, 1,0]),&2a\neq b.
	\end{cases}$$
	For $(a,b)\in \mathbf{Y}$, define
	\begin{align*}
		R_{a,b}^+&:=\{\alpha=-(\epsilon_i+\epsilon_j):1\leq i<j\leq n,\rht(\alpha)=a-b,j\leq a\},\\
		R_{a,b}^-&:=\{\alpha=-(\epsilon_i-\epsilon_j):1\leq i<j\leq n,\rht(\alpha)=a-b,j\geq 2n-a\},\\
		R_{a,b}&:=R_{a,b}^+\cup R_{a,b}^-\\
		&=\{\alpha=-(\epsilon_i\pm\epsilon_j): \ 1\leq i<j\leq n,\rht(\alpha)=a-b,n\leq \pm(n-j)+a\},
	\end{align*}
	and 
	$$Y_{a,b}=
	\sum_{\alpha\in R_{a,b}} E^\rho(-a\alpha).$$

	\begin{lemma}\label{D:EabXY}
		For $(a,b)\in \mathbb{Z}_{\geq 1}^2$, we have
		$$E_{a,b}=\delta_{(a,b)\in \mathbf{X}}X_{a,b}-\delta_{(a,b)\in \mathbf{X}'}X_{b,a}+\delta_{(a,b)\in \mathbf{Y}}Y_{a,b}-\delta_{(a,b)\in \mathbf{Y}'}Y_{b,a}.$$
	\end{lemma}
	
	\begin{proof}
		Set
		\begin{align*}
			\mathbf{X}_1:=&\{(a,b):n\leq b<2n-1,2a=b\},\\
			\mathbf{X}_2:=&\{(a,b): a<b,2|b,b\geq 2n-2a,b<2n-1,a\neq b,2a\neq b\},\\
			\mathbf{X}_3:=&
			\{(a,b): a<b,2|a,a\geq 2n-2b,a>2b-2n\}.
		\end{align*}
		Notice that $\mathbf{X}\cap \{(a,b):a<b\}$ is a disjoint union of $\mathbf{X}_1$ and $\mathbf{X}_2$, and $\mathbf{X}'\cap \{(a,b):a<b\}$ is $\mathbf{X}_3$.
		
		By \Cref{EabEba}, we may assume that $a<b$. Then it suffices to show that 
		\begin{align}\label{D:EabNegSum}
			E_{a,b}=\delta_{(a,b)\in \mathbf{X}_1}X_{a,b}
			+\delta_{(a,b)\in \mathbf{X}_2}X_{a,b}
			-\delta_{(a,b)\in \mathbf{X}_3}X_{b,a}
			+\delta_{(a,b)\in \mathbf{Y}}Y_{a,b}.
		\end{align}
		All terms in \eqref{Eab} are of the form $E^\rho(-a\alpha)$ with $\rht(\alpha)=a-b<0$, i.e., $\alpha\in\Delta^-$. 
		So \Cref{D:EwithNeg} gives all the cases of non-zero $E^\rho(-a\alpha)$ in $E_{a,b}$. 
		\begin{enumerate}[wide]
			\item $\alpha=-(\epsilon_i\pm\epsilon_n)(1\leq i\leq n-1)$, $a\geq i$ and $a=n-i$.
			
			Since $a-b=\rht(\alpha)=i-n$, we have $b=a+n-i=2a$. The condition for $a$ and $\alpha$ is equivalent to the requirement  that  $(a,b)\in \mathbf{X}_1$. 
			Now 
			$$E^\rho(-a\alpha)=(-1)^nE([2a,n-1,\dots,2,\pm 1]).$$ 
			So this case corresponds to the summand $\delta_{(a,b)\in \mathbf{X}_1}X_{a,b}$ in \eqref{D:EabNegSum}.
			
			\item  $\alpha=-(\epsilon_i\pm\epsilon_j)(1\leq i<j\leq n-1)$, $a\geq i$ and $n-i=\pm(n-j)+a$.
			
			Since $b-a=-\rht(\alpha)=(n-i)\pm(n-j)$, we have $b=2(n-i)$ and $n-j=|a-\frac{b}{2}|$. The condition for $a$ and $\alpha$ is equivalent to  the requirement that $(a,b)\in \mathbf{X}_2$. 
			Now
			$$E^\rho(-a\alpha)=(-1)^{n+a+\frac{b}{2}} E([a+\tfrac{b}{2},n-1,\dots,\widehat{\left|a-\tfrac{b}{2}\right|},\dots, 1,0]).$$
			So this case corresponds to the summand $\delta_{(a,b)\in \mathbf{X}_2}X_{a,b}$ in \eqref{D:EabNegSum}.
			
			\item $\alpha=-(\epsilon_i-\epsilon_j)(1\leq i<j\leq n-1)$, $a\geq i$ and $2n=2j+a$.
			
			We have $n-j=\frac{a}{2}$, and $i=n+\frac{a}{2}-b$ since $a-b=\rht(\alpha)=i-j$. The condition for $a$ and $\alpha$ is equivalent to  the requirement that $(a,b)\in \mathbf{X}_3$. 
			Now
			$$E^\rho(-a\alpha)=-(-1)^{n+b+\tfrac{a}{2}} E([b+\frac{a}{2},n-1,\dots,\widehat{b-\tfrac{a}{2}},\dots, 1,0]).$$
			So this case corresponds to the summand $-\delta_{(a,b)\in \mathbf{X}_3}X_{b,a}$ in \eqref{D:EabNegSum}.
			
			\item $\alpha=-(\epsilon_i\pm\epsilon_j)(1\leq i<j\leq n)$ and $n\leq \pm(n-j)+a$.
			
			Recall that $b-a=-\rht(\alpha)=(n-i)\pm(n-j)$. 
			Then 
			$$a+b=(b-a)+2a\geq (n-i)\pm(n-j)+2(n\mp (n-j))=2n+\rht(\epsilon_i\mp\epsilon_j)>2n$$ 
			and 
			$b-a=-\rht(\alpha)\leq 2n-3$.
			So $(a,b)\in\mathbf{Y}$. Then this case corresponds to the summand $\delta_{(a,b)\in \mathbf{Y}}Y_{a,b}$ in \eqref{D:EabNegSum}.\qedhere
		\end{enumerate}
	\end{proof}
	
	For any $\mathbf{P}\subset \mathbb{Z}_{\geq 1}^2$, define
	\begin{align*}
		X_{\mathbf{P}}:=&\sum_{(a,b)\in \mathbf{P}\cap\mathbf{X}}X_{a,b}-\sum_{(a,b)\in \mathbf{P}\cap\mathbf{X}'}X_{b,a},\\
		Y_{\mathbf{P}}:=&\sum_{(a,b)\in \mathbf{P}\cap\mathbf{Y}}Y_{a,b}-\sum_{(a,b)\in \mathbf{P}\cap\mathbf{Y}'}Y_{b,a}.
	\end{align*}
	\begin{lemma}\label{D:Ezero}
		Let $\mathbf{P}$ be a subset of $\mathbb{Z}_{\geq 1}^2$. Then
		\begin{enumerate}
			\item $E_{\mathbf{P}}=X_{\mathbf{P}}+Y_{\mathbf{P}}$.
			\item $E_{\mathbf{P}}=0$ if and only if $X_{\mathbf{P}}=Y_{\mathbf{P}}=0$.
		\end{enumerate}
	\end{lemma}
	\begin{proof}
		The equality follows from \Cref{D:EabXY}. 
		Note that the second coordinate of the dominant weight of any term in $X_{\mathbf{P}}$ is $n-1$, while that in $Y_{\mathbf{P}}$ is $\geq n$ by \Cref{D:EwithNeg}. So $E_{\mathbf{P}}=0$ if and only if $X_{\mathbf{P}}=Y_{\mathbf{P}}=0$.
	\end{proof}
	
	We shall give an equivalent condition on $\mathbf{P}$ for $Y_{\mathbf{P}}=0$. Here are some properties of $Y_{a,b}$.
	\begin{lemma}\label{D:Yproperty}
		Let $a,b\in\mathbb{Z}_{\geq 1}$ be such that $(a,b)\in \mathbf{Y}$. Then
		\begin{enumerate}
			\item $Y_{a,b}$ is non-zero.
			\item All $Y_{a,b}$'s are linearly independent.
		\end{enumerate}
	\end{lemma}
	
	\begin{proof}
		Any $E^\rho(-a\alpha)$ with $\alpha\in R_{a,b}$ must be as the last case in \Cref{D:EwithNeg}. It can be verified that the dominant weights of any two such $E^\rho(-a\alpha)$'s cannot be the same. This implies that all $E^\rho(-a\alpha)$'s with $\alpha\in R_{a,b}$ are linearly independent.
		
		Now it suffices to show that $R_{a,b}$ is non-empty for any $(a,b)\in \mathbf{Y}$. 
		Set $j=\lfloor(2n+a-b)/2\rfloor+1$ and $i=2n+a-b-j$. 
		Then the root $\alpha=-(\epsilon_i+\epsilon_j)$ is in $R_{a,b}^+$ since $i+j-2n=a-b=\rht(\alpha)$, $1\leq i<j\leq n$ and
		$$j=\lfloor(2n+2a-(a+b))/2\rfloor+1
		=n+a-\lceil (a+b)/2\rceil+1\leq n+a-(n+1)+1=a.$$
		Then the result follows.
	\end{proof}
	Define 
	\begin{align*}
		\mathbf{N}:=&\{(a,b):a+b>2n,|a-b|\leq 2n-3\}.
	\end{align*}

	\begin{lemma}\label{D:Yzero}
		Let $\mathbf{P}$ be a subset of $\mathbb{Z}_{\geq 1}^2$, and $\mathbf{Q}=A(\mathbf{P})$. 
		Then $Y_{\mathbf{P}}=Y_{\mathbf{Q}\cap \mathbf{N}}$, and $Y_{\mathbf{P}}=0$ if and only if $\mathbf{Q}\cap \mathbf{N}=\varnothing$.
	\end{lemma}
	
	\begin{proof}
		Notice  that
		\begin{align*}
			Y_{\mathbf{P}}
			=Y_{\mathbf{P}\cap\mathbf{Y}}-Y_{\mathbf{P}'\cap\mathbf{Y}}
			=Y_{\mathbf{Q}\cap\mathbf{Y}}-Y_{\mathbf{Q}'\cap\mathbf{Y}}
			=Y_{\mathbf{Q}\cap\mathbf{Y}}+Y_{\mathbf{Q}\cap\mathbf{Y}'}
			=Y_{\mathbf{Q}\cap \mathbf{N}},
		\end{align*}
		since $\mathbf{N}$ is the disjoint union of $\mathbf{Y},\mathbf{Y}'$ and $\{(a,a):a>n\}$.
		
		Since all $Y_{a,b}$ with $(a,b)\in\mathbf{Y}$ are linearly independent by \Cref{D:Yproperty}, the result follows from $(\mathbf{Q}\cap \mathbf{N})\cap (\mathbf{Q}\cap \mathbf{N})'=\varnothing$.
	\end{proof}

	\subsection{$M_{p,D}$ for type $\mathsf{D}$}
	
	Recall that $M_{p,D}=E_{\mathbf{P}}$ when $\mathbf{P}=\mathbf{P}_{p,D}$.
	Define the following sets:
	\begin{align*}
		{\mathbf{R}}_{p,D}:=&\{(a,b):2ab=pD,p|2b,p\nmid b\};\\
		{\mathbf{U}}_{p,D}:=&\{(a,b):2ab=pD,p|2b,p\nmid a\};\\
		{\mathbf{V}}_{p,D}:=&\{(a,b):2ab=pD,p|a,p\nmid 2b\};\\
		\mathbf{C}_0:=&\{(a,b):a+b\geq n,a\leq n-1,b\leq n-1\};\\
		\mathbf{C}_1:=&\{(a,b):a\geq n,b\leq n-1,a<b+n\}.
	\end{align*}
	
	Define 
	$$\widetilde{X}_t=\begin{cases}
		(-1)^{n}(E([t,n-1,\dots,2,1])+E[t,n-1,\dots,2,-1]),&t^2=2pD,\\
		(-1)^{n+t} E([t,n-1,n-2,\dots,
		\widehat{\sqrt{t^2-2pD}},\dots, 1,0]),&t^2\neq 2pD
	\end{cases}$$
	for $t$ satisfying that $t\geq n$ and $t^2-2pD\in\{0,1^2,2^2,\dots,(n-1)^2\}$. Notice that all such $\widetilde{X}_t$'s are linearly independent.
	
	For $(a,b)\in \mathbf{P}_{p,D}\cap\mathbf{X}$, we have $X_{a,b}=\widetilde{X}_{a+\frac{b}{2}}$, since
	$$|a-\tfrac{b}{2}|=\sqrt{(a+\tfrac{b}{2})^2-2ab}=\sqrt{(a+\tfrac{b}{2})^2-2pD}.$$ 
	Now $X_{\mathbf{P}}$ can be written as
	\begin{align}\label{D:Xalt}
		X_{\mathbf{P}}=\sum_{(a,b)\in \mathbf{P}\cap\mathbf{X}}\widetilde{X}_{a+\frac{b}{2}}
		-\sum_{(a,b)\in\mathbf{P}\cap \mathbf{X}'}\widetilde{X}_{b+\frac{a}{2}}.
	\end{align}
	
	\begin{lemma}\label{D:Xzero}
		Set $\mathbf{P}=\mathbf{P}_{p,D}$.
		Then 
		$$X_{\mathbf{P}}=\widetilde{X}_{\mathbf{R}_{p,D}\cap \mathbf{C}_0}
		+\widetilde{X}_{\mathbf{U}_{p,D}\cap \mathbf{C}_1}-\widetilde{X}_{\mathbf{V}_{p,D}\cap \mathbf{C}_1},$$
		where
		$$\widetilde{X}_{\mathbf{A}}:=\sum_{(a,b)\in\mathbf{A}} \widetilde{X}_{a+b}.$$
		Moreover, the sum $X_{\mathbf{P}}=0$
		if and only if $\mathbf{R}_{p,D}\cap \mathbf{C}_0=\mathbf{U}_{p,D}\cap\mathbf{C}_1=\mathbf{V}_{p,D}\cap\mathbf{C}_1=\varnothing$.
	\end{lemma}
	\begin{proof} We shall simply use the notation $\mathbf{P}$ (resp. $\mathbf{Q}$, etc.) instead of $\mathbf{P}_{p,D}$ (resp. $\mathbf{Q}_{p,D}$, etc.) if there is no ambiguity. 
		Also while deducing the formula of $X_{\mathbf{P}}$, we would regard all $\widetilde{X}_t$'s indexed by $t\in\mathbb{R}$ as linearly independent formal vectors.
		
		Define 
		$$\overline{\mathbf{X}}:=\{(a,b):   2|b,b\geq 2n-2a,b>2a-2n,b<2n-1\}.$$ Notice that
		$$\overline{\mathbf{X}}\setminus\mathbf{X}=\{(a,a):  2|a,3a\geq 2n,a<2n-1\}=\overline{\mathbf{X}}'\setminus\mathbf{X}'.$$
		Then from \eqref{D:Xalt}, we have
		\begin{align*}
			X_{\mathbf{P}}=\sum_{(a,b)\in \mathbf{P}\cap\overline{\mathbf{X}}} \widetilde{X}_{a+\frac{b}{2}}
			-\sum_{(a,b)\in \mathbf{P}\cap\overline{\mathbf{X}}'}\widetilde{X}_{b+\frac{a}{2}}
			=\sum_{(a,b)\in \mathbf{P}\cap\overline{\mathbf{X}}} \widetilde{X}_{a+\frac{b}{2}}
			-\sum_{(a,b)\in \mathbf{P}'\cap\overline{\mathbf{X}}}\widetilde{X}_{a+\frac{b}{2}}.
		\end{align*}
		
		Define a mapping $C$ which maps a set of pairs $\mathbf{A}\subset  \mathbb{Z}_{\geq 1}^2$ to
		$$C(\mathbf{A}):=\{(a,b)\in\mathbb{Z}_{\geq 1}\times\mathbb{Z}_{\geq 1}:(a,2b)\in \mathbf{A}\}.$$
		Then
		\begin{align*}
			C(\overline{\mathbf{X}})&=\{(a,b):b\geq n-a,b>a-n,b\leq n-1\},\\
			C(\mathbf{P})&=\{(a,b):2ab=pD,p|2b\},\\
			C(\mathbf{P}')&=\{(a,b):2ab=pD,p|a\}.
		\end{align*}
		Since $2|b$ for all $(a,b)\in\overline{\mathbf{X}}$,  the sum $X_{\mathbf{P}}$ can be written as
		$$X_{\mathbf{P}}=\sum_{(a,b)\in C(\mathbf{P}\cap\overline{\mathbf{X}})} \widetilde{X}_{a+b}
		-\sum_{(a,b)\in C(\mathbf{P}'\cap\overline{\mathbf{X}})}\widetilde{X}_{a+b}
		=\widetilde{X}_{C(\mathbf{P}\cap\overline{\mathbf{X}})}
		-\widetilde{X}_{C(\mathbf{P}'\cap\overline{\mathbf{X}})}.$$
		
		Note that $\mathbf{C}_0=C(\overline{\mathbf{X}})\cap C(\overline{\mathbf{X}})'$ is symmetric, and that 
		$C(\overline{\mathbf{X}})$ is the disjoint union of $\mathbf{C}_0$ and $\mathbf{C}_1$, so
		\begin{align*}
			X_{\mathbf{P}}
			=&\widetilde{X}_{C(\mathbf{P})\cap \mathbf{C}_0}
			+\widetilde{X}_{C(\mathbf{P})\cap \mathbf{C}_1}
			-\widetilde{X}_{C(\mathbf{P}')\cap \mathbf{C}_0}
			-\widetilde{X}_{C(\mathbf{P}')\cap \mathbf{C}_1}\\
			=&(\widetilde{X}_{C(\mathbf{P})\cap \mathbf{C}_0}-\widetilde{X}_{C(\mathbf{P}')'\cap \mathbf{C}_0})
			+\widetilde{X}_{C(\mathbf{P})\cap \mathbf{C}_1}-\widetilde{X}_{C(\mathbf{P}')\cap \mathbf{C}_1}.
		\end{align*}
		Then 
		\begin{align*}
			X_{\mathbf{P}}=&\widetilde{X}_{\mathbf{R}\cap \mathbf{C}_0}
			+\widetilde{X}_{C(\mathbf{P})\cap \mathbf{C}_1}-\widetilde{X}_{C(\mathbf{P}')\cap \mathbf{C}_1}\\
			=&\widetilde{X}_{\mathbf{R}\cap \mathbf{C}_0}
			+\widetilde{X}_{\mathbf{U}\cap \mathbf{C}_1}-\widetilde{X}_{\mathbf{V}\cap \mathbf{C}_1},
		\end{align*}
		since $C(\mathbf{P}')'\subset C(\mathbf{P})$,  $\mathbf{R}=C(\mathbf{P})\setminus C(\mathbf{P}')'$, $\mathbf{U}=C(\mathbf{P})\setminus C(\mathbf{P}')$ and $\mathbf{V}=C(\mathbf{P}')\setminus C(\mathbf{P})$. 
		
		Recall that all $\widetilde{X}_t$'s are linearly independent, and we have the following fact: for $(a_1,b_1)$, $(a_2,b_2)\in \{(a,b): 2ab=pD\}$, the sum $a_1+b_1=a_2+b_2$ if and only if $(a_1,b_1)=(a_2,b_2)$ or $(a_1,b_1)=(b_2,a_2)$. Then the equivalent condition for $X_{\mathbf{P}}=0$ follows from that $\mathbf{C}_0,\mathbf{C}_1,\mathbf{C}_1'$ are disjoint.
	\end{proof}

	Define
	\begin{align*}
		\mathbf{Q}_{p,D}:=&A(\mathbf{P}_{p,D})=\{(a,b):  ab=pD,p|b,p\nmid a\}.
	\end{align*}

	Combining \Cref{D:Ezero,D:Xzero,D:Yzero}, we obtain the following result on $M_{p,D}$.
	\begin{lemma}\label{D:criterion4M}
		Set $\mathbf{P}=\mathbf{P}_{p,D}$.
		Then 
		\begin{align}
			M_{p,D}=\widetilde{X}_{\mathbf{R}_{p,D}\cap \mathbf{C}_0}
			+\widetilde{X}_{\mathbf{U}_{p,D}\cap \mathbf{C}_1}-\widetilde{X}_{\mathbf{V}_{p,D}\cap \mathbf{C}_1}
			+Y_{\mathbf{Q}_{p,D}\cap \mathbf{N}}.
		\end{align}
		Furthermore,  $M_{p,D}=0$ if and only if 
		$$\mathbf{R}_{p,D}\cap \mathbf{C}_0=\mathbf{U}_{p,D}\cap\mathbf{C}_1=\mathbf{V}_{p,D}\cap\mathbf{C}_1=\mathbf{Q}_{p,D}\cap \mathbf{N}=\varnothing.$$
	\end{lemma}

	\subsection{Values of $D_p$ and corresponding weights}
	
	For simplicity, we shall simply use the notation $\mathbf{P}$ (resp. $\mathbf{Q}$, etc.) instead of $\mathbf{P}_{p,D}$ (resp. $\mathbf{Q}_{p,D}$, etc.) if there is no ambiguity. 
	
	\begin{lemma}\label{D:Dp}
		For $n\in\mathbb{Z}_{\geq 4},p\in\mathbb{Z}_{\geq 2}$ with $p<2n-2$, we have the following results on $D_p$:
		\begin{enumerate}[wide] 
			\item When $p=3,3|n-1$, $D_p\geq 2n-1$;
			\item When $p=5,n=7$, $D_p\geq 11$;
			\item In all other cases, set 
			$$s_1=\lfloor (2n-1)/p\rfloor,~ s_2=\lceil (2n-1)/p\rceil.$$
			Then $D_p\geq \min\{D_{(0)},D_{(1)}(s_1),D_{(2)}(s_2)\}$, where 
			$$D_{(0)}=
			\begin{cases}
				n-\frac{p}{2},& 2|p,\\
				+\infty,& 2\nmid p,
			\end{cases}$$
			$$D_{(1)}(s):=
			\begin{cases}
				s(n-sp/2),& 2\nmid s(p-1),\\
				s(2n-sp+1),& 2|s(p-1),
			\end{cases}$$
			and
			$$D_{(2)}(s):=
			\begin{cases}
				s(sp-2n+3),& 2\nmid s,\\
				s(sp/2-n+1),& 2| s.
			\end{cases}$$
		\end{enumerate}
	\end{lemma}
	
	\begin{proof}
		By \Cref{D:criterion4M}, it suffices to show that  $\mathbf{R}\cap \mathbf{C}_0=\mathbf{U}\cap\mathbf{C}_1=\mathbf{V}\cap\mathbf{C}_1=\mathbf{Q}\cap \mathbf{N}=\varnothing$ for $D$ less than the given bound.  
		\begin{enumerate}[wide]
			\item Assume that $D\leq 2n-2$. 
			\begin{enumerate}
				\item 
				The set $\mathbf{R}=\varnothing$,  since $2\nmid p$.
				
				\item 
				For $(a,b)\in \mathbf{U}\cap \mathbf{C}_1$, we have $a\geq n$ and $2b\geq 2p=6$ since $p=3$ is odd. Hence $2ab\geq 6n> 6(n-1)\geq pD$, which contradicts the fact that $2ab=pD$.  
				
				\item 
				For $(a,b)\in \mathbf{V}\cap \mathbf{C}_1$, we have $n\leq a<b+n$ and $3|a$. It follows that $a\geq n+2$, since $3|n-1$. Thus 
				$2ab\geq 2a(a-n+1)\geq 6(n+2)>6(n-1)\geq pD$, which also contradicts $2ab=pD$.
				
				\item 
				For $(a,b)\in \mathbf{Q}\cap \mathbf{N}$, we have $a>2n-b$. Note that $a\neq 1$ or $2$; otherwise, $b-a>2n-3$. Thus by $3|b$, $3|2n-2$ and $b= pD/a\leq D\leq 2n-2$, we have
				$$ab\geq b(2n-b+1)\geq 3(2n-2)\geq pD,$$
				where the equality holds if and only if $D=2n-2$ and $(a,b)=(3,2n-2)$ or $(2n-2,3)$.
				However, this contradicts $3\nmid a$.
				So $\mathbf{Q}\cap \mathbf{N}=\varnothing$.
			\end{enumerate}
			
			\item Assume that $D\leq 10$. 
			\begin{enumerate}
				\item 
				We have $\mathbf{R}=\varnothing$ since $2\nmid p$. It follows that $\mathbf{R}\cap\mathbf{C}_0=\varnothing$.
				
				\item Note that
				$$\mathbf{C}_1\cap \{(a,b):p|2ab\}=\{(10,4),(10,5),(10,6),(7,5),(8,5),(9,5),(11,5)\}.$$
				Then the condition $\mathbf{U}\cap\mathbf{C}_1=\mathbf{V}\cap\mathbf{C}_1=\varnothing$ holds,  since $2ab=pD\leq 50$ for $(a,b)\in \mathbf{U}\cup\mathbf{V}$. 
				
				\item For $(a,b)\in \mathbf{Q}\cap\mathbf{N}$, we have $ab=pD\leq 50$.
				So
				$$(a,b)\in \mathbf{N}\cap \{(a,b):ab\leq 50,p|b\}=\{(10,5),(5,10)\}.$$
				Since $p\nmid a$, the condition $\mathbf{Q}\cap\mathbf{N}=\varnothing$ follows.
			\end{enumerate}
			
			\item Assume that $D<\min\{D_{(0)},D_{(1)}(s_1),D_{(2)}(s_2)\}$. 
			
			Apply \Cref{AppenH} to the case with $N=2n-1$. 
			Note that $G(s)<F(s)$ for all $s$, so we have
			$$H_{\min}=\min\{D_{(0)},D_{(1)}(s_1),D_{(2)}(s_2)\},$$
			and hence $D<H_{\min}$. 
			\begin{enumerate}
				\item 
				Assume that  $(a,b)\in\mathbf{R}\cap\mathbf{C}_0$. 
				If $2\nmid p$, the set $\mathbf{R}$ is empty which contradicts. So
				$2|p$ and we have
				\begin{align*}
					2ab&\geq \min\{2ab:p|2b,p\nmid b,1\leq b\leq n-1,a=n-b\}\\
					&=\min\{pG(s):s\leq s_1,2\nmid s\}\geq pH_{\min}>pD.
				\end{align*}
				
			\item 	Assume that  $(a,b)\in\mathbf{U}\cap\mathbf{C}_1$. 
				If $2\nmid p$, then $2ab\geq 2pn\geq pN\geq pH_{\min}>pD$; 
				if $2|p$, then $2ab\geq pn\geq p(N-1)/2\geq pH_{\min}>pD$.
				
			\item	Assume that  $(a,b)\in\mathbf{V}\cap\mathbf{C}_1$. Then
				\begin{align*}
					2ab&\geq \min\{2ab:p|a,a\geq n,b=a-n+1\}\\
					&=\min\{pG(s):s\geq s_2,2|s\}\geq pH_{\min}>pD.
				\end{align*}
				
				These all contradict the equality $2ab=pD$. 
				
				\item Assume that  $(a,b)\in\mathbf{Q}\cap\mathbf{N}$. Then
				$$ab \geq \min\{ab:p|b,a=|2n-1-b|+2\}
				= \min\{pF(s)\}\geq pH_{\min}>pD,$$
				which is absurd. \qedhere
			\end{enumerate}
		\end{enumerate}
	\end{proof}

	\begin{proof}[Proof of \Cref{main2}]	
		By \Cref{thmNk} and \Cref{D:Dp}, it suffices to show that the sum $M_{p,D}$ with $D=D_p$ given is nonzero.
		Then the weights of minimal singular vectors are given by those dominant $\lambda_i$'s in the formula \eqref{M_pD2E}. 
		Recall that the dominant weight in $E^\rho(\cdot)$ can be calculated by \Cref{D:dominant_in_orbit}. 
		
		\begin{enumerate}[wide]
			\item This case follows from \Cref{integrable_case}.
			
			\item Set $D=2n-1$. Recall the formula of $M_{p,D}$ given in \Cref{D:criterion4M}.
			Since $2\nmid p$, 	the set $\mathbf{R}\cap \mathbf{C}_0=\varnothing$. 
			
			For $(a,b)\in \mathbf{U}\cap \mathbf{C}_1$, we have $a\geq n$ and $2b\geq 2p=6$ since $p=3$ is odd, and then $2ab\geq 6n> 3(2n-1)= pD$, which contradicts  $2ab=pD$.  
			For $(a,b)\in \mathbf{V}\cap \mathbf{C}_1$, we have $a\geq n$ and $3|a$. It follows that $a\geq n+2$ since $3|n-1$. And thus 
			$2ab\geq 2a(a-n+1)\geq 6(n+2)>3(2n-1)= pD$, which also contradicts $2ab=pD$.
			So $\mathbf{U}\cap \mathbf{C}_1=\mathbf{V}\cap \mathbf{C}_1=\varnothing$.

			For $(a,b)\in \mathbf{Q}\cap \mathbf{N}$, we have $a>2n-b$. Note that $a\neq 1$ or $2$; otherwise, $b-a>2n-3$. So $b= pD/a\leq D= 2n-1$. 
		By assumption, $3|b$ and $3|2n-2$. It follows that $3\leq b\leq 2n-2$. 
			If $3<b<2n-2$, then
			$ab\geq b(2n-b+1)\geq 6(2n-5)>pD$;
			if $b=2n-2$, then $a=pD/b$ is not an integer;
			if $b=3$, then $a=D$ satisfies all conditions.
			So $\mathbf{Q}\cap \mathbf{N}=\{(D,3)\}$.

			In summary, we have 
			$$M_{p,D}=Y_{\{(D,3)\}}=-Y_{3,D}=E^\rho(-D(\epsilon_1+\epsilon_3)),$$
			since
			\begin{align*}
				R_{3,D}=&\{\alpha=-(\epsilon_i\pm\epsilon_j):1\leq i<j\leq n,\rht(\alpha)=-(2n-4),n\leq \pm(n-j)+3\}\\
				=&\{-(\epsilon_1+\epsilon_3)\}.
			\end{align*}
			
			\item Set $D=11$. We have
			\begin{align*}
				M_{p,D}=&E_{11,5}
				=\sum_{\alpha\in\Delta,\rht(\alpha)=6}E^\rho(-11\alpha)\\
				=&E^\rho(-11(\epsilon_1-\epsilon_7))+\sum_{i=1}^3E^\rho(-11(\epsilon_i+\epsilon_{8-i}))\\
				=&E^\rho(-11(\epsilon_3+\epsilon_5))=E[9,7,6,5,3,1,0]=E^\rho[3,2,2,2,1,0,0]
			\end{align*}
			by \Cref{D:dominant_in_orbit} and \Cref{D:identical_elimination}.
			
			\item Set $D=20$. Then
			$$M_{p,D}=E_{20,5}+E_{10,10}+E_{5,20}+E_{4,25}+E_{2,50}+E_{1,100}.$$
			It is clear that $E_{10,10}=0$ and $E_{20,5}+E_{5,20}=0$ by \Cref{EabEba}. Also note that $E_{2,50}=E_{1,100}=0$, since $|b-a|>2n-3=21$.
			So 
			$$M_{p,D}=E_{4,25}
			=\sum_{\alpha\in\Delta,\rht(\alpha)=-21}E^\rho(-4\alpha)
			=E^\rho(4(\epsilon_1+\epsilon_2)).$$
			
			\item Set $D=2$. Then $M_{p,D}=E_{2,4}+E_{1,8}$. 
			Note that $E_{1,8}=0$ since $|b-a|>2n-3=5$. So
			\begin{align*}
				M_{p,D}=&E_{2,4}
				=\sum_{\alpha\in\Delta,\rht(\alpha)=-2}E^\rho(-2\alpha)\\
				=&E^\rho(2(\epsilon_1-\epsilon_3))+E^\rho(2(\epsilon_2-\epsilon_4))+E^\rho(2(\epsilon_2+\epsilon_4)).
			\end{align*}
			
			\item Set $D=4$.
			Then $M_{p,D}=E_{4,5}+E_{2,10}+E_{1,20}$. 
			Note that $E_{2,10}=E_{1,20}=0$ since $|b-a|>2n-3=5$. So
			\begin{align*}
				M_{p,D}=&E_{4,5}
				=\sum_{\alpha\in\Delta,\rht(\alpha)=-1}E^\rho(-4\alpha)\\
				=&E^\rho(4(\epsilon_1-\epsilon_2))+E^\rho(4(\epsilon_3-\epsilon_4))+E^\rho(4(\epsilon_3+\epsilon_4)).
			\end{align*}
			
			\item Use the notations in \Cref{AppenH} with $N=2n-1$.
			\begin{enumerate}
				\item If $2\nmid p$, then $\mathbf{R}_{p,D}\cap \mathbf{C}_0= \varnothing$ for any $D$. 
				
				If $2| p$, the minimum of $D$ such that $\mathbf{R}_{p,D}\cap \mathbf{C}_0\neq \varnothing$ is 
				\begin{align*}
					&\min\{2ab/p:p|2b,p\nmid b,(a,b)\in \mathbf{C}_0\}\\
					=&\min\{2ab/p:p|2b,p\nmid b,1\leq b\leq n-1,a=n-b\}\\
					=&\min\{2(n-b)b/p:p|2b,p\nmid b,1\leq b\leq n-1\}\\
					=&\min\{(n-ps/2)s:s\leq s_1,2\nmid s\}
					=\min\{G(s):s\leq s_1,2\nmid s\}.
				\end{align*}
				
				\item The minimum of $D$ such that $\mathbf{U}_{p,D}\cap \mathbf{C}_1\neq \varnothing$ is 
				\begin{align*}
					&\min\{2ab/p:p|2b,p\nmid a,(a,b)\in \mathbf{C}_1\}\\
					=&\min\{2ab/p:p|2b,p\nmid a,1\leq b\leq n-1,n\leq a\leq b+n-1\}.
				\end{align*}
				If $2\nmid p$, then this minimum is bigger than $ 2n-1=N$; if $2| p$, then this minimum is bigger than $n-1=(N-1)/2$.
				
				\item 
				The minimum of $D$ such that $\mathbf{V}_{p,D}\cap \mathbf{C}_1\neq \varnothing$ is 
				\begin{align*}
					&\min\{2ab/p:p|a,p\nmid 2b,(a,b)\in \mathbf{C}_1\}\\
					=&\min\{2ab/p:p|a,p\nmid 2b,n\leq a\leq 2n-2,a-n+1\leq b\leq n-1\}.
				\end{align*}
				
				If $2| s_2$ and $G(s_2)\leq G(s_1)$, then the minimum is attained when $a=s_2p/2$ and $b=a-n+1$, and the minimum is $s_2(s_2p/2-n+1)=G(s_2)$. Indeed, if $p|2b$, i.e., $p|2n-2$, then $ps_1=2n-2$. It follows that $s_1=s_2-1$ is odd and $p$ is even, so $G(s_1)=s_1< G(s_2)$, which is absurd.
				
				If $2| s_2$ and $G(s_2)> G(s_1)$, then the minimum is bigger than $G(s_2)-1\geq G(s_1)$.
				
				If $2\nmid s_2$, then the minimum is bigger than $G(s_2+1)-1$. 
				
				\item 
				Let $(a,b)\in\mathbf{Q}_{p,D}\cap\mathbf{N}$ and write
				$b=ps$. Then 
				$$
				D=ab/p=as
				\geq (|N-ps|+2)s
				=F(s)\geq H(s)\geq H_{\min}.
				$$
								This shows that no element of
				$\mathbf{Q}_{p,D}\cap\mathbf{N}$ can occur for
				$D<H_{\min}$.
				
				Suppose that $D=H_{\min}$. Then all the inequalities
				above must be equalities. Hence
				\[
				a=|N-ps|+2,\qquad
				F(s)=H(s)=H_{\min}.
				\]
				By \Cref{AppenH}(1), we have
				$s\in\{1,s_1,s_2\}$. If $s_1>1$, then
				\Cref{AppenH}(3), together with the exceptional cases
				already treated above, gives $F(1)>H_{\min}$.
				If $s_1=1$, then $1=s_1$. Therefore, it remains only
				to consider $s=s_1$ and $s=s_2$.
				
				For $i=1,2$, set $a_i=|N-ps_i|+2, b_i=ps_i$.
				We claim that $p\nmid a_i$ whenever $F(s_i)=H_{\min}$.
				If $s_1=s_2$, then $p\mid N$. Since $N$ is odd,
				$p$ is odd, while $a_1=a_2=2$, so the claim is clear.
				We may therefore assume that $s_2=s_1+1$.
				
				First assume that $H_{\min}=F(s_1)$ and $p\mid a_1$.
				Then $F(s_1)\leq F(s_2)$. Since
				\[
				a_1=N-ps_1+2=2n+1-ps_1
				\]
				and $3\leq a_1\leq p+1$, we must have $a_1=p$.
				Thus $(s_1+1)p=s_2p=2n+1$.
				In particular, $p$ and $s_2$ are odd, and $s_1$ is
				even. Moreover,
				\[
				F(s_1)-F(s_2)
				=s_1p-4s_2
				=(p-4)s_1-4.
				\]
				If $p=3$, then $3\mid n-1$; if $p=5$ and
				$s_1=2$ or $4$, then $n=7$ or $12$, respectively.
				These cases have already been treated separately.
				In all the remaining cases, either $p\geq7$, or
				$p=5$ and $s_1\geq6$, and hence
				$F(s_1)-F(s_2)>0$. This contradicts
				$F(s_1)\leq F(s_2)$.
				
				Next assume that $H_{\min}=F(s_2)$ and $p\mid a_2$.
				Since $a_2=ps_2-N+2=ps_2-2n+3$
				and $3\leq a_2\leq p+1$, we must have $a_2=p$.
				Therefore $s_1p=2n-3$.
				It follows that $p$ and $s_1$ are odd, so $s_2$ is
				even. By the definition of $H$, we then have
				\[
				H(s_2)=G(s_2)<F(s_2)=H_{\min},
				\]
				which is a contradiction.
				
				Finally, if $F(s_i)=H_{\min}$, then
				$b_i=ps_i\geq2$ and hence $(a_i,b_i)\in\mathbf{N}$.
				Moreover, $a_ib_i=pF(s_i)=pH_{\min},
				p\mid b_i, p\nmid a_i$.
				Thus
				\[
				\mathbf{Q}_{p,H_{\min}}\cap\mathbf{N}
				=
				\left\{
				\bigl(|N-ps_i|+2,ps_i\bigr):
				i\in\{1,2\},\ F(s_i)=H_{\min}
				\right\}.
				\]
				
				Assume that  $p|s_2p-2n+3$ and $F(s_1)\geq F(s_2)$. 
				Then $s_1p=2n-3$. It follows that
				$$F(s_2)-F(s_1)=ps_2-4s_1=(p-4)(s_1+1)+4>0,$$
				which also leads to a contradiction.
				
			\end{enumerate}
			
			In summary, we have obtained that the minimum of $D$ such that $M_{p,D}\neq 0$ is $H_{\min}$ by \Cref{D:criterion4M}.
			Then $H_{\min}=\min\{D_{(0)},D_{(1)}(s_1),D_{(2)}(s_2)\}$ by  \Cref{AppenH}. So $D_p=\min\{D_{(0)},D_{(1)}(s_1),D_{(2)}(s_2)\}$.

			Set $D=D_p$. As shown above, we have:
			\begin{enumerate}
				\item The set $\mathbf{R}\cap \mathbf{C}_0\neq \varnothing$ if and only if one of the following cases occurs: 
				\begin{enumerate}
					\item $2|p$ and $D_p=D_{(0)}=G(1)$;
					\item $2|p$, $2\nmid s_1$ and $D_p=D_{(1)}(s_1)=G(s_1)$.
				\end{enumerate}
				
				If only the first case occurs, the set
				$\mathbf{R}\cap \mathbf{C}_0=\{(n-\frac{p}{2},\frac{p}{2})\}$. 
				Then 
				\begin{align*}
					\widetilde{X}_{\mathbf{R}\cap \mathbf{C}_0}
					=&\widetilde{X}_{n}\\
					=&E([n,n-1,n-2,\dots,\widehat{|p-n|},\dots,1,0])\\
					=&
					\begin{cases}
						E^\rho(\epsilon_1+\epsilon_2+\cdots+\epsilon_{p}),& p<n,\\
						E^\rho(\epsilon_1+\epsilon_2+\cdots+\epsilon_{2n-p}), &p>n.
					\end{cases}
				\end{align*}
				Note that $p\neq n$. Otherwise, we have $s_2=2$ and hence $D_{(2)}(s_2)=2<p/2=D_{(0)}$, which contradicts.
				
				If only the second case occurs, the set
				$\mathbf{R}\cap \mathbf{C}_0=\{(n-\frac{s_1p}{2},\frac{s_1p}{2})\}$. 
				Then 
				\begin{align*}
					\widetilde{X}_{\mathbf{R}\cap \mathbf{C}_0}
					=&\widetilde{X}_{n}\\
					=&E([n,n-1,n-2,\dots,\widehat{s_1p-n},\dots,1,0])\\
					=&E^\rho(\epsilon_1+\epsilon_2+\cdots+\epsilon_{2n-s_1p}).
				\end{align*}
				
				If both cases happen, then the set
				$\mathbf{R}\cap \mathbf{C}_0=\{(n-\frac{p}{2},\frac{p}{2})\}\cup \{(n-\frac{s_1p}{2},\frac{s_1p}{2})\}$. 
				
				\item The set $\mathbf{U}\cap \mathbf{C}_1= \varnothing$ when $D=D_p$.
				
				\item The set $\mathbf{V}\cap \mathbf{C}_1\neq  \varnothing$ if and only if $2|s_2$ and $D_p=D_{(2)}=G(s_2)$. 
				In this case, the set $\mathbf{V}\cap \mathbf{C}_1=\{(\frac{s_2p}{2},\frac{s_2p}{2}-n+1)\}$.
				Then 
				\begin{align*}
					\widetilde{X}_{\mathbf{V}\cap \mathbf{C}_1}
					=&\widetilde{X}_{s_2p-n+1}\\
					=&-E([s_2p-n+1,n-2,\dots,1,0])\\
					=&-E^\rho((s_2p-2n+2)\epsilon_1).
				\end{align*}
				
				\item The set $\mathbf{Q}\cap \mathbf{N}\neq \varnothing$ if and only if at least  one of the following cases occurs.
				\begin{enumerate}
					\item $2| s_1(p-1)$ and $D_p=D_{(1)}=F(s_1)$;
					\item $2\nmid s_2$ and $D_p=D_{(2)}=F(s_2)$.
				\end{enumerate}
				
				If only the first case occurs, the set
				$\mathbf{Q}\cap \mathbf{N}=\{(2n+1-ps_1,ps_1)\}$. 
				Then 
				\begin{align*}
					Y_{\mathbf{Q}\cap\mathbf{N}}
					&=Y_{2n+1-ps_1,ps_1}\\
					&=E^\rho\bigl(
					(2n+1-ps_1)
					(\epsilon_{2n-ps_1}+\epsilon_{2n-ps_1+1})
					\bigr)\\
					&=E^\rho\bigl(
					2(\epsilon_1+\cdots+\epsilon_{2n-ps_1+1})
					\bigr).
				\end{align*}
				
				If only the second case occurs, the set
				$\mathbf{Q}\cap \mathbf{N}=\{(ps_2-2n+3,ps_2)\}$. 
				Then 
				\begin{align*}
					Y_{\mathbf{Q}\cap\mathbf{N}}
					&=Y_{ps_2-2n+3,ps_2}\\
					&=E^\rho\bigl(
					(ps_2-2n+3)(\epsilon_1+\epsilon_2)
					\bigr).
				\end{align*}
				
				If both cases happen, then the set
				$\mathbf{Q}\cap \mathbf{N}$ is the union of two cases. 
			\end{enumerate}
		\end{enumerate}
		Finally, we focus  on the number of linearly independent minimal singular vectors. From the discussion above, there exist some coefficient $a_i$ in $M_{p,D_p}$ \eqref{M_pD2E} such that $a_i\geq 2$ if and only if $p=2$ and $2|n$. In this case, we have $D_p=D_{(0)}=D_{(1)}=n-1$, and
		$M_{p,D_p}=2E^\rho(\epsilon_1+\epsilon_2)$.
	\end{proof}

	\subsection{Type $\mathsf{E}$}
	
	From \cite{Humphreys} (see also \cite{tauvel2005lie}), the root system of type $\mathsf{E}$ can be realized in $\mathbb{R}^8$ as follows: 
	\begin{itemize}
		\item $\mathsf{E}_6$: 
		The simple  roots are 
		\begin{align*}
			&\alpha_1=\tfrac12 (\epsilon_1 - \epsilon_2 - \epsilon_3 - \epsilon_4 - \epsilon_5 - \epsilon_6 - \epsilon_7 + \epsilon_8),  ~ \alpha_2=\epsilon_1 + \epsilon_2,\\
			&\alpha_3= \epsilon_2 - \epsilon_1, ~\alpha_4= \epsilon_3 - \epsilon_2, ~ \alpha_5=\epsilon_4 - \epsilon_3, ~ \alpha_6=\epsilon_5 - \epsilon_4.
		\end{align*}
		The highest root is $\theta=\alpha_1+2\alpha_2+2\alpha_3+3\alpha_4+2\alpha_5+\alpha_6$.
		
		\item $\mathsf{E}_7$: 
		The simple roots are those of $\mathsf{E}_6$ together with $\alpha_7=\epsilon_6 - \epsilon_5$.
		
		The highest root is $\theta=2\alpha_1+2\alpha_2+3\alpha_3+4\alpha_4+3\alpha_5+2\alpha_6+\alpha_7$.
		
		\item $\mathsf{E}_8$: 
		The simple roots are those of $\mathsf{E}_7$ together with $\alpha_8=\epsilon_7 - \epsilon_6$.
		
		The highest root is $\theta=2\alpha_1+3\alpha_2+4\alpha_3+6\alpha_4+5\alpha_5+4\alpha_6+3\alpha_7+2\alpha_8$.
	\end{itemize}
	
	Recall that $\mathbf{h}^\vee=12$ for type $\mathsf{E}_6$, $\mathbf{h}^\vee=18$ for type $\mathsf{E}_7$, and  $\mathbf{h}^\vee=30$ for type $\mathsf{E}_8$.
	The proof of \Cref{singE} is omitted since it is similar to that for type $\mathsf{A}$ or $\mathsf{D}$ and it suffices to calculate $D_p$ and  $\lambda_{sing}$ 
	case by case. We also verify the data in \Cref{tab:Esing} by an exact computer calculation. 
	
\section{Associated Weyl elements  in Kashiwara-Tanisaki character formulas}
	
	\subsection{Kashiwara-Tanisaki character formulas}
	
	For $\lambda\in\widehat{\kh}^*$, set 
	$$
	\widehat{\Delta}(\lambda)=\{\alpha\in\widehat{\Delta}^{re}:(\lambda+\widehat{\rho}|\alpha^{\vee})\in\mathbb{Z}\}, 
	$$
	$$
	\widehat{\Delta}_0(\lambda)=\{\alpha\in\widehat{\Delta}^{re}:(\lambda+\widehat{\rho}|\alpha^{\vee})=0\}.
	$$
	Notice that $\widehat{\Delta}(\lambda)$ and $\widehat{\Delta}_0(\lambda)$ are subsystems of $\widehat{\Delta}^{re}$. Denote the set  of positive roots, the set of negative roots, the set of simple roots, the set of simple reflections  and the Weyl group for  $\widehat{\Delta}(\lambda)$ by
	$\widehat{\Delta}^+(\lambda)$,  $\widehat{\Delta}^-(\lambda)$, $\widehat{\Pi}(\lambda)$, $\widehat{S}(\lambda)$ and $\widehat{W}(\lambda)$, respectively. Denote those for $\widehat{\Delta}_0(\lambda)$ by  $\widehat{\Delta}^{+}_0(\lambda)$, $\widehat{\Delta}^{-}_0(\lambda)$, $\widehat{\Pi}_0(\lambda)$, $\widehat{S}_0(\lambda)$and $\widehat{W}_0(\lambda)$.
	
	For a real root $\alpha\in\widehat{\Delta}$,  denote by $s_{\alpha}\in\widehat{W}$ the corresponding reflection. Then  $\widehat{\Pi}(\lambda)$ is  the set of $\alpha\in\widehat{\Delta}^+(\lambda)$ such that $s_{\alpha}(\widehat{\Delta}^+(\lambda)\backslash \{\alpha\})=\widehat{\Delta}^+(\lambda)\backslash \{\alpha\}$, and  $(\widehat{W}(\lambda), S(\lambda))$ is a Coxeter group \cite{KW89,KT00}. 
	
	\vskip 0.2cm
	For $w\in \widehat{W}(\lambda)$, denote by $\ell_{\lambda}(w)$ the length of $w$. Denote the Bruhat ordering of $\widehat{W}(\lambda)$ by $\geq_{\lambda}$. For $y,w\in \widehat{W}(\lambda)$, denote by $P^{\lambda}_{y,w}(q)\in\mathbb{Z}[q]$ the associated Kazhdan-Lusztig polynomial  \cite{KL79}, and by $Q^{\lambda}_{y,w}(q)\in\mathbb{Z}[q]$ the inverse Kazhdan-Lusztig polynomial defined by 
	$$
	\sum\limits_{x\leq_{\lambda}y\leq_{\lambda}z}(-1)^{\ell_{\lambda}(y)-\ell_{\lambda}(x)}Q^{\lambda}_{x,y}(q)P^{\lambda}_{y,z}(q)=\delta_{x,z},
	$$
	for any $x,z\in \widehat{W}(\lambda)$.  Set
	\begin{align*}
		{\mathcal C}&=\{\lambda\in\widehat{\kh}^*: (\delta|\lambda+\widehat{\rho})\neq 0 \}, \\
		{\mathcal C}^+&=\{\lambda\in{\mathcal C}:(\lambda+\widehat{\rho}|\alpha^{\vee})\geq 0, \  {\rm for \ any} \ \alpha\in\widehat{\Delta}^{+}(\lambda)\}, \\
		{\mathcal C}^-&=\{\lambda\in{\mathcal C}:(\lambda+\widehat{\rho}|\alpha^{\vee})\leq 0, \  {\rm for \ any} \ \alpha\in\widehat{\Delta}^{+}(\lambda)\}.
	\end{align*}
	Then $\widehat{W}_0(\lambda)$ is a finite group for any $\lambda\in{\mathcal C}$ \cite{KT00}. 
	
	For $\lambda\in\widehat{\kh}^*$, let $M(\lambda)$ (resp. $L(\lambda)$) be the Verma module (resp. the simple quotient) of $\widehat{\kg}$  with highest weight $\lambda$. 
	The following lemma comes from \cite{KK79,DGK82,Ku87,KW89}.
	\begin{lemma}\label{lweyl3} Let $\lambda\in{\mathcal C}^+$, and $w\in\widehat{W}(\lambda)$. If $L(\mu)$ is a subquotient of $M(w\circ \lambda)$, then there exists $y\in \widehat{W}(\lambda)$ such that $\mu=(yw)\circ \lambda$, and $w\leq_{\lambda}yw$.
	\end{lemma}	
	We have the following character formulas  from \cite{KT00}.
	\begin{theorem}\label{KT}
		\begin{enumerate}
			\item 	Let $\lambda\in{\mathcal C}^+$,  then for any $w\in\widehat{W}(\lambda)$ which is the longest element of $w\widehat{W}_0(\lambda)$, 
			$$
			\operatorname{ch}(L(w\circ \lambda))=\sum\limits_{w\leq_{\lambda}y\in \widehat{W}(\lambda)}(-1)^{\ell_{\lambda}(y)-\ell_{\lambda}(w)}Q^{\lambda}_{w,y}(1)\operatorname{ch}(M(y\circ \lambda)).
			$$
			\item  	Let $\lambda\in{\mathcal C}^-$,  then for any $w\in\widehat{W}(\lambda)$ which is the shortest element of $w\widehat{W}_0(\lambda)$, 
			$$
			\operatorname{ch}(L(w\circ \lambda))=\sum\limits_{w\geq_{\lambda}y\in \widehat{W}(\lambda)}(-1)^{\ell_{\lambda}(w)-\ell_{\lambda}(y)}P^{\lambda}_{y,w}(1)\operatorname{ch}(M(y\circ \lambda)).
			$$
		\end{enumerate}  
	\end{theorem}
	\subsection{Equivalence of categories}
	
	The following theorem comes from \cite{fiebig2006combinatorics}.
	\begin{theorem}\label{Fi06}
		Let ${\widehat{\kh}}\subset{\widehat{\mathfrak{b} }}\subset{\widehat{\kg}}$ and ${\widehat{\kh}}'\subset{\widehat{\mathfrak{b} }}'\subset{\widehat{\kg}}'$ be symmetrizable Kac-Moody Lie algebras with Cartan and Borel subalgebras, respectively. Let  $\Lambda\in\widehat{\kh}^*/\sim$ and $\Lambda'\in(\widehat{\kh}')^*/\sim'$ be two equivalence classes outside the critical hyperplanes and let ${\cal O}_{\Lambda}$ and ${\cal O}'_{\Lambda'}$ be the corresponding indecomposable blocks, where $\sim$ and $\sim'$ are the usual relation in $\widehat{\kh}^*$ and $(\widehat{\kh}')^*$, respectively (see \cite{KK79}). Suppose the following conditions hold:
		\begin{enumerate}
			\item \ There exist $\lambda\in\Lambda$ and $\lambda'\in\Lambda'$ which are either both dominant or both antidominant; 
			
			\item $(W(\lambda), S(\lambda))\cong (W(\lambda'), S(\lambda'))$ as Coxeter groups;
			 
			\item $W_0(\lambda)\cong W_{0}(\lambda')$ under the same isomorphism and both sets are finite.
		\end{enumerate}
		Then there exists an equivalence of categories
		$$
		{\cal O}_{\Lambda}\cong {\cal O'}_{\Lambda'}.
		$$
	\end{theorem}
\subsection{Criteria to determine  longest Weyl elements}	
	We now assume that $\mathfrak{g}$ is  of simply-laced type.  Let $\kappa_{p,q}=-\mathbf{h}^\vee+\frac{p}{q}$  such that  $q\geq 1$, $(p,q)=1$, and $p\geq 2$. By \cite{DGK82}, there exist $\Lambda^{(p,q)}\in{\mathcal C}^+$  and $\omega\in \widehat{W}(\Lambda^{(p,q)})$ 
	such that 
	$$
	\kappa_{p,q}\Lambda_0=\omega\circ \Lambda^{(p,q)}.
	$$
	Also there exists $\tilde{\omega}\in \widehat{W}(\Lambda^{(p,q)})$ such that 
	$$
\Lambda^{(p,q)}_{sing}=	\tilde{\omega}\circ \Lambda^{(p,q)},
	$$
where 	$\Lambda^{(p,q)}_{sing}$ denotes  a weight of minimal singular vectors given in previous sections.  

In the rest of the paper, we shall give the longest Weyl elements $\tilde{y}$ and $\tilde{z}$ in $\widehat{W}$ satisfying that
$$\tilde{y}\circ(\Lambda^{(p,q)} )=\kappa_{p,q}\Lambda_0,~~~
\tilde{z}\circ{\Lambda^{(p,q)}}=\Lambda_{sing}^{(p,q)}.$$
	Let $\{\alpha_1,\dots,\alpha_\ell\}$ be the simple root system of $\kg$, and $\{\alpha_0,\alpha_1,\dots,\alpha_\ell\}$  the simple root system of $\widehat{\kg}$.
	It is obvious that 
	$$
	\widehat{\Pi}(\Lambda^{(p,q)})=\widehat{\Pi}(\kappa_{p,q}\Lambda_0)=\{\beta_0=q\delta-\theta, \alpha_1, \cdots, \alpha_\ell\},
	$$
	and $$ \widehat{W}(\Lambda^{(p,q)})=\widehat{W}(\kappa_{p,q}\Lambda_0), \ \ 
	\widehat{W}(\Lambda^{(p,1)})=\widehat{W}.$$
	Then as shown in \cite{JS25},  we have 
	\begin{lemma}\label{lemma-iso}
		Let $\mathfrak{g}$ be of simply-laced type. 
		Then for  any $\kappa_{p,q}=-\mathbf{h}^\vee+\frac{p}{q}$ with $q\geq 1$, $p\geq 2$  and $(p,q)=1$, we have
		$$
		(\widehat{W}(\Lambda^{(p,q)}), S(\Lambda^{(p,q)}))\cong (\widehat{W}, S)
		$$
		and 
		$$
		(\widehat{W}_0(\Lambda^{(p,q)}), S_0(\Lambda^{(p,q)}))\cong (\widehat{W}_0(\Lambda^{(p,1)}), S_0(\Lambda^{(p,1)})),
		$$
		$$
		(\widehat{W}_0(\kappa_{p,q}\Lambda_0), S_0(\kappa_{p,q}\Lambda_0))\cong (\widehat{W}_0(\kappa_{p,1}\Lambda_0), S_0(\kappa_{p,1}\Lambda_0)).
		$$
		as Coxeter groups via the morphism $\sigma: (\widehat{W}, S)\to 	(\widehat{W}(\Lambda^{(p,q)}), S(\Lambda^{(p,q)}))
		$
		defined by
		$$\sigma (s_{\alpha_0})=s_{\beta_0}, ~~~
		\sigma (s_{\alpha_i})=s_{\alpha_i}, ~ i=1,2,\cdots,\ell.	
		$$
	\end{lemma}
	By  \Cref{Fi06} and \Cref{lemma-iso}, we have
	$$
	{\cal O}_{\Lambda^{(p,1)}}\cong {\cal O}_{\Lambda^{(p,q)}}.
	$$
	Therefore, we can reduce the problem to the case $q=1$.  In the following, we simply denote 	
	$$\kappa=\kappa_{p,1}, \ \  \Lambda=\Lambda^{(p,1)}.$$
	We  may write
	$$\Lambda=  \kappa\Lambda_0+D\delta+\overline{\Lambda},$$
	where $D\in\mathbb{Z}$ and $\overline{\Lambda}\in\kh^*$.

	The following lemma would assist us to check if a weight is in the shifted-orbit of $\kappa\Lambda_0$ by its restriction on $\kh^*+\mathbb{C}\Lambda_0$. 
	\begin{lemma}\label{orbitmoddelta}
		Let $\kappa\in\mathbb{C}$ satisfy $\kappa\neq-\mathbf{h}^\vee$.
		We have
		$$\widehat{W}\circ(\kappa\Lambda_0)\equiv W\circ(0)+(\kappa+\mathbf{h}^\vee)Q
		\equiv W\circ((\kappa+\mathbf{h}^\vee)Q) \pmod{\mathbb{C}\delta}.$$
		Moreover, for $\lambda\in\kh^*$, we have
		\begin{enumerate}
			\item if $w\circ(\lambda)=-(\kappa+\mathbf{h}^\vee)\gamma$ for some $w\in W,\gamma\in Q$, then 
			$$\kappa\Lambda_0+D\delta+\lambda=(t_\gamma w)^{-1}\circ (\kappa\Lambda_0)\in \widehat{W}\circ(\kappa\Lambda_0);$$
			\item if $\lambda=w\circ((\kappa+\mathbf{h}^\vee)\gamma)$ for some $w\in W,\gamma\in Q$, then 
			$$\kappa\Lambda_0+D\delta+\lambda=(wt_\gamma)\circ (\kappa\Lambda_0)\in \widehat{W}\circ(\kappa\Lambda_0),$$
		\end{enumerate}
		where 
		\begin{equation}\label{adde1}
		D=\frac{1}{2(\kappa+\mathbf{h}^\vee)}(\|\rho\|^2-\|\lambda+\rho\|^2).
		\end{equation}
		
	\end{lemma}
	\begin{proof}
		The lemma  follows from $\widehat{W}=T\rtimes W$, $t_{w(\alpha)}=wt_\alpha w^{-1}$ and the following deduction:
		\begin{align*}
			(wt_\gamma)\circ (\kappa\Lambda_0)-(\kappa\Lambda_0+\lambda)
			=&-((\rho|\gamma)+\frac{1}{2}(\kappa+\mathbf{h}^\vee)(\gamma|\gamma))\delta+w(\rho)-\rho+(\kappa+\mathbf{h}^\vee)w(\gamma)-\lambda\\
			=&-((\rho|\gamma)+\frac{1}{2}(\kappa+\mathbf{h}^\vee)(\gamma|\gamma))\delta,
		\end{align*}
		\begin{align*}
			-((\rho|\gamma)+\frac{1}{2}(\kappa+\mathbf{h}^\vee)(\gamma|\gamma))
			=&-((\rho|\gamma)+\frac{1}{2}(w^{-1}(\lambda+\rho)-\rho|\gamma))
			=-\frac{1}{2}(w^{-1}(\lambda+\rho)+\rho|\gamma)\\
			=&-\frac{1}{2(\kappa+\mathbf{h}^\vee)}(w^{-1}(\lambda+\rho)+\rho|w^{-1}(\lambda+\rho)-\rho)\\
			=&-\frac{1}{2(\kappa+\mathbf{h}^\vee)}(\|w^{-1}(\lambda+\rho)\|^2-\|\rho\|^2)=D.
		\end{align*}
	\end{proof}
	
	Define the following subset of roots: 
	\begin{align*}
		\widehat{\Delta}_p^+:=&\{m\delta-\overline{\alpha}\in \widehat{\Delta}_{re}^+: \rht(\overline{\alpha})=mp,~ \overline{\alpha}\in \Delta\}.
	\end{align*}
	
	\begin{lemma}\label{reflectionsinW0}
		The group $\widehat{W}_0(\kappa\Lambda_0)$ is generated by
		the reflections $s_\alpha$ with $\alpha\in\widehat{\Delta}_p^+$.
	\end{lemma}
	\begin{proof}
		The group $\widehat{W}_0(\kappa\Lambda_0)$ is generated by the reflections contained in it. For $\alpha=m\delta-\overline{\alpha}\in \widehat{\Delta}^{re,+}$, the reflection $s_\alpha$ is in the group $\widehat{W}_0(\kappa\Lambda_0)$ if and only if 
		$$0=(\alpha|\kappa\Lambda_0+\hat{\rho})=m(\kappa+\mathbf{h}^\vee)-(\overline{\alpha}|\rho)=mp-\rht(\overline{\alpha}).$$
	\end{proof}

	\begin{lemma}\label{criteriononlongesty}
		\begin{enumerate}
			\item Let $y\in \widehat{W}$ be such that $y\circ \Lambda=\kappa\Lambda_0$. Then the left coset 
			$$y\widehat{W}_0(\Lambda)=\{t_\gamma w: w\circ(\overline{\Lambda})=-p\gamma,~\gamma\in Q,~w\in W\}.$$
			\item The longest element $\widetilde{y}$ in $y\widehat{W}_0(\Lambda)$ is the one satisfying that for all $\alpha\in \widehat{\Delta}_p^+$,
			$$\widetilde{y}^{-1}(\alpha)<0,$$
			or equivalently, writing $\widetilde{y}=t_\gamma w$,  for all $\alpha=m\delta-\overline{\alpha}\in \widehat{\Delta}_p^+$,
			\begin{equation}\label{gammaandwforlongesty}
				((\overline{\alpha}|\gamma)-m)\delta+w^{-1}(\overline{\alpha})>0.
			\end{equation}
		\end{enumerate}
	\end{lemma}
	
	
	\begin{proof}
		\begin{enumerate}[wide]
			\item This follows from \Cref{orbitmoddelta}.
			\item It is clear that  $y\widehat{W}_0(\Lambda)=\widehat{W}_0(\kappa\Lambda_0)y$. 
			So by \Cref{reflectionsinW0}, the longest element $\widetilde{y}$ in $y\widehat{W}_0(\Lambda)$ satisfies that $\ell(s_\alpha \widetilde{y})<\ell (\widetilde{y})$ for any $\alpha \in \widehat{\Delta}_{p}^+$, that is, $\widetilde{y}^{-1}(\alpha)<0$. The equivalent condition \eqref{gammaandwforlongesty} is obtained by the following calculation:
			$$\widetilde{y}^{-1}(\alpha)=w^{-1}t_{-\gamma}(m\delta-\overline{\alpha})=w^{-1}(m\delta-\overline{\alpha}-(\overline{\alpha}|\gamma)\delta)=(m-(\overline{\alpha}|\gamma))\delta-w^{-1}(\overline{\alpha}).$$
		\end{enumerate}
	\end{proof}

	\begin{lemma}\label{criteriononlongestz}
		\begin{enumerate}
			\item Let $z_0\in \widehat{W}$ be such that $z_0\circ (\kappa\Lambda_0)=\Lambda_{sing}$, and $z\in \widehat{W}$  such that $z\circ (\Lambda)=\Lambda_{sing}$. Then the left coset 
			$$z_0\widehat{W}_0(\kappa\Lambda_0)=\{wt_\gamma : w\circ(p\gamma)=\overline{\Lambda_{sing}},~\gamma\in Q,~w\in W\}.$$
			\item Let $y\in\widehat W$ satisfy
			$y\circ\Lambda=\kappa\Lambda_0$, and let $\widetilde y$
			be the longest element of the left coset
			$y\widehat W_0(\Lambda)$. Then 
			the longest element $\widetilde{z}$ in $z\widehat{W}_0(\Lambda)$ is $\widetilde{z_0}\widetilde{y}$, where $\widetilde{z_0}\in z_0\widehat{W}_0(\kappa\Lambda_0)$ is the one satisfying that for any $\alpha \in \widehat{\Delta}_{p}^+$,
			$$\widetilde{z_0}(\alpha)>0,$$
			or equivalently, writing $\widetilde{z_0}=wt_{\gamma}$,  for all $\alpha=m\delta-\overline{\alpha}\in \widehat{\Delta}_p^+$,
			\begin{equation}\label{gammaandwforlongestz}
				((\overline{\alpha}|\gamma)+m)\delta-w(\overline{\alpha})> 0.
			\end{equation}
		\end{enumerate}
	\end{lemma}
	
	\begin{proof}
		The proof is analogous to that of \Cref{criteriononlongesty}.
		Here we explain the condition $\widetilde{z_0}(\alpha)>0$ for $\alpha \in \widehat{\Delta}_{p}^+$.
		Since $\tilde{y}^{-1}\widehat{W}_0(\kappa\Lambda_0) \tilde{y}=\widehat{W}_0(\Lambda)$, the reflections in $\widehat{W}_0(\Lambda)$ are those $\tilde{y}^{-1}s_\alpha \tilde{y}=s_{\tilde{y}^{-1}\alpha}$ for all $\alpha\in\widehat{\Delta}_p^+$.
		Recall that $\widetilde{y}$ satisfies that $\widetilde{y}^{-1}(\alpha)<0$ for all $\alpha\in\widehat{\Delta}_p^+$. 
		Then for $\beta\in\widehat{\Delta}^+$ such that $s_\beta\in \widehat{W}_0(\Lambda)$, we have $\beta=-\widetilde{y}^{-1}(\alpha)$ for some  $\alpha\in \widehat{\Delta}_p^+$, and hence
		$\widetilde{z_0}\widetilde{y}(\beta)=-\widetilde{z_0}(\alpha)<0$, which is equivalent to $\ell(\widetilde{z_0}\widetilde{y}s_\beta)<\ell(\widetilde{z_0}\widetilde{y})$.
	\end{proof}
	
	Define 
	$$
	{\Delta}_p^+:=\{\overline{\alpha}\in {\Delta}^+: \rht(\overline{\alpha})=p\}.
	$$
	
	\begin{coro}\label{gammaandwforlongestyandzm1}
		Let $\kg$ be of type $\mathsf{A}$ or $\mathsf{D}$. The element $\widetilde{y}=t_\gamma w\in \widehat{W}$ in  \Cref{criteriononlongesty} is the one satisfying that
		\begin{enumerate}
			\item $w\circ(\overline{\Lambda})=-p\gamma$;
			\item $((\overline{\alpha}|\gamma)-1)\delta+w^{-1}(\overline{\alpha})>0$ for all $\overline{\alpha}\in {\Delta}_p^+$;
			\item if $\kg$ is of type $\mathsf{D}_n$, $2|p$ and $n>\frac{3}{2}p$, then $((\overline{\alpha}|\gamma)-2)\delta+w^{-1}(\overline{\alpha})>0$ holds for $\overline{\alpha}=\epsilon_{n-\frac{3}{2}p}+\epsilon_{n-\frac{p}{2}}$.
		\end{enumerate} 
		The element $\widetilde{z}$ in  \Cref{criteriononlongestz} is $\widetilde{z_0}\widetilde{y}$, where $\widetilde{z_0}=w t_{\gamma}\in \widehat{W}$ is the one satisfying that
		\begin{enumerate}
			\item $w\circ(p\gamma)=\overline{\Lambda_{sing}}$;
			\item $((\overline{\alpha}|\gamma)+1)\delta-w(\overline{\alpha})>0$ for all $\overline{\alpha}\in {\Delta}_p^+$;
			\item if $\kg$ is of type $\mathsf{D}_n$, $2|p$ and $n>\frac{3}{2}p$, then $((\overline{\alpha}|\gamma)+2)\delta-w(\overline{\alpha})>0$ holds for $\overline{\alpha}=\epsilon_{n-\frac{3}{2}p}+\epsilon_{n-\frac{p}{2}}$.
		\end{enumerate} 
	\end{coro}
	
	\begin{proof}
		For type $\mathsf{A}$, any root with height $mp$ can be decomposed as a sum of $m$ roots with height $p$, which can be seen from the standard realization of the root system. 
		So we can verify the condition \eqref{gammaandwforlongesty} and \eqref{gammaandwforlongestz} only for $m=1$ from the linearity. 
		
		For type $\mathsf{D}$, a positive root is either $\epsilon_i-\epsilon_j$ or $\epsilon_i+\epsilon_j$.
		The former obviously decomposes when its height is a multiple of $p$.
		For $\epsilon_i+\epsilon_j$ of height $mp$, one can repeatedly subtract roots of the form $\epsilon_a-\epsilon_{a+p}$ (height $p$) to reduce the height by $p$, until either a single root of height $p$ remains or the process is blocked. 
		The only indecomposable case occurs when $2| p$, $n>\frac{3}{2}p$, and the root equals $\epsilon_{n-\frac{3}{2}p}+\epsilon_{n-\frac{p}{2}}$ (height $2p$). 
		This root is exactly covered by the additional condition in the statement.
	\end{proof}
	
	\section{Characterization of Weyl elements $\tilde{y}$ and $\tilde{z}_0$ }
	\subsection{Type $\mathsf{A}$}
	
	Recall that $W\cong \mathfrak{S}_n$. We shall use usual notations for  symmetric groups. The notation $(a_1~a_2~\cdots~ a_m)$ stands for the cycle on $a_1,a_2\dots,a_m$, and 
	$$\binom{a_1~a_2~\cdots ~a_n}{b_1~b_2~\cdots~b_n}$$
	stands for the permutation such that $a_i\mapsto b_i$, $1\leq i\leq n$. 
	
	For any $i\in\mathbb{Z}$, denote the increasing (resp. decreasing) sequence of those $x$'s between $1$ and $n$ with $x\equiv i\pmod p$ by $\overline{i}$ (resp. $\overline{i}^\downarrow$), that is, 
	\begin{multline*}
		\overline{i}=i-Np,\dots,i-2p,i-p,i,i+p,i+2p,\dots,i+Mp \\
		(\text{resp. }\overline{i}^\downarrow=i+Mp,\dots,i+p,i,i-p,\dots,i-Np),
	\end{multline*}
	where $1\in (i-(N+1)p,i-Np]$ and $n\in [i+Mp,i+(M+1)p)$.
	Also, denote by $\overline{i}_\rightarrow$ the increasing sequence of those $x$'s between $n/2$ and $n$ with $x\equiv i\pmod p$, and $\overline{i}_\leftarrow$ the increasing sequence of those $x$'s between $1$ and $n/2$ with $x\equiv i\pmod p$. 
	In addition, we shall use the notation $\overline{\epsilon_i}$ for the sum of all  $\epsilon_\bullet$'s with index in $\overline{i}$.
	We denote by $|A|$ the number of elements in $A$.
	
	\begin{lemma}
		Assume that $\kg$ is of type $\mathsf{A}_{n-1}$.  
		\begin{enumerate}[wide]
			\item If $2|n-p$, the weight $\overline{\Lambda}+\rho$ is 
			$$[\underset{|\overline{(n-p)/2+1}|}{\underbrace{\tfrac{p+1}{2}-1,\dots,\tfrac{p+1}{2}-1}},
			\underset{|\overline{(n-p)/2+2}|}{\underbrace{\tfrac{p+1}{2}-2,\dots,\tfrac{p+1}{2}-2}},
			\dots,
			\underset{|\overline{(n-p)/2+p}|}{\underbrace{\tfrac{p+1}{2}-p,\dots,\tfrac{p+1}{2}-p}}].$$
			\item If $2\nmid n-p$, the weight $\overline{\Lambda}+\rho$ is 
			$$[\underset{|\overline{(n-p+1)/2}|/2}{\underbrace{\tfrac{p}{2},\dots,\tfrac{p}{2}}},
			\underset{|\overline{(n-p+1)/2+1}|}{\underbrace{\tfrac{p}{2}-1,\dots,\tfrac{p}{2}-1}},
			\underset{|\overline{(n-p+1)/2+2}|}{\underbrace{\tfrac{p}{2}-2,\dots,\tfrac{p}{2}-2}},
			\dots,
			\underset{|\overline{(n-p+1)/2+p-1}|}{\underbrace{\tfrac{p}{2}-(p-1),\dots,\tfrac{p}{2}-(p-1)}},
			\underset{|\overline{(n-p+1)/2}|/2}{\underbrace{-\tfrac{p}{2},\dots,-\tfrac{p}{2}}}].$$
		\end{enumerate}
	\end{lemma}
	
	\begin{proof}
		With $\gamma$ and $w$ given in \Cref{Alongesty}, one can verify  by \Cref{orbitmoddelta} that 
		the weight $\Lambda=\kappa\Lambda_0+D\delta+\overline{\Lambda}$ is in the shifted-orbit of $\kappa\Lambda_0$ for some $D$.
		Then it suffices to show that 
		$$(\kappa\Lambda_0+D\delta+\overline{\Lambda}+\widehat{\rho}|\alpha)=(p\Lambda_0+\overline{\Lambda}+\rho|\alpha)\geq 0$$ 
		for any $\alpha\in\widehat{\Pi}$, since $\widehat{\Delta}(\kappa\Lambda_0)=\widehat{\Delta}$.
		It can be checked directly that the condition holds for $\alpha=\alpha_1,\dots,\alpha_{n-1}$. For $\alpha=\alpha_0=\delta-\theta$, the condition turns out to be 
		$(\overline{\Lambda}+\rho|\theta)\leq p$, 
		which also holds since $\theta=\epsilon_1-\epsilon_n$.
	\end{proof}
	
	Given $n$ and $p$, 
	we have $\Delta_p^+=\{\epsilon_i-\epsilon_{i+p}:1\leq i\leq n-p\}$.
	Then the following results on longest elements can be verified by
	\Cref{gammaandwforlongestyandzm1}.
	
	\begin{theorem}\label{Alongesty}
		Assume that $\kg$ is of type $\mathsf{A}_{n-1}$. 
		Let $\tilde{y}$ be the longest Weyl element in $\widehat{W}$ satisfying that
		$$\tilde{y}\circ(\Lambda )=\kappa\Lambda_0.$$
		\begin{enumerate}[wide]
			\item If $2|n-p$, the Weyl element $\widetilde{y}=t_\gamma w$ is given by
			$$\gamma=[\overset{n}{\overbrace{\dots,\underset{p}{\underbrace{2,2,\dots,2}},\underset{p}{\underbrace{1,1,\dots,1}},\underset{p}{\underbrace{0,0,\dots,0}},\underset{p}{\underbrace{-1,-1,\dots,-1}},\underset{p}{\underbrace{-2,-2,\dots,-2}},\dots}}]$$
			and 
			$$w=
			\binom{1~2~3~\cdots~n-1~ n}{\overline{(n-p)/2+1}~\overline{(n-p)/2+2}~\cdots~\overline{(n-p)/2+p}}.$$
			\item If $2\nmid n-p$, the Weyl element $\widetilde{y}=t_\gamma w$ is given by
			$$\gamma=[\overset{n}{\overbrace{\dots,\underset{p}{\underbrace{2,2,\dots,2}},\underset{p}{\underbrace{1,1,\dots,1}},\underset{p-1}{\underbrace{0,0,\dots,0}},\underset{p}{\underbrace{-1,-1,\dots,-1}},\underset{p}{\underbrace{-2,-2,\dots,-2}},\dots}}]$$
			and
			$$
			w=\binom{1~2~3~\cdots~n-1~ n}{\overline{(n+p+1)/2}_\rightarrow~\overline{(n-p+1)/2+1}~\cdots~\overline{(n-p+1)/2+p-1}~\overline{(n-p+1)/2}_\leftarrow}.
			$$
		\end{enumerate}
	\end{theorem}
	\begin{theorem}\label{Alongestz}
		Assume that $\kg$ is of type $\mathsf{A}_{n-1}$. Let $\tilde{z}$  be the longest Weyl element in $\widehat{W}$ satisfying that
		$$\tilde{z}\circ{\Lambda}=\Lambda_{sing},$$
		where $\Lambda_{sing}$ is a weight of minimal singular vectors given in \Cref{JS25b}.
		Then $\widetilde{z}=\widetilde{z_0}\widetilde{y}$, where $\widetilde{z_0}$ is given as follows: 
		\begin{enumerate}[wide] 
			\item When $p\geq n$, $\widetilde{z_0}=s_0=s_\theta t_{-\theta}$;
			\item When $p=2<n$, $\widetilde{z_0}=(\overline{1} ~\overline{2}^\downarrow) t_{-(\alpha_1+\alpha_3+\cdots+\alpha_{n-1})}$  if $2|n$, and
			 $\widetilde{z_0}=t_{\theta}s_{n-1}$ or $t_{\theta}s_1$ respectively if $2\nmid n$;
			\item When $n=5$ and $p=3$, three weights of minimal singular vectors correspond to $\widetilde{z_0}=
				(1452)t_{-(\alpha_1+\alpha_4)},~
				(543)(21)t_{\alpha_2},~
				(123)(45)t_{\alpha_3}$, 
			respectively;
			\item When $n=7$ and  $p=4$, $\widetilde{z_0}=(1573)t_{-(\alpha_1+\alpha_2+\alpha_5+\alpha_6)}$;
			\item When $n=8$ and $p=3$, $\widetilde{z_0}=(147852)t_{-(\alpha_1+\alpha_4+\alpha_7)}$;
			\item For other cases, set $s_1=\lfloor n/p\rfloor, s_2=\lceil n/p\rceil$ and define function $D(s)=(|sp-n|+1)s$. Set $D_p=\min\{D(s_1),D(s_2)\}$. Then $D_p=D(s_1)$ corresponds to 
			$$\widetilde{z_0}=
			(\overline{0}^\downarrow~1~2~\cdots~n-s_1p~\overline{n+1}~n~n-1~\cdots~s_1p+1)
			t_{\overline{\epsilon_p}-\overline{\epsilon_{n+1-p}}},$$
			and $D_p=D(s_2)$ corresponds to 
			$$\widetilde{z_0}=(\overline{1}~\overline{n}^\downarrow)t_{\overline{\epsilon_n}-\overline{\epsilon_1}}.$$
		\end{enumerate}
	\end{theorem}

	\subsection{Type $\mathsf{D}$}
	
	Recall that $W\cong(\mathbb{Z}_2)^{n-1}\rtimes\mathfrak{S}_n$. 
	For $I\subset \{1,2,\dots,n\}$, 
	we shall use the notation $\operatorname{flip}(I)$ for the action $$\epsilon_{i}\mapsto -\epsilon_{i}~(i\in I), ~~~\epsilon_i\mapsto\epsilon_i~(i\notin I)$$ 
	on weights. 
	
	For odd $p$, define a subset of $\mathbb{Z}$ as
	$$A_p:=\{a+b+c:a\in 4p\mathbb{Z}, b\in\{1,3,5,\dots,2p-1\}, c=\frac{p+1}{2} \ {\rm or} \ \frac{3p+1}{2}\}.$$
	For even $p$, define a subset of $\mathbb{Z}$ as
	$$B_p:=\{a+2b+\frac{p}{2}-1: a\in 2p\mathbb{Z}, b\in \{1,2,3, \dots,\frac{p}{2}\}\}.$$
	
	\begin{lemma}
		Assume that $\kg$ is of type $\mathsf{D}_{n}$. 
		\begin{enumerate}[wide]
			\item If $2\nmid p$ and $n\notin A_p$, the weight $\overline{\Lambda}+\rho$ is 
			$$[\underset{|\overline{(p-1)/2+1}|+|\overline{(p+1)/2+1}|}{\underbrace{\tfrac{p-1}{2},\dots,\tfrac{p-1}{2}}},
			\underset{|\overline{(p-3)/2+1}|+|\overline{(p+3)/2+1}|}{\underbrace{\tfrac{p-3}{2},\dots,\tfrac{p-3}{2}}},
			\dots,
			\underset{|\overline{2}|+|\overline{p}|}{\underbrace{1,\dots,1}},
			\underset{|\overline{1}|}{\underbrace{0,\dots,0}}].$$
			\item If $2\nmid p$ and $n\in A_p$, the weight $\overline{\Lambda}+\rho$ is 
			$$[\tfrac{p+1}{2},\underset{|\overline{(p-1)/2+1}|+|\overline{(p+1)/2+1}|-1}{\underbrace{\tfrac{p-1}{2},\dots,\tfrac{p-1}{2}}},
			\underset{|\overline{(p-3)/2+1}|+|\overline{(p+3)/2+1}|}{\underbrace{\tfrac{p-3}{2},\dots,\tfrac{p-3}{2}}},
			\dots,
			\underset{|\overline{2}|+|\overline{p}|}{\underbrace{1,\dots,1}},
			\underset{|\overline{1}|}{\underbrace{0,\dots,0}}].$$
			\item If $2|p$, the weight $\overline{\Lambda}+\rho$ is 
			$$[\underset{|\overline{p/2+1}|}{\underbrace{\tfrac{p}{2},\dots,\tfrac{p}{2}}},
			\underset{|\overline{(p-2)/2+1}|+|\overline{(p+2)/2+1}|}{\underbrace{\tfrac{p-2}{2},\dots,\tfrac{p-2}{2}}},
			\dots,
			\underset{|\overline{2}|+|\overline{p}|}{\underbrace{1,\dots,1}},
			\underset{|\overline{1}|}{\underbrace{0,\dots,0}}].$$
		\end{enumerate}
	\end{lemma}
	
	\begin{proof}
		The proof is analogous to that of type $\mathsf{A}$. 
		Note that a weight $[a_1,a_2,\dots,a_n]$ is in $Q$ if and only if $a_i\in\mathbb{Z}$ and $\sum a_i$ is even, which is guaranteed by the definitions of  $A_p$ and $B_p$, respectively. It can be checked directly that the condition $w(\overline{\Lambda}+\rho)= \rho-p\gamma$ is satisfied for  $w$ and $\gamma$ given in \Cref{Dlongesty}.
	\end{proof}

	Given $n$ and $p$, 
	we have 
	\begin{equation*}
		\Delta_p^+
		=
		\{\epsilon_i-\epsilon_{i+p}:1\leq i\leq n-p\}
		\cup
		\{\epsilon_i+\epsilon_j:
		1\leq i<j\leq n,\ i+j=2n-p\}.
	\end{equation*}
	Then the following results on longest elements can be verified by
	\Cref{gammaandwforlongestyandzm1}.
	
	\begin{theorem}\label{Dlongesty}
		Assume that $\kg$ is of type $\mathsf{D}_{n}$. 
		Let $\tilde{y}$ be the longest Weyl element in $\widehat{W}$ satisfying that
		$$\tilde{y}\circ(\Lambda )=\kappa\Lambda_0.$$
		Define $Q_I=\{n\}$ if $|I|$ is odd, otherwise empty.
		\begin{enumerate}[wide]
			\item If $2\nmid p$ and $n\notin A_p$, the Weyl element $\widetilde{y}=t_\gamma w$ is given by
			$$\gamma=[\dots,\underset{p}{\underbrace{2,2,\dots,2}},\underset{p}{\underbrace{1,1,\dots,1}},\underset{\frac{p+1}{2}}{\underbrace{0,0,\dots,0}}],$$
			\begin{align*}
				w=\operatorname{flip}(I\cup Q_I)
				\binom{1\cdots n}{\overline{n-\tfrac{p-1}{2}}~\overline{n+\tfrac{p-1}{2}}^{\downarrow}~\cdots~\overline{n-2}~\overline{n+2}^{\downarrow}~\overline{n-1}~\overline{n+1}^{\downarrow}~\overline{n}},
			\end{align*}
			where $I=\{\overline{n+1},\overline{n+2},\dots,\overline{n+\tfrac{p-1}{2}}\}$.
			\item If  $2\nmid p$ and $n\in A_p$, the Weyl element $\widetilde{y}=t_\gamma w$ is given by
			$$\gamma=[m+1,m+1,\dots,m+1,\underset{p-1}{\underbrace{m,m,\dots,m}},\dots,\underset{p}{\underbrace{2,2,\dots,2}},\underset{p}{\underbrace{1,1,\dots,1}},\underset{\frac{p+1}{2}}{\underbrace{0,0,\dots,0}}],$$
			\begin{align*}
				w=
				\operatorname{flip}(I\cup Q_I)
				\begin{pmatrix}
					1\cdots n\\
					\begin{array}{l}
						\overline{n-\frac{p-1}{2}}\; \overline{n+\frac{p-1}{2}}^{\downarrow}\; \cdots\; \overline{n-2}\; \overline{n+2}^{\downarrow}\; \overline{n-1}\; \overline{n+1}^{\downarrow}\; \overline{n}
					\end{array} 
				\end{pmatrix},
			\end{align*}
			where $m=\lfloor(n-\frac{p+1}{2})/p\rfloor$ and
			$I=\{\overline{n+1},\overline{n+2},\dots,\overline{n+\tfrac{p-1}{2}},\min\{\overline{n-\frac{p-1}{2}}\}\}$.
			\item If $2| p$ and $n\notin B_p$, the Weyl element $\widetilde{y}=t_\gamma w$ is given by
			$$\gamma=[\dots,\underset{p}{\underbrace{3,3,\dots,3}},\underset{p}{\underbrace{2,2,\dots,2}},\underset{p}{\underbrace{1,1,\dots,1}},\underset{\frac{p}{2}}{\underbrace{0,0,\dots,0}}],$$
			\begin{align*}
				w=
				\operatorname{flip}(I\cup Q_I)
				\binom{1\cdots n}{\overline{n-\tfrac{p}{2}}^{\downarrow}~\overline{n-\tfrac{p-2}{2}}~\overline{n+\tfrac{p-2}{2}}^{\downarrow}~\cdots~\overline{n-2}~\overline{n+2}^{\downarrow}~\overline{n-1}~\overline{n+1}^{\downarrow}~\overline{n}},
			\end{align*}
			where $I=\{\overline{n+1},\overline{n+2},\dots,\overline{n+\tfrac{p}{2}}\}$.
			\item If $2| p$ and $n\in B_p$, the Weyl element $\widetilde{y}=t_\gamma w$ is given by
			$$\gamma=[\dots,\underset{p}{\underbrace{3,3,\dots,3}},\underset{p}{\underbrace{2,2,\dots,2}},\underset{p-1}{\underbrace{1,1,\dots,1}},\underset{\frac{p}{2}+1}{\underbrace{0,0,\dots,0}}],$$
			\begin{align*}
				w=
				\operatorname{flip}(I\cup Q_I)
				\binom{1\cdots n}{\overline{n-\tfrac{p}{2}}^{\downarrow}~\overline{n-\tfrac{p-2}{2}}~\overline{n+\tfrac{p-2}{2}}^{\downarrow}~\cdots~\overline{n-2}~\overline{n+2}^{\downarrow}~\overline{n-1}~\overline{n+1}^{\downarrow}~\overline{n}},
			\end{align*}
			where $I=\{\overline{n+1},\overline{n+2}, \dots, \overline{n+\tfrac{p-2}{2}}, \overline{n+\tfrac{p}{2}}\}\setminus\{n-\tfrac{p}{2}\}$.
		\end{enumerate}
	\end{theorem}
	\begin{theorem}\label{Dlongestz}
		Assume that $\kg$ is of type $\mathsf{D}_{n}$. Let $\tilde{z}$  be the longest Weyl element in $\widehat{W}$ satisfying that
		$$\tilde{z}\circ{\Lambda}=\Lambda_{sing},$$
		where $\Lambda_{sing}$ is a  weight of minimal singular vectors given in \Cref{main2}.
		Then $\widetilde{z}=\widetilde{z_0}\widetilde{y}$, where $\widetilde{z_0}$ is given as follows: 
		\begin{enumerate}[wide] 
			\item When $p\geq 2n-2$, $\widetilde{z_0}=s_0=s_\theta t_{-\theta}$;
			\item When $p=3$ and $3|n-1$,   $\widetilde{z_0}=(23) t_{[1,0,1,0,0,\dots,0]}$;
			\item When $p=5$ and $n=7$, $\widetilde{z_0}=(13)(245) t_{[0,0,1,0,1,0,0]}$;
			\item When $p=5$ and $n=12$, $\widetilde{z_0}=(1~6~11~8~3~2)
			\operatorname{flip}(\{11,12\}) t_{[-1,1,1,0,0,-1,0,1,0,0,-1,0]}$;
			\item When $p=4$ and $n=4$,  three weights of minimal singular vectors correspond to  $\widetilde{z_0}=(13) \operatorname{flip}(\{1,4\})t_{[-1,0,1,0]}$, $(1234)\operatorname{flip}(\{2,4\})t_{[0,-1,0,-1]}$,  $(1234)\operatorname{flip}(\{2,3\})t_{[0,-1,0,1]}$, respectively;
			\item When $p=5$ and $n=4$,   three weights of minimal singular vectors correspond to   $\widetilde{z_0}=(12)\operatorname{flip}(\{1,4\}) t_{[-1,1,0,0]}$, $(1324)\operatorname{flip}(\{3,4\})t_{[0,0,-1,-1]}$, $(1324)\operatorname{flip}(\{2,3\})t_{[0,0,-1,1]}$, respectively;
			
			\item In all other cases, as defined in \Cref{main2}, the case $D=D_{(0)}$ corresponds to 
			$$\widetilde{z_0}=
			\begin{cases}
				(123\cdots p)\operatorname{flip}(\{m,n\})
				t_{\epsilon_p-\epsilon_{m}},&p<n,~3p\neq 2n,\\
				(123\cdots p)\operatorname{flip}(\{p,n\})
				t_{-2\epsilon_p},&p<n,~3p=2n,\\
				(123\cdots 2n-p)
				\operatorname{flip}(\{m\})t_{-\epsilon_m}w(2n),&p>n,
			\end{cases}
			$$ 
			the case $D=D_{(1)}$ corresponds to 
			$$\widetilde{z_0}=
			\begin{cases}
				(123\cdots r)\operatorname{flip}(\{m\})t_{-\epsilon_m}w(2n),
				&2\nmid s_1(p-1),\\
				(12)(123\cdots r+1)^2w(2n+1)w(2n),
				&2|s_1(p-1),
			\end{cases}
			$$ 
			and the case $D=D_{(2)}$ corresponds to 
			$$\widetilde{z_0}=
			\begin{cases}
				(12)w(1)w(2),& 2\nmid s_2,\\
				\operatorname{flip}(\{n\})w(1),& 2| s_2,
			\end{cases}$$
			where $r=2n-s_1p$, $m=n-p/2$, and
			$$w(j)=\begin{cases}
				\operatorname{flip}(\{\min\overline{j}\})
				t_{\epsilon_{\min\overline{j}}-\epsilon_m},
				&\overline{j}=\overline{2n-j}=\overline{m},\\
				\operatorname{flip}(\{n\})
				t_{\epsilon_{\min\overline{j}}},
				&\overline{j}=\overline{2n-j}=\overline{n},\\
				(\overline{j}~\overline{2n-j}^\downarrow)
				\operatorname{flip}(\{\max \overline{j}\})
				t_{\overline{\epsilon_{2n-j}}-\overline{\epsilon_j}},
				&\overline{j}\neq\overline{2n-j}.
			\end{cases}$$
		\end{enumerate}
	\end{theorem}

	\subsection{Type $\mathsf{E}$}
	
	The following Weyl elements were obtained by a finite computation and
	can be verified directly using \Cref{criteriononlongesty} and
	\Cref{criteriononlongestz}.
	
	\begin{lemma}
		Assume that $\kg$ is of type $\mathsf{E}$. 
		\begin{enumerate}[wide]
			\item If $p\geq \mathbf{h}^\vee$, the weight $\overline{\Lambda}+\rho=\rho$, that is, $\overline{\Lambda}=0$.
			\item If $p< \mathbf{h}^\vee$, the weight $\overline{\Lambda}+\rho$ is given by \Cref{tab:ELam}, where $(c_1,c_2,\dots,c_\ell)$ denotes the weight $c_1\alpha_1+c_2\alpha_2+\cdots+c_\ell\alpha_\ell$.
		\end{enumerate}
	\end{lemma}
	
	\begin{table}[H] 
		\centering
		\footnotesize
		\caption{ $\overline{\Lambda}+\rho$ for type $\mathsf{E}$  ($p<\mathbf{h}^\vee$)}
		\label{tab:ELam}
		\begin{tabular}{cccc}
			\toprule
			$p$  & $\mathsf{E}_6$ & $\mathsf{E}_7$  & $\mathsf{E}_8$ \\
			\midrule
			2  & $(1, 2, 2, 3, 2, 1)$ & $(2, 7/2, 4, 6, 9/2, 3, 3/2)$ & $(4,5,7,10,8,6,4,2)$ \\
			3  & $(2, 3, 4, 6, 4, 2)$ & $(3, 9/2, 6, 9, 15/2, 5, 5/2)$ & $(5,8,10,15,12,9,6,3)$ \\
			4  & $(2, 3, 4, 6, 4, 2)$ & $(3, 9/2, 6, 9, 15/2, 5, 5/2)$ & $(6,9,12,18,15,12,8,4)$ \\
			5  & $(4, 5, 7, 10, 7, 4)$ & $(5, 15/2, 10, 15, 23/2, 8, 9/2)$ & $(8,12,16,24,20,15,10,5)$ \\
			6  & $(4, 5, 7, 10, 7, 4)$ & $(5, 15/2, 10, 15, 23/2, 8, 9/2)$ & $(8,12,16,24,20,15,10,5)$ \\
			7  & $(4, 6, 8, 11, 8, 4)$ & $(7, 19/2, 13, 19, 29/2, 10, 11/2)$ & $(10,15,20,30,25,19,13,7)$ \\
			8  & $(5, 7, 9, 13, 9, 5)$ & $(7, 21/2, 14, 21, 33/2, 12, 13/2)$ & $(12,18,24,36,29,22,15,8)$ \\
			9  & $(6, 8, 11, 15, 11, 6)$ & $(9, 25/2, 17, 25, 39/2, 14, 15/2)$ & $(12,18,24,36,29,22,15,8)$ \\
			10 & $(7, 9, 13, 18, 13, 7)$ & $(9, 27/2, 18, 26, 41/2, 14, 15/2)$ & $(14,21,28,42,34,26,18,9)$ \\
			11 & $(8, 11, 15, 21, 15, 8)$ & $(10, 29/2, 19, 28, 43/2, 15, 15/2)$ & $(18,26,35,52,42,32,22,11)$ \\
			12 & -- & $(11, 31/2, 21, 30, 47/2, 16, 17/2)$ & $(18,26,35,52,42,32,22,11)$ \\
			13 & -- & $(12, 33/2, 23, 33, 51/2, 18, 19/2)$ & $(18,27,36,54,44,34,23,12)$ \\
			14 & -- & $(13, 37/2, 25, 36, 57/2, 20, 21/2)$ & $(20,29,39,58,47,36,25,13)$ \\
			15 & -- & $(14, 41/2, 27, 40, 63/2, 22, 23/2)$ & $(22,32,43,64,52,40,27,14)$ \\
			16 & -- & $(15, 45/2, 30, 44, 69/2, 24, 25/2)$ & $(24,36,48,71,58,44,30,15)$ \\
			17 & -- & $(17, 49/2, 33, 48, 75/2, 26, 27/2)$ & $(27,40,53,79,64,49,33,17)$ \\
			18 & -- & -- & $(28,41,55,81,66,50,34,17)$ \\
			19 & -- & -- & $(29,42,57,84,68,52,35,18)$ \\
			20 & -- & -- & $(30,44,59,87,71,54,37,19)$ \\
			21 & -- & -- & $(31,46,61,91,74,57,39,20)$ \\
			22 & -- & -- & $(32,48,64,95,78,60,41,21)$ \\
			23 & -- & -- & $(34,50,67,100,82,63,43,22)$ \\
			24 & -- & -- & $(36,53,71,105,86,66,45,23)$ \\
			25 & -- & -- & $(38,56,75,111,90,69,47,24)$ \\
			26 & -- & -- & $(40,59,79,117,95,72,49,25)$ \\
			27 & -- & -- & $(42,62,83,123,100,76,51,26)$ \\
			28 & -- & -- & $(44,65,87,129,105,80,54,27)$ \\
			29 & -- & -- & $(46,68,91,135,110,84,57,29)$ \\
			\bottomrule
		\end{tabular}
	\end{table}
	
	\begin{theorem}\label{Elongesty}
		Assume that $\kg$ is of type $\mathsf{E}$. 
		Let $\tilde{y}$ be the longest Weyl element in $\widehat{W}$ satisfying that
		$$\tilde{y}\circ(\Lambda )=\kappa\Lambda_0.$$
		\begin{enumerate}[wide]
			\item If $p\geq \mathbf{h}^\vee$, the Weyl element $\widetilde{y}=1$.
			\item If $p< \mathbf{h}^\vee$, the Weyl element $\widetilde{y}=t_\gamma w$ is given by \Cref{tab:allyE}, where $(c_1,c_2,\dots,c_\ell)$ denotes the weight $c_1\alpha_1+c_2\alpha_2+\cdots+c_\ell\alpha_\ell$, and the sequence $i_1,i_2,\dots,i_m$ in the column of $w$ stands for the product of simple reflections $s_{i_1}s_{i_2}\cdots s_{i_m}$.
		\end{enumerate}
	\end{theorem}
	
	{\footnotesize
	\begin{longtable}{c c c}
		\caption{$\widetilde{y}=t_\gamma w$ for type $\mathsf{E}$  ($p<\mathbf{h}^\vee$)}\label{tab:allyE}\\ 
		\toprule
		$p$ & $\gamma$ & $w$ \\
		\midrule
		\endfirsthead
		\multicolumn{3}{c}{{\tablename\ \thetable\ (continued)}}\\
		\toprule
		$p$ & $\gamma$ & $w$ \\
		\midrule
		\endhead
		\bottomrule
		\multicolumn{3}{r}{continued on next page}\\
		\endfoot
		\bottomrule
		\endlastfoot
		
		{$\mathsf{E}_6$}&&\\
		
		$2$  & $(4, 6, 8, 11, 8, 4)$ & $3, 2, 4, 5, 6, 4, 3, 2, 4, 5, 1, 3, 4, 2$ \\
		$3$  & $(3, 4, 6, 8, 6, 3)$ & $5, 3, 4, 1, 3, 2, 4, 5, 6, 5, 4, 1, 3, 2, 4, 5, 4, 1, 3, 2, 4$ \\
		$4$  & $(2, 3, 4, 6, 4, 2)$ & $4, 3, 2, 4, 5, 6, 4, 1, 3, 2, 4, 5, 4, 1, 3, 2, 4$ \\
		$5$  & $(2, 3, 4, 6, 4, 2)$ & $4, 5, 6, 2, 4, 5, 3, 4, 1, 3, 2, 4, 5, 6, 4, 3, 2, 4, 5, 4, 1, 3, 2, 4, 1$ \\
		$6$  & $(1, 2, 2, 3, 2, 1)$ & $2, 4, 5, 6, 3, 4, 5, 1, 3, 2, 4$ \\
		$7$  & $(1, 2, 2, 3, 2, 1)$ & $2, 4, 5, 6, 3, 4, 1, 3, 2, 4, 5, 4, 3, 2$ \\
		$8$  & $(1, 2, 2, 3, 2, 1)$ & $2, 4, 5, 3, 4, 1, 3, 2, 4, 5, 6, 5, 4, 3, 2, 4, 1$ \\
		$9$  & $(1, 2, 2, 3, 2, 1)$ & $2, 4, 5, 3, 4, 1, 3, 2, 4, 5, 6, 5, 4, 3, 2, 4, 5, 1, 3$ \\
		$10$ & $(1, 2, 2, 3, 2, 1)$ & $2, 4, 5, 3, 4, 1, 3, 2, 4, 5, 6, 5, 4, 3, 2, 4, 5, 1, 3, 4 $\\
		$11$ & $(1, 2, 2, 3, 2, 1)$ & $2, 4, 5, 3, 4, 1, 3, 2, 4, 5, 6, 5, 4, 3, 2, 4, 5, 1, 3, 4, 2$ \\
		\midrule
		
		{$\mathsf{E}_7$}&&\\
		
		$2$  & $(9,13,18,26,20,14,7)$ & $4,3,1,6,5,4,3,5,6,7,2,4,3,1,5,4,3,6,5,4,7,2,4,3,1,5,4,3,6,5,4,2$ \\
		$3$  & $(6,9,12,18,14,10,5)$ & $\makecell[t]{6,5,4,3,1,5,4,7,6,5,7,2,4,3,1,5,4,3,6,5,4,7,6,2,4,3,1,5,4,3,6,5,\\7,6,2,4,5}$ \\
		$4$  & $(4,6,8,12,9,6,3)$ & $4,3,1,5,4,6,2,4,3,1,5,4,3,6,5,7,6,2,4,5$ \\
		$5$  & $(4,6,8,12,9,6,3)$ & $\makecell[t]{4,3,1,5,4,3,6,5,7,2,4,3,1,5,4,3,6,5,4,7,6,5,2,4,3,1,5,4,3,6,5,4,\\7,6,2,4,3,5,4,6,7}$ \\
		$6$  & $(3,4,6,8,6,4,2)$ & $3,1,4,3,5,4,6,5,7,2,4,3,1,5,4,3,6,5,4,7,6,2,4,3,5,4$ \\
		$7$  & $(3,4,6,8,6,4,2)$ & $\makecell[t]{3,1,4,3,5,4,6,5,7,6,2,4,3,1,5,4,3,6,5,4,7,6,5,2,4,3,1,5,4,3,6,5,\\7,2,4}$ \\
		$8$  & $(2,3,4,6,5,4,2)$ & $6,5,4,3,1,7,6,5,4,3,2,4,3,1,5,4,3,6,5,4,7,6,2,4,3,5,4,6$ \\
		$9$  & $(2,3,4,6,5,4,2)$ & $\makecell[t]{6,5,4,3,1,7,6,5,4,3,2,4,3,1,5,4,3,6,5,4,7,6,5,2,4,3,1,5,4,3,6,5,\\2,4}$ \\
		$10$ & $(2,2,3,4,3,2,1)$ & $1,3,4,5,6,7,2,4,3,1,5,4,3,6,5,4,2,4,3,5$ \\
		$11$ & $(2,2,3,4,3,2,1)$ & $1,3,4,5,6,7,2,4,3,1,5,4,3,6,5,4,2,4,3,1,5,4,6$ \\
		$12$ & $(2,2,3,4,3,2,1)$ & $1,3,4,5,6,7,2,4,3,1,5,4,3,6,5,4,2,4,3,1,5,4,3,6,5,7$ \\
		$13$ & $(2,2,3,4,3,2,1)$ & $1,3,4,5,6,7,2,4,3,1,5,4,3,6,5,4,2,4,3,1,5,4,3,6,5,4,7,6$ \\
		$14$ & $(2,2,3,4,3,2,1)$ & $1,3,4,5,6,7,2,4,3,1,5,4,3,6,5,4,2,4,3,1,5,4,3,6,5,4,7,6,5,2$ \\
		$15$ & $(2,2,3,4,3,2,1)$ & $1,3,4,5,6,7,2,4,3,1,5,4,3,6,5,4,2,4,3,1,5,4,3,6,5,4,7,6,5,2,4$ \\
		$16$ & $(2,2,3,4,3,2,1)$ & $1,3,4,5,6,7,2,4,3,1,5,4,3,6,5,4,2,4,3,1,5,4,3,6,5,4,7,6,5,2,4,3$ \\
		$17$ & $(2,2,3,4,3,2,1)$ & $1,3,4,5,6,7,2,4,3,1,5,4,3,6,5,4,2,4,3,1,5,4,3,6,5,4,7,6,5,2,4,3,1$ \\
		\midrule
		
		{$\mathsf{E}_8$}&&\\
		
		$2$ & $(24,36,48,71,58,44,30,15)$ & $\makecell[t]{7,5,6,3,4,5,2,4,5,6,7,8,3,4,5,6,7,1,3,4,5,6,2,4,5,6,7,8,3,4,5,6,7,\\1,3,2,4,5,6,7,8,3,4,5,1,3,4,2,4,5,6,7,3,4,5,6,1,3,4,5,2,4,3,1}$\\
		$3$ & $(16,24,32,48,39,30,20,10)$ & $\makecell[t]{4,3,5,6,5,4,3,1,7,6,5,8,2,4,5,6,7,8,3,4,5,6,7,1,3,4,5,2,4,5,6,7,8,\\3,4,5,6,7,1,3,4,2,4,5,6,7,8,3,4,5,6,7,1,3,4,2,4,5,6,3,4,5,1,3,4,2}$\\
		$4$ & $(12,18,24,36,29,22,15,8)$ & $\makecell[t]{4,3,1,5,4,6,8,7,6,5,4,3,2,4,5,6,7,8,3,4,5,6,7,1,3,4,5,6,2,4,5,6,7,\\8,3,4,5,6,7,1,3,4,5,2,4,5,6,7,8,3,4,5,1,3,4,2,4,5,6,7,3,4,5,6}$\\
		$5$ & $(10,15,20,30,24,18,12,6)$ & $\makecell[t]{4,5,6,7,3,4,5,1,3,2,4,5,6,7,8,3,4,5,6,7,1,3,4,5,6,2,4,5,6,7,8,3,4,\\5,6,7,1,3,4,5,2,4,5,6,7,8,3,4,5,6,7,1,3,4,5,6,2,4,5,6,7,3,4,5,6,1,\\3,2,4,5}$\\
		$6$ & $(8,12,16,24,20,15,10,5)$ & $\makecell[t]{5,4,3,1,6,5,4,3,7,6,5,8,7,2,4,5,6,7,8,3,4,5,6,7,1,3,4,5,2,4,5,6,7,\\8,3,4,5,6,7,1,3,4,5,2,4,5,6,7,8,3,4,5,6,7,1,3,4,2,4,5,6,3,4,5}$\\
		$7$ & $(7,10,14,20,16,12,8,4)$ & $\makecell[t]{3,1,4,3,5,4,6,5,7,6,8,2,4,5,6,7,8,3,4,5,6,7,1,3,4,5,6,2,4,5,6,7,8,\\3,4,5,6,1,3,4,2,4,5,6,7,8,3,4,5,6,1,3,2,4,5}$\\
		$8$ & $(6,9,12,18,15,12,8,4)$ & $\makecell[t]{6,5,4,3,1,7,6,5,4,3,8,7,6,5,4,2,4,5,6,7,8,3,4,5,6,7,1,3,4,5,6,2,4,\\5,6,7,8,3,4,5,6,7,1,3,4,5,2,4,5,6,7,8,3,4,5,6,1,3,4,2,4,5,3,4}$\\
		$9$ & $(5,8,10,15,12,9,6,3)$ & $\makecell[t]{2,4,5,6,7,8,3,4,5,6,7,1,3,4,5,6,2,4,5,6,7,8,3,4,5,6,7,1,3,4,5,2,4,\\5,6,7,8,3,4,5,6,1,3,4,2,4,5,3,4}$\\
		$10$ & $(5,8,10,15,12,9,6,3)$ & $\makecell[t]{2,4,5,6,7,8,3,4,5,6,7,1,3,4,5,6,2,4,5,6,7,8,3,4,5,6,7,1,3,4,5,6,2,\\4,5,6,7,8,3,4,5,6,7,1,3,4,5,2,4,5,6,7,3,4,5,6,1,3,4,2,4,5,3,4}$\\
		$11$ & $(5,8,10,15,12,9,6,3)$ & $\makecell[t]{2,4,5,6,7,8,3,4,5,6,7,1,3,4,5,6,2,4,5,6,7,8,3,4,5,6,7,1,3,4,5,6,2,\\4,5,6,7,8,3,4,5,6,7,1,3,4,5,6,2,4,5,6,7,8,3,4,5,6,1,3,4,5,2,4,5,6,\\7,3,4,5,1,3,2,4}$\\
		$12$ & $(4,5,7,10,8,6,4,2)$ & $\makecell[t]{1,3,4,5,6,7,8,2,4,5,6,7,3,4,5,6,1,3,4,5,2,4,5,6,7,8,3,4,5,6,1,3,4,\\2,4,5,6,7,3,4}$\\
		$13$ & $(4,5,7,10,8,6,4,2)$ & $\makecell[t]{1,3,4,5,6,7,8,2,4,5,6,7,3,4,5,6,1,3,4,5,2,4,5,6,7,8,3,4,5,6,7,1,3,\\4,5,2,4,5,6,7,8,3,4,5,6,2,4}$\\
		$14$ & $(4,5,7,10,8,6,4,2)$ & $\makecell[t]{1,3,4,5,6,7,8,2,4,5,6,7,3,4,5,6,1,3,4,5,2,4,5,6,7,8,3,4,5,6,7,1,3,\\4,5,6,2,4,5,6,7,8,3,4,5,6,7,1,2,4,5,3,4,1}$\\
		$15$ & $(4,5,7,10,8,6,4,2)$ & $\makecell[t]{1,3,4,5,6,7,8,2,4,5,6,7,3,4,5,6,1,3,4,5,2,4,5,6,7,8,3,4,5,6,7,1,3,\\4,5,6,2,4,5,6,7,8,3,4,5,6,7,1,3,2,4,5,6,3,4,5,1,3,2,4}$\\
		$16$ & $(4,5,7,10,8,6,4,2)$ & $\makecell[t]{1,3,4,5,6,7,8,2,4,5,6,7,3,4,5,6,1,3,4,5,2,4,5,6,7,8,3,4,5,6,7,1,3,\\4,5,6,2,4,5,6,7,8,3,4,5,6,7,1,3,2,4,5,6,7,3,4,5,1,3,4,2,4,5,3}$\\
		$17$ & $(4,5,7,10,8,6,4,2)$ & $\makecell[t]{1,3,4,5,6,7,8,2,4,5,6,7,3,4,5,6,1,3,4,5,2,4,5,6,7,8,3,4,5,6,7,1,3,\\4,5,6,2,4,5,6,7,8,3,4,5,6,7,1,3,2,4,5,6,7,8,3,4,5,1,3,4,2,4,5,6,3,\\4,1}$\\
		$18$ & $(2,3,4,6,5,4,3,2)$ & $\makecell[t]{8,7,6,5,4,3,1,2,4,5,6,7,8,3,4,5,6,7,2,4,5,6,3,4,5,1,3,4,2,4,5,6,7,\\3,4,5,1,3}$\\
		$19$ & $(2,3,4,6,5,4,3,2)$ & $\makecell[t]{8,7,6,5,4,3,1,2,4,5,6,7,8,3,4,5,6,7,2,4,5,6,3,4,5,1,3,4,2,4,5,6,7,\\8,3,4,5,6,1,3,4}$\\
		$20$ & $(2,3,4,6,5,4,3,2)$ & $\makecell[t]{8,7,6,5,4,3,1,2,4,5,6,7,8,3,4,5,6,7,2,4,5,6,3,4,5,1,3,4,2,4,5,6,7,\\8,3,4,5,6,7,1,3,4,5,2}$\\
		$21$ & $(2,3,4,6,5,4,3,2)$ & $\makecell[t]{8,7,6,5,4,3,1,2,4,5,6,7,8,3,4,5,6,7,2,4,5,6,3,4,5,1,3,4,2,4,5,6,7,\\8,3,4,5,6,7,1,3,4,5,6,2,4}$\\
		$22$ & $(2,3,4,6,5,4,3,2)$ & $\makecell[t]{8,7,6,5,4,3,1,2,4,5,6,7,8,3,4,5,6,7,2,4,5,6,3,4,5,1,3,4,2,4,5,6,7,\\8,3,4,5,6,7,1,3,4,5,6,2,4,5,3}$\\
		$23$ & $(2,3,4,6,5,4,3,2)$ & $\makecell[t]{8,7,6,5,4,3,1,2,4,5,6,7,8,3,4,5,6,7,2,4,5,6,3,4,5,1,3,4,2,4,5,6,7,\\8,3,4,5,6,7,1,3,4,5,6,2,4,5,3,4,1}$\\
		$24$ & $(2,3,4,6,5,4,3,2)$ & $\makecell[t]{8,7,6,5,4,3,1,2,4,5,6,7,8,3,4,5,6,7,2,4,5,6,3,4,5,1,3,4,2,4,5,6,7,\\8,3,4,5,6,7,1,3,4,5,6,2,4,5,3,4,1,3,2}$\\
		$25$ & $(2,3,4,6,5,4,3,2)$ & $\makecell[t]{8,7,6,5,4,3,1,2,4,5,6,7,8,3,4,5,6,7,2,4,5,6,3,4,5,1,3,4,2,4,5,6,7,\\8,3,4,5,6,7,1,3,4,5,6,2,4,5,3,4,1,3,2,4}$\\
		$26$ & $(2,3,4,6,5,4,3,2)$ & $\makecell[t]{8,7,6,5,4,3,1,2,4,5,6,7,8,3,4,5,6,7,2,4,5,6,3,4,5,1,3,4,2,4,5,6,7,\\8,3,4,5,6,7,1,3,4,5,6,2,4,5,3,4,1,3,2,4,5}$\\
		$27$ & $(2,3,4,6,5,4,3,2)$ & $\makecell[t]{8,7,6,5,4,3,1,2,4,5,6,7,8,3,4,5,6,7,2,4,5,6,3,4,5,1,3,4,2,4,5,6,7,\\8,3,4,5,6,7,1,3,4,5,6,2,4,5,3,4,1,3,2,4,5,6}$\\
		$28$ & $(2,3,4,6,5,4,3,2)$ & $\makecell[t]{8,7,6,5,4,3,1,2,4,5,6,7,8,3,4,5,6,7,2,4,5,6,3,4,5,1,3,4,2,4,5,6,7,\\8,3,4,5,6,7,1,3,4,5,6,2,4,5,3,4,1,3,2,4,5,6,7}$\\
		$29$ & $(2,3,4,6,5,4,3,2)$ & $\makecell[t]{8,7,6,5,4,3,1,2,4,5,6,7,8,3,4,5,6,7,2,4,5,6,3,4,5,1,3,4,2,4,5,6,7,\\8,3,4,5,6,7,1,3,4,5,6,2,4,5,3,4,1,3,2,4,5,6,7,8}$\\

	\end{longtable}
	}
	
	\begin{theorem}\label{Elongestz}
		Assume that $\kg$ is of type $\mathsf{E}$. Let $\tilde{z}$  be the longest Weyl element in $\widehat{W}$ satisfying that
		$$\tilde{z}\circ{\Lambda}=\Lambda_{sing},$$
		where $\Lambda_{sing}$ is a weight of minimal singular vectors given in \Cref{singE}.
		Then $\widetilde{z}=\widetilde{z_0}\widetilde{y}$, where $\widetilde{z_0}$ is given as follows: 
		\begin{enumerate}[wide]
			\item If $p\geq \mathbf{h}^\vee$, the Weyl element $\widetilde{z_0}=s_0$.
			\item If $p< \mathbf{h}^\vee$, the Weyl element $\widetilde{z_0}=wt_\gamma$ is given by \Cref{tab:allz0E}, where $(c_1,c_2,\dots,c_\ell)$ denotes the weight $c_1\alpha_1+c_2\alpha_2+\cdots+c_\ell\alpha_\ell$, and the sequence $i_1,i_2,\dots,i_m$ in the column of $w$ stands for the product of simple reflections $s_{i_1}s_{i_2}\cdots s_{i_m}$.
		\end{enumerate}
	\end{theorem}
	
	{\footnotesize
		\begin{longtable}{c c c}
			\caption{$\widetilde{z_0}=w t_\gamma$ for type $\mathsf{E}$  ($p<\mathbf{h}^\vee$)}\label{tab:allz0E}\\ 
			\toprule
			$p$ & $w$ & $\gamma$ \\
			\midrule
			\endfirsthead
			\multicolumn{3}{c}{{\tablename\ \thetable\ (continued)}}\\
			\toprule
			$p$ & $w$ & $\gamma$ \\
			\midrule
			\endhead
			\bottomrule
			\multicolumn{3}{r}{continued on next page}\\
			\endfoot
			\bottomrule
			\endlastfoot
			
			{$\mathsf{E}_6$}&&\\
			
			$2$ & $2$ & $(1, 1, 2, 3, 2, 1)$ \\
			$3$ & $1, 2, 3, 6, 5$ & $(0, 0, 0, 1, 0, 0)$ \\
			$4$ & $3, 4, 5, 4, 2, 3, 4$ & $(1, 0, 1, 0, 1, 1)$ \\
			$5$ & $1, 5, 4, 2, 6, 5, 4$ & $(0, 0, 1, 0, 0, 0)$ \\
			$5$ & $3, 1, 4, 2, 3, 4, 6$ & $(0, 0, 0, 0, 1, 0)$ \\
			$6$ & $1, 4, 2, 3, 4, 6, 5, 4, 3$ & $(0, 0, -1, -1, -1, 0)$ \\
			$7$ & $4, 3, 1, 5, 4, 2, 3, 6, 5$ & $(0, 0, 0, 1, 0, 0)$ \\
			$8$ & $3, 1, 5, 4, 2, 3, 4, 6, 5, 4, 2, 3, 4$ & $(0, -1, -1, -2, -1, 0)$ \\
			$9$ & $3, 1, 5, 4, 2, 3, 4, 5, 6, 5, 4, 2, 3, 1, 4$ & $(-1, -1, -1, -2, -1, -1)$ \\
			$10$ & $4, 3, 1, 5, 4, 2, 3, 4, 5, 6, 5, 4, 2, 3, 1$ & $(-1, -1, -1, -1, -1, -1)$ \\
			$11$ & $2, 4, 3, 1, 5, 4, 2, 3, 4, 5, 6, 5, 4, 3, 1$ & $(-1, 0, -1, -1, -1, -1)$ \\
			\midrule
			
			{$\mathsf{E}_7$}&&\\
			
			$2$ & ${7,5,3}$ & $(1,1,1,2,1,1,0)$\\
			$3$ & ${2,5,6}$ & $(1,1,2,3,2,1,1)$\\
			$4$ & ${4,2,3,4,5,6,7,1,3}$ & $(0,1,0,2,1,0,-1)$\\
			$5$ & ${5,1,3,4,2,6,7,3,4,1,3}$ & $(0,0,0,1,2,1,0)$\\
			$6$ & ${3,4,2,7,6,1,3,4,5}$ & $(0,0,0,0,-1,0,0)$\\
			$7$ & ${5,4,2,6,7,5,6,4,5,3,4,5,6,7,1,3,4,5,6}$ & $(-1,0,-2,-2,-2,-2,-1)$\\
			$8$ & ${5,6,7,4,5,6,3,4,5,1,3,4,2}$ & $(0,-1,0,0,0,0,0)$\\
			$9$ & ${7,6,5,1,3,4,2,6,7,5,6,3,4,5,1,3,4,2,6,5,1,3,4}$ & $(-1,-2,-2,-4,-2,-1,0)$\\
			$10$ & ${6,7,5,6,3,4,5,1,3,4,2,4,5,6,7,3,4}$ & $(0,-1,-1,-2,-1,-1,-1)$\\
			$11$ & ${5,6,7,4,5,6,3,4,5,1,3,4,2,4,5,6,7}$ & $(0,-1,0,-1,-1,-1,-1)$\\
			$12$ & ${2,6,7,4,5,6,3,4,5,1,3,4,2,4,5,6,3,4,5,1,3}$ & $(-1,-1,-2,-2,-2,-1,0)$\\
			$13$ & ${4,2,5,6,7,4,5,6,3,4,5,1,3,4,2,4,5,3,4,1,3}$ & $(-1,-1,-2,-2,-1,0,0)$\\
			$14$ & ${2,5,6,7,4,5,6,3,4,5,1,3,4,2,4,5,6,7,3,4,5,6,1,3,4}$ & $(-1,-1,-2,-3,-2,-2,-1)$\\
			$15$ & ${4,2,5,6,7,4,5,6,3,4,5,1,3,4,2,4,5,6,7,3,4,5,6,1,3}$ & $(-1,-1,-2,-2,-2,-2,-1)$\\
			$16$ & ${3,4,2,5,6,7,4,5,6,3,4,5,1,3,4,2,4,5,6,7,3,4,5,6,1}$ & $(-1,-1,-1,-2,-2,-2,-1)$\\
			$17$ & ${1,3,4,2,5,6,7,4,5,6,3,4,5,1,3,4,2,4,5,6,7,3,4,5,6}$ & $(0,-1,-1,-2,-2,-2,-1)$\\
			\midrule
			
			{$\mathsf{E}_8$}&&\\
			
			$2$ & ${8}$ & $(2,3,4,6,5,4,3,1)$\\
			$3$ & ${8}$ & $(2,3,4,6,5,4,3,1)$\\
			$4$ & ${2,4,7,8,6,7,3}$ & $(1,1,1,2,2,1,0,0)$\\
			$5$ & ${1,3,4,2,6,7,8}$ & $(1,1,2,3,3,2,1,0)$\\
			$6$ & ${5,4,2,6,5,4,3,1,7,6,5,4,2,3,4,8,7,6,5}$ & $(0,0,0,0,-2,0,0,0)$\\
			$7$ & ${5,4,2,6,7,8,5,6,7,4,5}$ & $(1,1,2,2,1,1,0,0)$\\
			$8$ & ${3,6,5,4,2,3,4,5,6,7,8,6,5,4,3,1,7,6,5,4,3,4,5}$ & $(-1,0,-2,-2,-2,-1,-1,-1)$\\
			$9$ & ${3,1,4,3,5,4,2,6,7,8,5,6,7}$ & $(1,1,2,3,2,1,0,0)$\\
			$10$ & ${3,1,5,4,3,7,6,5,4,2,3,1,7,6,5,4,3,8,7,6,5,4,2,6,5,4,7}$ & $(0,-2,-1,-3,-2,-1,-1,0)$\\
			$11$ & ${5,4,2,3,1,4,3,5,4,2,6,7,8}$ & $(1,1,2,3,3,2,1,0)$\\
			$12$ & ${2,4,3,1,7,6,5,4,3,8,7,6,5,4,2,4,3,5,4,6,5}$ & $(0,-1,-1,-2,-2,-1,0,0)$\\
			$13$ & ${2,4,3,5,4,2,6,5,4,3,7,6,5,4,2,1,8,7,6,5,4}$ & $(0,0,1,0,0,0,0,0)$\\
			$14$ & ${5,4,2,6,5,4,3,1,7,6,5,4,3,8,7,6,5,4,2,4,3,1,5,4,3}$ & $(-1,-1,-2,-2,-1,0,0,0)$\\
			$15$ & $\makecell[t]{5,4,2,4,3,1,5,4,3,8,7,6,5,4,2,3,1,7,6,5,4,3,8,7,6,5,4,2,6,5,4,3,1,\\7,6,5,4,3,4,5,8,7,6}$ & $(-2,-2,-4,-5,-4,-4,-2,-1)$\\
			$16$ & ${2,4,3,5,4,2,6,5,4,3,1,7,6,5,4,3,8,7,6,5,4,2,4,3,1}$ & $(-1,-1,-1,-1,0,0,0,0)$\\
			$17$ & ${5,4,2,3,1,4,3,5,4,2,6,5,4,3,1,7,6,5,4,3,8,7,6,5,4}$ & $(0,1,0,0,0,0,0,0)$\\
			$18$ & ${1,3,4,2,6,5,4,3,1,7,6,5,4,3,8,7,6,5,4,2,4,3,1,5,4,3,6,5,4,2,7,6,5}$ & $(-1,-2,-2,-3,-3,-2,-1,0)$\\
			$19$ & ${3,1,4,3,5,4,2,6,5,4,3,1,7,6,5,4,3,8,7,6,5,4,2,4,3,1,5,4,3,6,5,4,2}$ & $(-1,-2,-2,-3,-2,-1,0,0)$\\
			$20$ & $\makecell[t]{1,3,5,4,2,6,5,4,3,1,7,6,5,4,3,8,7,6,5,4,2,4,3,1,5,4,3,6,5,4,2,7,6,\\5,4,8,7}$ & $(-1,-2,-2,-4,-3,-2,-2,-1)$\\
			$21$ & $\makecell[t]{3,1,4,3,5,4,2,6,5,4,3,1,7,6,5,4,3,8,7,6,5,4,2,4,3,1,5,4,3,6,5,4,2,\\7,6,8,7}$ & $(-1,-2,-2,-3,-2,-2,-2,-1)$\\
			$22$ & $\makecell[t]{4,2,3,1,4,3,5,4,2,6,5,4,3,1,7,6,5,4,3,8,7,6,5,4,2,4,3,1,5,4,3,6,5,\\7,6,8,7}$ & $(-1,-1,-2,-2,-2,-2,-2,-1)$\\
			$23$ & $\makecell[t]{6,5,4,2,3,1,4,3,5,4,2,6,5,4,3,1,7,6,5,4,3,8,7,6,5,4,2,4,3,5,4,6,5,\\7,6,8,7}$ & $(0,-1,-1,-2,-2,-2,-2,-1)$\\
			$24$ & $\makecell[t]{2,3,1,4,3,5,4,2,6,5,4,3,1,7,6,5,4,3,8,7,6,5,4,2,4,3,1,5,4,3,6,5,4,\\2,7,6,5,4,3,1,8,7,6,5,4}$ & $(-2,-2,-3,-5,-4,-3,-2,-1)$\\
			$25$ & $\makecell[t]{4,2,3,1,4,3,5,4,2,6,5,4,3,1,7,6,5,4,3,8,7,6,5,4,2,4,3,1,5,4,3,6,5,\\4,2,7,6,5,4,3,1,8,7,6,5}$ & $(-2,-2,-3,-4,-4,-3,-2,-1)$\\
			$26$ & $\makecell[t]{5,4,2,3,1,4,3,5,4,2,6,5,4,3,1,7,6,5,4,3,8,7,6,5,4,2,4,3,1,5,4,3,6,\\5,4,2,7,6,5,4,3,1,8,7,6}$ & $(-2,-2,-3,-4,-3,-3,-2,-1)$\\
			$27$ & $\makecell[t]{6,5,4,2,3,1,4,3,5,4,2,6,5,4,3,1,7,6,5,4,3,8,7,6,5,4,2,4,3,1,5,4,3,\\6,5,4,2,7,6,5,4,3,1,8,7}$ & $(-2,-2,-3,-4,-3,-2,-2,-1)$\\
			$28$ & $\makecell[t]{7,6,5,4,2,3,1,4,3,5,4,2,6,5,4,3,1,7,6,5,4,3,8,7,6,5,4,2,4,3,1,5,4,\\3,6,5,4,2,7,6,5,4,3,1,8}$ & $(-2,-2,-3,-4,-3,-2,-1,-1)$\\
			$29$ & $\makecell[t]{8,7,6,5,4,2,3,1,4,3,5,4,2,6,5,4,3,1,7,6,5,4,3,8,7,6,5,4,2,4,3,1,5,\\4,3,6,5,4,2,7,6,5,4,3,1}$ & $(-2,-2,-3,-4,-3,-2,-1,0)$\\
			
		\end{longtable}
	}
	
	\appendix
	
	\section{Discussion on Minimum Values}
	
	Let $N,p,M\in \mathbb{Z}_{\geq 1}$. 
	Define a function $F_{N,p,M}(s)$ on $\mathbb{Z}_{\geq 1}$ by 
	$$F_{N,p,M}(s)=(|N-sp|+M)s.$$
	Set $s_1=\lfloor N/p\rfloor$ and $s_2=\lceil N/p\rceil$.
	
	\begin{lemma}\label{Appen0}
		Assume that $N\geq p$ and $N\geq M$. Set $F=F_{N,p,M}$, and  \(F_{\min} = \min\{F(s) : s \in \mathbb{Z}_{\geq 1}\}\). Then we have 
		$$F_{\min}=\min\{F(1),F(s_1),F(s_2)\}$$
		and $F(s)>F_{\min}$ for $s\neq 1,s_1,s_2$.
	\end{lemma}
	\begin{proof}
		Note that the function 
		\begin{align}\label{eq:F-piecewise}
			F(s)=\begin{cases}
				(N-ps+M)s,& s=1,2,\dots,\lfloor N/p \rfloor;\\
				(ps-N+M)s,& s=\lfloor N/p \rfloor+1,\lfloor N/p \rfloor+2,\dots.
			\end{cases}
		\end{align}
		So $F$ attains its minimum only when $s=1$, $s=s_1$, or $s=s_2$. 
	\end{proof}
	
	\begin{lemma}\label{Appen1}
		Assume that $N>p\geq 3$ and $(p,N)\neq (3,8)$. Set $F=F_{N,p,1}$. Then we have 
		$$F_{\min}=\min\{F(s_1),F(s_2)\}$$
		and $F(s)>F_{\min}$ for $s\neq s_1,s_2$.
	\end{lemma}
	\begin{proof}
		By \Cref{Appen0}, it suffices to show that $F(1)> \min\{F(s_1),F(s_2)\}$ holds when $s_1\geq 2$.
		
		Set $N_0:=N-ps_1$. When $N_0=0$, we have $F(s_1)=F(s_2)=s_1$, while 
		$$F(1)=s_1p-p+1=(s_1-1)(p-1)+s_1>s_1.$$
		When $N_0\neq 0$, we have $s_2=s_1+1$, and hence
		\begin{align*}
			&F(1)=ps_1+N_0-p+1,\\
			&F(s_1)=(N_0+1)s_1,\\
			&F(s_2)=(p-N_0+1)(s_1+1).
		\end{align*}
		Then we show that $F(1)> \min\{F(s_1),F(s_2)\}$ holds:
		\begin{enumerate}
			\item Assume that $F(s_1)<F(s_2)$ while $F(s_1)\geq F(1)$. Then we have 
			$$0\geq F(1)-F(s_1)=(p-1-N_0)(s_1-1),$$
			which gives $N_0=p-1$. So $F(s_1)=ps_1$ and $F(s_2)=2(s_1+1)$.
			Then 
			$$F(s_1)-F(s_2)=s_1(p-2)-2\geq 2(3-2)-2=0,$$
			contradicting $F(s_1)<F(s_2)$.
			\item  Assume that $F(s_1)>F(s_2)$ while $F(s_2)\geq F(1)$. So
			$$0\geq F(1)-F(s_2)=N_0s_1+2N_0-2p-s_1,$$
			$$0<F(s_1)-F(s_2)=2N_0s_1+N_0-ps_1-p-1.$$
			Then we have
			$$ps_1+p-N_0+1<2N_0s_1\leq 2s_1+4p-4N_0,$$
			and hence
			$$(p-2)(s_1-3)<5-3N_0.$$
			
			If $s_1=2$, then note that
			\begin{align*}
				0<&2(F(s_1)-F(s_2))+3(F(s_2)-F(1))\\
				=&2F(s_1)+F(s_2)-3F(1)=-2N_0+4,
			\end{align*}
			which gives $N_0=1$. Then $F(s_1)=4$ and $F(s_2)=3p$, which contradicts $F(s_1)>F(s_2)$.
			
			If $s_1\geq 3$, then $5-3N_0>(p-2)(s_1-3)\geq 0$, which gives $N_0=1$. Then we have 
			$$F(s_1)-F(s_2)=6-(p-2)(s_1-3)-4p\leq 6-4p<0,$$
			which also contradicts $F(s_1)>F(s_2)$.
			\item Assume that $F(s_1)=F(s_2)$ while $F(s_1)\geq F(1)$. 
			As in the first case, we can obtain $N_0=p-1$ from $F(s_1)\geq F(1)$. 
			Then by $F(s_1)=F(s_2)$, we have
			$ps_1=2(s_1+1)$, i.e., $p=2+\frac{2}{s_1}$. Recall that $s_1\geq 2$, so the only possible value is $p=3$ when $s_1=2$. However, it follows that $N=ps_1+N_0=8$ with $p=3$, which is the case excluded.\qedhere
		\end{enumerate}
	\end{proof}
	
	\begin{lemma}\label{AppenH}
		Assume that $N>p\geq 2$ with $N$ odd. Set $F=F_{N,p,2}$ and $G=F_{N,p,1}/2$. 
		Define
		$$H(s)=\begin{cases}
			\min\{F(s),G(s)\}, &s\leq s_1,2|p,2\nmid s,\\
			\min\{F(s),G(s)\}, &s\geq s_2,2|s,\\
			F(s), &\text{others}.
		\end{cases}$$
		Then we have
		\begin{enumerate}[wide] 
			\item 
			$H_{\min}=\min\{H(1),H(s_1),H(s_2)\}$ 
			and $H(s)>H_{\min}$ for $s\neq 1,s_1,s_2$;
			\item $H_{\min}\leq N$ if $2\nmid p$, and $H_{\min}\leq (N-1)/2$ if $2|p$;
			\item If $s_1>1$ and $(p,N)\notin \{(5,13),(5,23)\}\cup\{(3,6m+1):m\in\mathbb{Z}_{\geq 1}\}$, then $H_{\min}<F(1)$.
		\end{enumerate}
	\end{lemma}
	
	\begin{proof}
		\begin{enumerate}[wide]
			\item It is clear that $G(s)<F(s)$ for any $s\in\mathbb{Z}_{\geq 1}$, 
			and hence $\min\{F(s),G(s)\}=G(s)$. 
			By \Cref{Appen0}, it suffices to show that $G(s)>H_{\min}$ for any $s\in S$, where
			$$S:=\{s:1<s<s_1,2|p,2\nmid s\}\cup \{s:s>s_2,2|s\}.$$
			\begin{enumerate}
				\item Assume that $2\nmid p$ and $2|s_2$. Then
				$S=\{s:s\geq s_2+2,2|s\}$ and $H_{\min}\leq H(s_2)=G(s_2)$.
				By \eqref{eq:F-piecewise}, we have $G(s_2)<G(s)$ for any $s\in S$.
				\item Assume that $2\nmid p$ and $2\nmid s_2$. Then 
				$S=\{s:s\geq s_2+1,2|s\}$.  
				By \eqref{eq:F-piecewise}, we have $G(s_2+1)\leq G(s)$ for any $s\in S$. Then we show that $G(s_2+1)>F_{\min}$.
				
				Set $N_0':=ps_2-N$. Then $F(s_2)=(N_0'+2)s_2$ and $G(s_2+1)=(N_0'+p+1)(s_2+1)/2$. 
				
				If $N_0'\leq p-2$, we have $G(s_2+1)>F(s_2)\geq F_{\min}$ by
				$$2(G(s_2+1)-F(s_2))=(p-N_0'-3)s_2+(p+N_0'+1)>0.$$
				If $N_0'=p-1$, then $F(s_1)=3s_1=3(s_2-1)$ and $G(s_2+1)=p(s_2+1)$.
				By $2\nmid p$, we have $p\geq 3$ and hence $G(s_2+1)>F(s_1)\geq F_{\min}$.
				\item Assume that $2|p$ and $2\nmid s_1$. Note that $s_2\neq s_1$; otherwise, $N=s_1p=s_2p$ is even, which is a contradiction. So $s_2=s_1+1$. 
				Then $H(1)=G(1)$, $H(s_1)=G(s_1)$ and $H(s_2)=G(s_2)$.
				It follows that $H_{\min}\leq H(s_2)=G(s_2)<G(s)$ for any $s\in \{s:s>s_2,2|s\}$. 
				
				Now it suffices to show that $H_{\min}<G(s)$ for any $s\in \{s:1<s<s_1,2\nmid s\}$. Note that $H(1)=G(1)$ and $H(s_1)=G(s_1)$, so $H_{\min}\leq \min\{G(1),G(s_1)\}$.
				By \eqref{eq:F-piecewise}, we have $\min\{G(1),G(s_1)\}<G(s)$ for any $s\in \{s:1<s<s_1\}$. 
				\item Assume that $2|p$ and $2|s_1$. Analogously, we have $s_2=s_1+1$. So $H(1)=G(1)$, $H(s_1)=F(s_1)$ and $H(s_2)=F(s_2)$.
				
				For any $s\in \{s:1<s\leq s_1-1,2\nmid s\}$, 
				by \eqref{eq:F-piecewise}, we have $G(1)< G(s)$ since $\tfrac{N+1}{p}-(s_1-1)>1$. 
				Then it follows that $H_{\min}\leq H(1)=G(1)<G(s)$.
				
				By \eqref{eq:F-piecewise}, we have $G(s_2+1)\leq G(s)$ for any $s\in \{s:s>s_2,2|s\}$. Set $N_0':=ps_2-N$. Analogously to the case with $2\nmid p$ and $2\nmid s_2$, we can obtain that $G(s_2+1)>F(s_1)$ when $N_0'\leq p-2$, or $N_0'=p-1$ and $p\geq 3$. 
				When $p=2$ and $N_0'=1$, it follows that $G(s_2+1)=2(s_2+1)$ and $G(1)=s_2-1$, so $G(1)<G(s_2+1)$. Therefore, we have shown that $H_{\min}<G(s_2+1)\leq G(s)$ for any $s\in \{s:s>s_2,2|s\}$. 
			\end{enumerate}
			\item If $2\nmid p$, then $H_{\min}\leq F(1)=N-p+2\leq N$.
			
			If $2|p$, then $2H_{\min}\leq 2G(1)=N-p+1\leq N-1$.
			
			\item If $2| p$,  (3) is true,  since $H_{\min}\leq H(1)=G(1)<F(1)$. Assume  now that $2\nmid p$.
			
			When $N-s_1p<p-2$, by \eqref{eq:F-piecewise}, we have $F(s_1)<F(1)$ since $\frac{N+2}{p}-s_1<1\leq \frac{N+2}{2p}$. 
			
			When $N-s_1p=p-2$, $s_1=s_2-1$ is even.
			Then $F(1)=s_1p$ and $H(s_2)=F(s_2)=4s_2$. 
			Note that $H(s_2)<F(1)$ if and only if $(p-4)s_1>4$. 
			Since $p$ is odd and $s_1$ is even, the inequality does not hold only when $(p,s_1)\in \{(5,2),(5,4)\}\cup\{(3,2m):m\in\mathbb{Z}_{\geq 1}\}$, that is, $(p,N)\in \{(5,13),(5,23)\}\cup\{(3,6m+1):m\in\mathbb{Z}_{\geq 1}\}$.
			
			When $N-s_1p=p-1$, $s_2=s_1+1$ is even. Then $F(1)=s_1p+1$ and $H(s_2)=G(s_2)=s_2$. It follows that $H_{\min}\leq H(s_2)<F(1)$.\qedhere
		\end{enumerate}
	\end{proof}
	
	\section*{Statements and Declarations}
	
	\noindent\textbf{Funding.}
	This work was supported by the National Natural Science Foundation of China
	(Grant Nos. 12171312 and 12571028).
	
	\medskip
	\noindent\textbf{Competing interests.}
	The authors have no relevant financial or non-financial interests to disclose.
	
	\medskip
	\noindent\textbf{Data availability.}
	All data supporting the findings of this study are available within the article.

	\bigskip
	\begingroup
	\small
	\hypersetup{urlcolor=black}
	\noindent\textsc{School of Mathematical Sciences, Shanghai Jiao Tong University,}\\
	\textsc{Shanghai 200240, China}\\
	\textit{Email address}: \href{mailto:cpjiang@sjtu.edu.cn}{\texttt{cpjiang@sjtu.edu.cn}}
	\par\medskip
	\noindent\textsc{School of Mathematical Sciences, Shanghai Jiao Tong University,}\\
	\textsc{Shanghai 200240, China}\\
	\textit{Email address}: \href{mailto:dailycrazy@sjtu.edu.cn}{\texttt{dailycrazy@sjtu.edu.cn}}
	\par
	\endgroup


\begin{thebibliography}{AKMPP20}
		
		\bibitem[AKMPP20]{AKMPP20}
		Dra{\v{z}}en Adamovi{\'c}, Victor~G. Kac, Pierluigi M{\"o}seneder Frajria,
		Paolo Papi, and Ozren Per{\v{s}}e.
		\newblock An application of collapsing levels to the representation theory of
		affine vertex algebras.
		\newblock \emph{International Mathematics Research Notices},
		2020(13):4103--4143, 2020.
		
		\bibitem[AP08]{AP08}
		Dra{\v{z}}en Adamovi{\'c} and Ozren Per{\v{s}}e.
		\newblock Representations of certain non-rational vertex operator algebras of
		affine type.
		\newblock \emph{Journal of Algebra}, 319(6):2434--2450, 2008.
		
		\bibitem[APV23]{APV23}
		Dra{\v{z}}en Adamovi{\'c}, Ozren Per{\v{s}}e, and Ivana Vukorepa.
		\newblock On the representation theory of the vertex algebra
		$L_{-5/2}(\mathfrak{sl}_4)$.
		\newblock \emph{Communications in Contemporary Mathematics}, 25(2),
		Article 2150104, 2023.
		
		\bibitem[ADFLM25]{ADFLM25}
		Tomoyuki Arakawa, Xuanzhong Dai, Justine Fasquel, Bohan Li, and Anne Moreau.
		\newblock On some simple orbifold affine {VOA}s at non-admissible level
		arising from rank one 4{D} {SCFT}s.
		\newblock \emph{Communications in Mathematical Physics}, 406,
		Article 30, 2025.
		
		\bibitem[AJM21]{AJM21}
		Tomoyuki Arakawa, Cuipo Jiang, and Anne Moreau.
		\newblock Simplicity of vacuum modules and associated varieties.
		\newblock \emph{Journal de l'\'Ecole polytechnique~--- Math\'ematiques},
		8:169--191, 2021.
		
		\bibitem[AM17]{AM17}
		Tomoyuki Arakawa and Anne Moreau.
		\newblock Sheets and associated varieties of affine vertex algebras.
		\newblock \emph{Advances in Mathematics}, 320:157--209, 2017.
		
		\bibitem[DGK82]{DGK82}
		Vinay~V. Deodhar, Ofer Gabber, and Victor~G. Kac.
		\newblock Structure of some categories of representations of
		infinite-dimensional Lie algebras.
		\newblock \emph{Advances in Mathematics}, 45(1):92--116, 1982.
		
		\bibitem[Fie06]{fiebig2006combinatorics}
		Peter Fiebig.
		\newblock The combinatorics of category $\mathcal{O}$ over symmetrizable
		Kac--Moody algebras.
		\newblock \emph{Transformation Groups}, 11(1):29--49, 2006.
		
		\bibitem[FZ92]{FZ92}
		Igor~B. Frenkel and Yongchang Zhu.
		\newblock Vertex operator algebras associated to representations of affine and
		Virasoro algebras.
		\newblock \emph{Duke Mathematical Journal}, 66(1):123--168, 1992.
		
		\bibitem[GK07]{gorelik2007simplicity}
		Maria Gorelik and Victor~G. Kac.
		\newblock On simplicity of vacuum modules.
		\newblock \emph{Advances in Mathematics}, 211(2):621--677, 2007.
		
		\bibitem[Hu72]{Humphreys}
		James~E. Humphreys.
		\newblock \emph{Introduction to Lie Algebras and Representation Theory}.
		\newblock Graduate Texts in Mathematics, volume~9.
		Springer-Verlag, New York, 1972.
		
		\bibitem[Jan77]{jantzen1977kontravariante}
		Jens~C. Jantzen.
		\newblock Kontravariante Formen auf induzierten Darstellungen halbeinfacher
		Lie-Algebren.
		\newblock \emph{Mathematische Annalen}, 226(1):53--65, 1977.
		
		\bibitem[JS25]{JS25}
		Cuipo Jiang and Jingtian Song.
		\newblock Associated varieties of simple affine {VOA}s
		$L_k(\mathfrak{sl}_3)$ and $W$-algebras
		$W_k(\mathfrak{sl}_3,f)$.
		\newblock \emph{Communications in Mathematical Physics}, 406,
		Article 104, 2025.
		
		\bibitem[Kac90]{kac1990infinite}
		Victor~G. Kac.
		\newblock \emph{Infinite-Dimensional Lie Algebras}.
		\newblock Third edition. Cambridge University Press, Cambridge, 1990.
		
		\bibitem[KK79]{KK79}
		Victor~G. Kac and David~A. Kazhdan.
		\newblock Structure of representations with highest weight of
		infinite-dimensional Lie algebras.
		\newblock \emph{Advances in Mathematics}, 34(1):97--108, 1979.
		
		\bibitem[KW89]{KW89}
		Victor~G. Kac and Minoru Wakimoto.
		\newblock Classification of modular invariant representations of affine
		algebras.
		\newblock In \emph{Infinite-dimensional Lie Algebras and Groups
			(Luminy-Marseille, 1988)}, volume~7 of
		\emph{Advanced Series in Mathematical Physics}, pages 138--177.
		World Scientific Publishing, Teaneck, NJ, 1989.
		
		\bibitem[KT00]{KT00}
		Masaki Kashiwara and Toshiyuki Tanisaki.
		\newblock Characters of irreducible modules with non-critical highest weights
		over affine Lie algebras.
		\newblock In \emph{Representations and Quantizations (Shanghai, 1998)},
		pages 275--296.
		China Higher Education Press, Beijing, 2000.
		
		\bibitem[KL79]{KL79}
		David Kazhdan and George Lusztig.
		\newblock Representations of Coxeter groups and Hecke algebras.
		\newblock \emph{Inventiones Mathematicae}, 53(2):165--184, 1979.
		
		\bibitem[Kum87]{Ku87}
		Shrawan Kumar.
		\newblock Extension of the category $\mathcal{O}^{\mathfrak{g}}$ and a
		vanishing theorem for the Ext functor for Kac--Moody algebras.
		\newblock \emph{Journal of Algebra}, 108(2):472--491, 1987.
		
		\bibitem[LL04]{LL04}
		James Lepowsky and Haisheng Li.
		\newblock \emph{Introduction to Vertex Operator Algebras and Their
			Representations}.
		\newblock Progress in Mathematics, volume~227.
		Birkh{\"a}user Boston, Boston, MA, 2004.
		
		\bibitem[MFF86]{MFF86}
		Fyodor~G. Malikov, Boris~L. Feigin, and Dmitry~B. Fuchs.
		\newblock Singular vectors in Verma modules over Kac--Moody algebras.
		\newblock \emph{Functional Analysis and Its Applications},
		20(2):103--113, 1986.
		
		\bibitem[Pe08]{Pe08}
		Ozren Per{\v{s}}e.
		\newblock Vertex operator algebras associated to certain admissible modules
		for affine Lie algebras of type $A$.
		\newblock \emph{Glasnik Matemati\v{c}ki, Serija III},
		43(63), no.~1, 41--57, 2008.
		
		\bibitem[Sha72]{shapovalov1972bilinear}
		N.~N. Shapovalov.
		\newblock On a bilinear form on the universal enveloping algebra of a complex
		semisimple Lie algebra.
		\newblock \emph{Functional Analysis and Its Applications},
		6(4):307--312, 1972.
		
		\bibitem[TY05]{tauvel2005lie}
		Patrice Tauvel and Rupert~W.~T. Yu.
		\newblock \emph{Lie Algebras and Algebraic Groups}.
		\newblock Springer Monographs in Mathematics.
		Springer-Verlag, Berlin, 2005.
		
	\end{thebibliography}
\end{document}